\documentclass[a4paper,11pt]{amsart}
\usepackage{subfiles}

\usepackage[english,french]{babel}
\usepackage{graphicx}
\usepackage{amsmath} 
\usepackage[all]{xy}
\usepackage{amsthm}
\usepackage{amsfonts}
\usepackage{amssymb}
\usepackage{mathrsfs}
\usepackage{amsbsy}
\usepackage{mathrsfs}
\usepackage{stmaryrd}
\usepackage{amsmath}
\usepackage{mathtools}
\usepackage[utf8]{inputenc}
\usepackage{calligra}
\usepackage{mathtools}
\usepackage{aurical}
\usepackage[T1]{fontenc}

\makeatletter
\DeclareFontFamily{OMX}{MnSymbolE}{}
\DeclareSymbolFont{MnLargeSymbols}{OMX}{MnSymbolE}{m}{n}
\SetSymbolFont{MnLargeSymbols}{bold}{OMX}{MnSymbolE}{b}{n}
\DeclareFontShape{OMX}{MnSymbolE}{m}{n}{
    <-6>  MnSymbolE5
   <6-7>  MnSymbolE6
   <7-8>  MnSymbolE7
   <8-9>  MnSymbolE8
   <9-10> MnSymbolE9
  <10-12> MnSymbolE10
  <12->   MnSymbolE12
}{}
\let\llangle\@undefined
\let\rrangle\@undefined
\DeclareMathDelimiter{\llangle}{\mathopen}%
                     {MnLargeSymbols}{'164}{MnLargeSymbols}{'164}
\DeclareMathDelimiter{\rrangle}{\mathclose}%
                     {MnLargeSymbols}{'171}{MnLargeSymbols}{'171}
\makeatother

\usepackage[usenames,dvipsnames]{color}
\usepackage{color}

\RequirePackage{color}
\definecolor{bwmagenta}{rgb}{0.0,0.45,0.6}

\definecolor{bwblue}{rgb}{0.4,0.1,0.2}

\usepackage[usenames,dvipsnames]{xcolor}

\usepackage[bookmarks,colorlinks=true,pagebackref,final,breaklinks=true,pdfsubject={LaTeX}]{hyperref}
\hypersetup{
        bookmarksnumbered=true,
        linkcolor=bwblue,
        citecolor=bwmagenta,
}

\usepackage{calligra}
\makeatletter
\def\@splitop#1#2\@nil{$\mathscr{#1}\!\!$\calligra#2\,\,}
\newcommand*\DeclareCursiveOperator[2]{%
  \newcommand#1{\mathop{\mbox{\@splitop#2\@nil}}\nolimits}}
\makeatother
\DeclareCursiveOperator{\TAY}{Hess}
\DeclareCursiveOperator{\HOM}{Hom}
\DeclareCursiveOperator{\Det}{D{}et}

\newtheorem{theo}{Theorem}[section]
\newtheorem{cor}[theo]{Corollary}

\newtheorem{lemma}[theo]{Lemma}
\newtheorem{remark}[theo]{Remark}

\newtheorem{proposition}[theo]{Proposition}
\newtheorem{example}[theo]{Example}

\newtheorem{definition}[theo]{Definition}
\newtheorem{theorem}{Theorem}[section]

\newcommand{\quo}[1]{ \mathbf{Z}/p^{n}\mathbf{Z}  }

\usepackage[OT2,T1]{fontenc}
\DeclareSymbolFont{cyrletters}{OT2}{wncyr}{m}{n}
\DeclareMathSymbol{\sha}{\mathalpha}{cyrletters}{"58}

\newcommand{\ord}{\circ}

\newcommand{\Z}{\mathbf{Z}}
\newcommand{\Q}{\mathbf{Q}}

\newcommand{\C}{\mathbf{C}}

\newcommand{\et}{\rm{\acute et}}

\usepackage{lipsum}

\usepackage{suffix}

\newcommand\chapterauthor[1]{\authortoc{#1}\printchapterauthor{#1}}
\WithSuffix\newcommand\chapterauthor*[1]{\printchapterauthor{#1}}

\makeatletter
\newcommand{\printchapterauthor}[1]{%
  {\parindent0pt\vspace*{-25pt}%
  \linespread{1.1}\large\scshape#1%
  \par\nobreak\vspace*{35pt}}
  \@afterheading%
}
\newcommand{\authortoc}[1]{%
  \addtocontents{toc}{\vskip-10pt}%
  \addtocontents{toc}{%
    \protect\contentsline{chapter}%
    {\hskip1.3em\mdseries\scshape\protect\scriptsize#1}{}{}}
  \addtocontents{toc}{\vskip5pt}%
}
\makeatother

\usepackage{turnstile}

\newcommand{\invlim}{\mathop{\varprojlim}\limits}
\DeclareMathOperator{\fin}{f}

\DeclareMathOperator{\wt}{wt}
\newcommand{\kk}{{k_{_{\!\circ}}}}

\DeclareMathOperator{\cyc}{cyc}

\newcommand{\hkappa}{\boldsymbol{\kappa}}
\newcommand{\uepsilon}{\underline{\varepsilon}_{\cyc}}
\newcommand{\groupringr}{{\Lambda^{^{{\!\!\!\!\circ}}}_r}}
\newcommand{\groupring}{{\Lambda^{^{\!\!\!\!\circ}}}}

\newcommand{\hphi}{\boldsymbol{\phi}}
\DeclareMathOperator{\spec}{sp}
\DeclareMathOperator{\dR}{\mathrm{dR}}
\DeclareMathOperator{\cris}{cris}
\DeclareMathOperator{\nr}{nr}
\newcommand{\hU}{{\mathbb{U}}}
\newcommand{\hD}{{\mathbb{D}}}
\newcommand{\hV}{{\mathbb{V}}}

\newcommand{\half}{{{\ndtstile{2}{1}}}}
\newcommand{\mhalf}{{{\ndtstile{2}{-1}}}}

\newcommand{\llll}{{\ell_{_{\!\circ}}}}
\newcommand{\mm}{{m_{_{\!\circ}}}}

\newcommand{\Hr}{{{\mathbb H}^{111}(X_r)}}
\newcommand{\Hinfty}{{{\mathbb H}^{111}(X_\infty^*)}}
\newcommand{\Hrord}{{{\mathbb H}^{111}_{\ord}(X_r)}}

\newcommand{\hlambda}{\boldsymbol{\lambda}}

\newcommand{\cH}{\mathcal H}
\newcommand{\co}{o}

\newcommand{\cE}{\mathcal E}

\newcommand{\cW}{\mathcal W}
\newcommand{\cA}{\mathcal A}

\newcommand{\Gal}{\mathrm{Gal\,}}
\newcommand{\GL}{\mathrm{GL}}

\newcommand{\Fil}{\mathrm{Fil}}

\newcommand{\Aut}{\mathrm{Aut}}
\newcommand{\Fr}{\mathrm{Fr}}

\newcommand{\ordi}{{\mathrm{ord}}}
\newfont{\gotip}{eufb10 at 12pt}

\newcommand{\cO}{{\mathcal O}}
\newcommand{\cM}{{\mathcal M}}

\newcommand{\cL}{{\mathcal L}}

\newcommand{\ra}{\rightarrow}
\newcommand{\lra}{\longrightarrow}

\newcommand{\SL}{{\mathrm {SL}}}

\newcommand{\CH}{{\mathrm{CH}}}
\newcommand{\AJ}{{\mathrm{AJ}}}

\newcommand{\testphi}{\breve{\phi}}
\newcommand{\testhphi}{\breve{\hphi}}
\newcommand{\testf}{\breve{f}}
\newcommand{\testg}{\breve{g}} 
\newcommand{\testh}{\breve{h}}

\newcommand{\htf}{\breve{\mathbf{f}}}
\newcommand{\htg}{\breve{\mathbf{g}}}
\newcommand{\hth}{\breve{\mathbf{h}}}
\newcommand{\Lp}{{\mathscr{L}_p}}

\newcommand{\hf}{{\mathbf{f}}}
\newcommand{\hg}{{\mathbf{g}}}
\newcommand{\hh}{{\mathbf h}}

\newcommand{\heta}{{\mathbf \eta}}
\newcommand{\cl}{{\mathrm{cl}}}

\newcommand{\XX}{\mathbb{X}}
\newcommand{\YY}{\mathbb{Y}}

\numberwithin{equation}{section}

\begin{document}

\title[$p$-adic families of diagonal cycles]{$p$-adic families of diagonal cycles}
\author{Henri Darmon and Victor Rotger}

\begin{abstract}
This note provides the construction of a three-variable family of cohomology classes arising from diagonal cycles on a triple product of towers of modular curves, and proves a reciprocity law relating it to 
the three variable triple-product $p$-adic $L$-function associated to a triple of Hida families by means of Perrin-Riou's $\Lambda$-adic regulator.

\medskip\medskip
\noindent
{\bf {\em R\'esum\'e}}. ---
On construit une famille \`a trois variables 
de classes de cohomologie associ\'ee  
\`a des  cycles diagonaux sur le   produit triple 
de tours de courbes modulaires, et on d\'emontre
 une loi de r\'eciprocit\'e  qui r\'ealise
la fonction $L$ $p$-adique  d'un triplet de familles de Hida comme l'image de cette famille de classes de cohomologie par le  r\'egulateur $\Lambda$-adique de Perrin-Riou.
\end{abstract}

\address{H. D.: Montreal, Canada}
\email{darmon@math.mcgill.ca}
\address{V. R.: IMTech, UPC and Centre de Recerca Matem\`{a}tiques, C. Jordi Girona 1-3, 08034 Barcelona, Spain}
\email{victor.rotger@upc.edu}


\maketitle

\begin{center}
{\em To Bernadette Perrin-Riou on her 65-th birthday}
\end{center}

\tableofcontents


\medskip
 \noindent
{\bf Acknowledgements.} The first author was supported by an NSERC Discovery grant, and the
  second author was supported by Grant MTM2015-63829-P. The second author also acknowledges the financial support by ICREA under the ICREA Academia programme.  This project has received funding from the European Research Council (ERC) under the European
Union's Horizon 2020 research and innovation programme (grant agreement No 682152).

\section*{Introduction}

The main purpose of this article is to supply a construction of a three-variable family of cycles  interpolating the generalized diagonal cycles introduced in \cite{DR1}, 
 and to prove a reciprocity law relating this family to 
the three variable triple-product $p$-adic $L$-function associated to a triple of Hida families by means of Perrin-Riou's $\Lambda$-adic regulator. 

In order to give a flavor of our construction, let us describe in more detail the organization and contents of this article.
  
After reviewing some background in the first section, in section \ref{sec:GKC} we construct for every $r\geq 1$ a completely explicit family of cycles in the cube $X_r^3$ of the modular curve $X_r=X_1(Mp^r)$ of $\Gamma_1(Mp^r)$-level structure. This family is parametrized by the space of $\SL_2(\Z/p^r\Z)$-orbits of  the set 
$$\Sigma_r := ((\Z/p^r\Z \times \Z/p^r\Z)')^3  \subset  ((\Z/p^r\Z)^2)^3 $$ 
of  triples of  primitive row vectors of length $2$
 with entries in $\Z/p^r\Z$, on which $\GL_2(\Z/p^r\Z)$ acts diagonally by right multiplication. Any triple in $\Sigma_r$ gives rise to a twisted diagonal embedding of the modular curve $\XX(p^r)$ of $\Gamma_1(M)\cup \Gamma(p^r)$-level structure into the three-fold $X_r^3$ and the associated cycle is defined as the image of this map: we refer to \eqref{Dabc} for the precise recipe.

The parameter space $\Sigma_r/\SL_2(\Z/p^r\Z)$ is closely related to $((\Z/p^r\Z)^\times)^3$ and as shown throughout \S \ref{sec:GKC}, the associated family of global cohomology classes introduced in Definition \ref{def-kappa-abc} can be packaged into a global $\Lambda$-adic cohomology class parametrized by three copies of weight space.

Along \S \ref{sec:higherweight} and \S \ref{sec:cris}  we study the higher weight and cristalline specialisations of this family and we eventually prove in Theorem \ref{prop:main-crystalline} that they interpolate the classes introduced in  \cite{DR1} as claimed above.

Finally, in \S \ref{3.1} we recall Garrett-Hida's triple product $p$-adic $L$-function associated to a triple of Hida families $(\hf,\hg,\hh)$ and prove in Theorem \ref{L=L} a reciprocity law expressing the latter as the image of our three-variable cohomology classes (as specified in Definition \ref{def-family}) under Perrin-Riou's $\Lambda$-adic regulator.

 It is instructive to compare the construction of our family  to the  
  approach taken in 
 \cite{DR2}, which associated to a triple $(f,\hg,\hh)$ consisting of a {\em fixed} newform $f$ and
 a pair $(\hg,\hh)$ of Hida families a {\em one-variable} family of cohomology classes instead of the two-variable family that one might
 have felt entitled to {\em a priori}.  
 This shortcoming of the earlier approach can be understood  by noting  that the space of
 embeddings of $\XX(p^r)$ into $X_1(M)\times X_r \times X_r$ as above in which the projection to the first factor is fixed
  is naturally parametrized by the coset
 space $M_2(\Z/p^r\Z)'/\SL_2(\Z/p^r\Z)$, where $M_2(\Z/p^r\Z)'$  denotes the set of $2\times 2 $ matrices whose rows are not divisible by $p$.
 The resulting cycles are therefore   parametrized by the coset space
 $\GL_2(\Z/p^r\Z)/\SL_2(\Z/p^r\Z) = (\Z/p^r\Z)^\times$,
 whose inverse limit with $r$ is the one dimensional $p$-adic space $\Z_p^\times$ rather than a two-dimensional one. 
 
As mentioned already in our previous article in this volume, these  
cycles are of  interest in their own right, and shed a useful complementary perspective on the construction of the $\Lambda$-adic cohomology classes for the triple product when compared to \cite{BSV1}.
Indeed, their study forms the basis for the ongoing PhD thesis of David Lilienfeldt
\cite{Li21}, and has let to  interesting open questions as those that are explored by Castella and Hsieh in \cite{CaHs-Rank2}.

\section{Background}
\subsection{Basic notations}
Fix an algebraic closure $\bar\Q$ of $\Q$.
All   the number fields 
that   arise will be  viewed as embedded in this algebraic closure.
For each such $K$, let
 $G_{K}:= \Gal(\bar\Q/K)$ denote its absolute
Galois group.
Fix an odd prime $p$ and
 an embedding $\bar\Q \hookrightarrow \bar\Q_p$; let $\ordi_p$ denote the 
  resulting   $p$-adic valuation on $\bar\Q^\times$, normalized in 
such a way that $\ordi_p(p)=1$.

For a variety $V$ defined over $K\subset\bar\Q$,
let $\bar V$ denote the base change of $V$ to $\bar\Q$. 
If $\mathcal F$ is an \'etale sheaf on $V$, 
write $H^i_{\et}(\bar V,\mathcal F)$  for the $i$th 
\'etale cohomology group  of $\bar V$ with values in $\mathcal F$, equipped 
with its continuous action by $G_K$. 

Given a prime $p$, let
 $\Q(\mu_{p^\infty}) = \cup_{r\geq 1} \Q(\zeta_{r})$ be the cyclotomic extension of $\Q$ obtained by adjoining
to $\Q$ a  primitive $p^r$-th root of unity $\zeta_r$.
 Let $$\varepsilon_{\cyc}: G_\Q \lra \Gal(\Q(\mu_{p^\infty})/\Q) \stackrel{\simeq}{\lra} \Z_p^\times$$ 
 denote the $p$-adic cyclotomic character. 
 It can be factored as $\varepsilon_{\cyc} = \omega \langle \varepsilon_{\cyc}\rangle$,
 where 
  $$\omega: G_\Q \lra \mu_{p-1} \qquad  \quad \langle \varepsilon_{\cyc} \rangle: G_\Q \lra1+p\Z_p$$ are 
   obtained by composing   $\varepsilon_{\cyc}$ with the   
 projection onto the first and second factors in the
 canonical decomposition $\Z_p^\times \simeq \mu_{p-1} \times (1+p\Z_p)$. 
 If $\cM$ is a $\Z_p[G_\Q]$-module and $j$ is an integer, write 
  $\cM(j)=\cM\otimes \varepsilon_{\cyc}^j$ for  the $j$-th Tate twist of $\cM$. 
  
  Let 
  $$  \groupringr := \Z_p[(\Z/p^r\Z)^\times], \qquad \groupring := \Z_p[[\Z_p^\times]] := \invlim_r \groupringr $$
  denote the group ring and completed group ring attached to the profinite group $\Z_p^\times$.
  The ring $\groupring$  is equipped with $p-1$ distinct algebra homomorphisms $\omega^i:\groupring \rightarrow \Lambda$
  (for $0\le i \le p-2$) 
  to the local ring 
   $$  \Lambda=\Z_p[[1+p\Z_p]] = \invlim \Z_p[1+p \Z/p^r\Z] \simeq \Z_p[[T]],$$ 
  where  $\omega^i$    sends a group-like element $a\in \Z_p^\times$ to 
  $\omega^i(a) \langle a\rangle \in \Lambda$.
   These homomorphisms identify $\groupring$ with the direct sum  
  $$ \groupring = \bigoplus_{i=0}^{p-2}  \Lambda. $$
 The local ring $\Lambda$ is called the one variable 
  Iwasawa algebra. 
More generally, for any integer $t\geq 1$, let 
 $$ \groupring^{\otimes t} :=  \groupring \hat\otimes_{\Z_p}  \stackrel{t}{\ldots} \hat\otimes_{\Z_p}\groupring, \qquad 
  \Lambda^{\otimes t} = \Lambda \hat\otimes_{\Z_p}  \stackrel{t}{\ldots} \hat\otimes_{\Z_p}\Lambda\simeq \Z_p[[T_1,\ldots T_t]].$$ 
The latter ring is called the  Iwasawa algebra in $t$ variables, and  is isomorphic to the power series ring in $t$ variables over $\Z_p$,
while 
$$ \groupring^{\otimes t} = \bigoplus_{\alpha} \Lambda^{\otimes t},$$
the sum running over the $(p-1)^t$  distinct $\Z_p^\times$ valued characters of $(\Z/p\Z)^{\times t}$.

\subsection{Modular forms and Galois representations}
\label{subsec:forms}

Let 
 $$\phi = q+\sum_{n\geq 2} a_n(\phi)q^n \in S_{k}(M,\chi)$$
be a cuspidal modular form of weight $k \geq 1$, level $M$ and character  $\chi:(\Z/M\Z)^\times \ra \C^\times$, and assume that $\phi$ is an eigenform with respect to all good Hecke operators $T_\ell$, $\ell \nmid M$. 

Fix an odd prime number $p$ (that in this section may or may not divide $M$).
Let $\cO_\phi$ denote the valuation ring of the finite extension 
of $\Q_p$  generated by the fourier coefficients of $\phi$, and let
  $\mathbb T$ denote the Hecke algebra generated over $\Z_p$ by the good Hecke operators $T_\ell$ with 
 $\ell \nmid M$ and by the diamond operators acting on $S_k(M,\chi)$. 
 The eigenform $\phi$ gives rise to an algebra homomorphism
 $$ \xi_\phi: \mathbb T \lra \cO_\phi$$
 sending $T_\ell$ to $a_\ell(\phi)$ and  the diamond operator $\langle \ell\rangle$ to 
 $\chi(\ell)$. 

A fundamental construction of Shimura, Deligne, and Serre-Deligne attaches to 
$\phi$  an irreducible   Galois representation 
$$
\varrho_\phi: G_\Q \lra  \Aut(V_\phi) \simeq \GL_2(\cO_\phi)
$$
of rank $2$, unramified at all primes  $\ell\nmid M p$, and for which 
\begin{equation}
 \label{eqn:deligne-rep}
 \det(1-\varrho_\phi(\Fr_\ell) x) = 
1-a_\ell(\phi)x+\chi(\ell)\ell^{k-1} x^2,
\end{equation}
where $\Fr_\ell$ denotes 
the arithmetic Frobenius element at $\ell$.
     This property characterizes the semi-simplification of   $\varrho_\phi$ up to isomorphism.

When $k:= \kk+2 \ge 2$, the representation $V_\phi$ can be realised in the $p$-adic  \'etale cohomology of an 
appropriate Kuga-Sato variety. Since this realisation is important for the construction of generalised Kato
classes, we now briefly recall its salient features.
 Let $Y=Y_1(M)$ and
$X=X_1(M)$ denote the open and closed modular curve representing the fine 
moduli functor of isomorphism classes of pairs $(A,P)$ formed by
 a (generalised) elliptic curve $A$  together with a torsion 
point $P$ on $A$ of exact order $M$. Let
\begin{equation}\label{piA}
\pi: \mathcal A_{\circ} \,\lra \,Y
\end{equation}
denote the universal   elliptic curve over $Y$.

The  $\kk$-th open Kuga-Sato variety  
over $Y$  
is the    $\kk$-fold   fiber product 
\begin{equation}\label{KugaSato}
\cA^\kk_{\circ}  := \cA_{\circ} \times_Y \stackrel{(\kk)}{\dots} \times_Y \cA_{\circ}
\end{equation}
 of $\cA_{\circ}$ over $Y$. 
 The variety $\cA^\kk_\circ$  admits a smooth compactification 
$\cA^\kk$ which is fibered over $X$ and  is called
the $\kk$-th Kuga-Sato variety over $X$; we refer to Conrad's appendix in \cite{BDP} for more details.
The geometric points  in $\cA^\kk$ that lie above $Y$ are in bijection with
isomorphism classes of 
 tuples $[(A,P),P_1,\dots,P_\kk]$, where $(A,P)$ is  associated 
 to a point of $Y$ as in the previous paragraph and $P_1,...,P_\kk$ are  points on $A$.

The representation $V_\phi$ is realised (up to a suitable Tate twist) 
in  the  middle degree \'etale cohomology 
$H^{\kk+1}_{\et}(\bar\cA^\kk, \Z_p)$.
More precisely,  let
$$
\mathcal H_r := R^1\pi_*\,\Z/p^r\Z (1), \qquad \quad
  \mathcal H := R^1\pi_*\,\Z_p(1),
$$
and for 
any  $\kk\geq 0$, define 
\begin{equation}
\label{eqn:sheaves-of-symmetric-tensors}
\cH_r^\kk := \mathrm{TSym}^\kk(\mathcal H_r), \qquad 
\cH^\kk := \mathrm{TSym}^\kk(\mathcal H)
\end{equation}
 to be the sheaves of symmetric $\kk$-tensors of $\mathcal H_r$ and 
 $\mathcal H$, respectively.
As defined in e.g.\,\cite[(2.1.2)]{BDP}, there
 is an idempotent $\epsilon_{\kk}$ in the ring of rational 
 correspondences of $\cA^{\kk}$ whose induced
  projector on the \'etale cohomology groups of this variety satisfy:
  \begin{equation}
\label{idempotent}
\epsilon_{\kk} \big( H^{\kk+1}_{\et}(\bar\cA^{\kk},\Z_p(\kk))\big)  =  H^1_{\et}(\bar X,\cH^{\kk}).
\end{equation}

Define the $\cO_\phi$-module 
 \begin{equation}
\label{Vphi-class}
V_\phi(M) := 
H^1_{\et}(\bar X,\cH^{\kk}(1))\otimes_{\mathbb T, \xi_\phi} \cO_\phi,
\end{equation} 
and write 
\begin{equation}
\label{vpi}
\varpi_\phi: H^1_{\et}(\bar X,\cH^{\kk}(1)) \lra  V_\phi(M)
\end{equation}
for the canonical projection of ${\mathbb  T}[G_\Q]$-modules arising from \eqref{Vphi-class}. 
Deligne's results and the theory of newforms show that the 
 module $V_\phi(M)$ is the direct sum of several copies of a locally free module $V_\phi$ of rank $2$ over $\cO_\phi$ that satisfies 
 \eqref{eqn:deligne-rep}.

Let $\alpha_\phi$ and $\beta_\phi$ the two roots of the $p$-th Hecke 
polynomial $T^2-a_p(\phi)T+\chi(p)p^{k-1}$, 
ordered in such a way that $\ordi_p(\alpha_\phi) \leq \ordi_p(\beta_\phi)$. (If $\alpha_\phi$ and $\beta_\phi$ have the same $p$-adic valuation, simply fix an arbitrary ordering of  the two roots.)
We set $\chi(p)=0$ whenever $p$ divides the primitive level of $\phi$ and thus $\alpha_\phi = a_p(\phi)$ and $\beta_\phi = 0$ in this case. 
The eigenform $\phi$ is said to be {\em ordinary}
 at $p$ 
 when   $\ordi_p(\alpha_\phi) = 0$.
  In that case,   there is an exact sequence of $G_{\Q_p}$-modules
\begin{equation}\label{dec}
0 \ra V_\phi^+ \lra V_\phi \lra V_\phi^- \ra 0, 
\qquad V_\phi^+ \simeq \cO_\phi(\varepsilon_{\cyc}^{k-1} \chi  \psi^{-1}_\phi), \quad  V_\phi^-\simeq \cO_\phi(\psi_\phi),
\end{equation}
where $\psi_\phi$ is the unramified character of $G_{\Q_p}$ sending $\Fr_p$ to $\alpha_\phi$.

\subsection{Hida families and $\Lambda$-adic Galois representations}
\label{subsec:hida}

Fix a prime $p\geq 3$. The formal spectrum $$\mathcal W := \mathrm{Spf}(\Lambda)$$ of
 the Iwasawa algebra $\Lambda=\Z_p[[1+p\Z_p]]$ is  called the {\em weight space} 
attached to $\Lambda$. 
The  
$A$-valued points  of $\mathcal W$ over a $p$-adic ring $A$ are given by
\begin{equation*}
\mathcal W(A) = \mathrm{Hom}_{\rm alg}(\Lambda, A) =  \mathrm{Hom}_{\rm grp}(1+p\Z_p,A^\times),
\end{equation*}
 where the Hom's in this definition denote continuous homomorphisms of $p$-adic rings
   and profinite groups respectively. 
 Weight space is equipped with a distinguished collection of   {\em arithmetic points}
   $\nu_{\kk,\varepsilon}$ , indexed by 
 integers $\kk \geq 0$ and    Dirichlet characters  $\varepsilon: (1+ p\Z/p^r\Z)  \rightarrow \Q_p(\zeta_{r-1})^\times$
 of $p$-power conductor. 
 The  point 
$\nu_{\kk,\varepsilon} \in \mathcal W (\Z_p[\zeta_r])$  is  defined by 
$$\nu_{\kk,\varepsilon}(n) = \varepsilon(n) n^{\kk}, $$
and the notational shorthand $\nu_\kk:=\nu_{\kk,1}$ is adopted throughout.
More generally, if $\tilde\Lambda$ is any finite flat $\Lambda$-algebra, a point $x \in
\tilde \cW:= \mathrm{Spf}(\tilde\Lambda)$ is said to be arithmetic if its restriction to $\Lambda$ 
agrees with $\nu_{\kk,\varepsilon}$ for some $\kk$ and $\epsilon$. The integer $k=\kk +2$ is called
the {\em weight } of $x$ and denoted $\wt(x)$.

Let 
\begin{equation}
\label{ecyc}
\underline{\varepsilon}_{\cyc}: G_\Q   \lra  \Lambda^\times
\end{equation}
denote the $\Lambda$-adic cyclotomic character which sends a Galois element $\sigma$ to the group-like element $[\langle\varepsilon_{\cyc}(\sigma)\rangle]$. This character interpolates 
the powers of the cyclotomic character, in the sense that 
\begin{equation}\label{ecycint}
\nu_{\kk,\varepsilon}\circ \underline{\varepsilon}_{\cyc} = 
\varepsilon  \cdot \langle \varepsilon_{\cyc} \rangle^{\kk} = \varepsilon \cdot \varepsilon_{\cyc}^{\kk}  \cdot \omega^{-\kk}.
\end{equation}

Let $M\geq 1$ be an integer not divisible by $p$.

\begin{definition}\label{def-Hida-fam} A {\em Hida family of tame level $M$ and tame character $\chi: (\Z/M\Z)^\times \,\ra \,\bar\Q_p^\times$}  is 
a formal $q$-expansion
$$
\hphi = \sum_{n\geq 1} a_n(\hphi) q^n \in \Lambda_{\hphi}[[q]]
$$
with coefficients in 
a finite flat $\Lambda$-algebra $\Lambda_{\hphi}$, such that for any
arithmetic point $x\in  \cW_{\hphi}:= \mathrm{Spf}(\Lambda_{\hphi})$
above $\nu_{\kk,\epsilon}$, where $\kk\ge 0$ and  
$\varepsilon$ is a character of conductor $p^r$,
the  series
$$
\hphi_x := \sum_{n\geq 1} x(a_n(\hphi)) q^n \in \bar\Q_p[[q]]  
$$
is the $q$-expansion of a classical $p$-ordinary eigenform  in the space
 $S_{k}(Mp^r,\chi \varepsilon \omega^{-\kk})$ of cusp forms
of weight $k = \kk+2$, level $Mp^r$ and nebentype $\chi \varepsilon \omega^{-\kk}$. 
\end{definition}

By enlarging $\Lambda_{\hphi}$ if necessary, we shall assume throughout that $\Lambda_{\hphi}$ contains the $M$-th roots of unity.

\begin{definition}\label{tame-cris} Let $x\in \cW_{\hphi}$ 
be an arithmetic point lying above the point  $\nu_{\kk,\epsilon}$ of weight space.
The point $x$ is said to be
\begin{itemize}
\item {\em  tame} if the  character $\epsilon$ is  tamely ramified, i.e., factors through $(\Z/p\Z)^\times$.
\item {\em crystalline} if  $\epsilon\omega^{-\kk}  =1$, i.e., if the weight $k$ specialisation of $\hphi$ at $x$ has trivial nebentypus character at $p$.
\end{itemize}
We let $\cW_{\hphi}^\circ$ denote the set of crystalline arithmetic points of $\cW_{\hphi}$. 
\end{definition}
Note that a crystalline point is necessarily tame but of course there are tame points that are not crystalline. The justification for this terminology is that  the Galois representation
$ V_{\phi_x}$ is crystalline at $p$ when $x$ is crystalline.

If $x$ is a crystalline point, then the classical form $\hphi_x$ is always old at $p$ if $k>2$. In that case there exists an eigenform $\phi_x^\circ$ of level $M$ such that $\hphi_x$ is the ordinary $p$-stabilization of $\hphi^\circ_x$. If the weight is $k=1$ or $2$, $\hphi_x$ may be either old or new at $p$; if it is new at $p$ then we set $\hphi_x^\circ=\hphi_x$ in order to have uniform notations.

 We say $\hphi$ is residually irreducible if the mod $p$ Galois representation associated to the Deligne representations associated to $\phi_x^{\circ}$ for any crystalline classical point is irreducible.

Finally, the Hida family $\hphi$ is said to be {\em primitive} of tame level $M_{\hphi}\mid M$
 if for all but finitely many
  arithmetic points   $x\in \cW_{\hphi}$ of weight $k \ge 2$, the modular form $\hphi_x$
   arises from a newform of level $M_{\hphi}$. 

 The following theorem of Hida and Wiles
associates a two-dimensional Galois representation 
 to a Hida family $\hphi$ (cf.\,e.g.\,\cite[Th\'eor\`{e}me 7]{MT}). 
 
 \begin{theorem}
 \label{Vphi-thm} 
 Assume $\hphi$ is residually irreducible. Then there is a rank two  $\Lambda_{\hphi}$-module $\mathbb{V}_{\hphi}$  equipped with a Galois 
action
\begin{equation}
\label{Vphi}
\varrho_{\hphi}: G_{\Q}   \lra  \Aut_{\Lambda_{\hphi}}(\mathbb{V}_{\hphi}) \simeq  \GL_2(\Lambda_{\hphi}),
\end{equation}
such that, for all arithmetic points $x:\Lambda_{\hphi} \lra \bar\Q_p$,
$$\mathbb{V}_{\hphi}\otimes_{x,\Lambda_{\hphi}} \bar \Q_p \simeq V_{\hphi_x} \otimes  \bar \Q_p.$$

Let
$$
\psi_{\hphi}: G_{\Q_p}\lra \Lambda_{\hphi}^\times 
$$
denote the unramified character sending a Frobenius element $\Fr_p$  to $\mathbf{a}_p(\hphi)$. The restriction of $\mathbb{V}_{\hphi}$ to $G_{\Q_p}$ admits a filtration
\begin{equation}\label{Wiles-ord}
0 \, \ra \, \mathbb{V}_{\hphi}^+ \, \ra \, {\mathbb{V}_{\hphi}}  \, \ra \, \mathbb{V}_{\hphi}^- \, \ra \, 0 \quad \mbox{ where } \,
\mathbb{V}_{\hphi}^+ \simeq \Lambda_{\hphi}(\psi_{\hphi}^{-1} \chi \varepsilon_{\cyc}^{-1}\uepsilon ) \mbox{\ \ and\ \ }
\mathbb{V}_{\hphi}^- \simeq \Lambda_{\hphi}(\psi_{\hphi}).
\end{equation}

\end{theorem}
The explicit construction of the Galois representation $\mathbb{V}_{\hphi}$ plays an important role in defining
the generalised Kato classes, 
and we now recall its main features.

For all   $0\le r <s$,  let 
$$ X_r := X_1(Mp^r), \qquad X_{r,s} := X_1(Mp^r) \times_{X_0(Mp^r)} X_0(Mp^s),$$
where the fiber product is taken relative to the natural projection maps. In particular,
   \begin{itemize}
   \item the curve $X:= X_0 := X_1(M)$ represents the functor of 
    elliptic curves $A$ with  $\Gamma_1(M)$-level
    structure, i.e., with a marked point of order $M$;
  \item the curve $X_r$ represents the functor classifying  pairs  
   $(A, P)$ consisting of a generalized elliptic curve $A$ with  $\Gamma_1(M)$-level
    structure
 and a point $P$  of order 
$p^r$  on $A$; 
   \item  the curve $X_{0,s} = X_1(M)\times_{X_0(M)} X_0(Mp^s)$  classifies
pairs $(A, C)$ consisting of a generalized elliptic curve $A$ with $\Gamma_1(M)$ structure
 and a cyclic  subgroup scheme 
$C$ of order 
$p^s$  on $A$;
\item the curve $X_{r,s}$  classifies
pairs $(A, P, C)$ consisting of a generalized elliptic curve $A$ with $\Gamma_1(M)$ structure, 
a point $P$ of order $r$ on $A$ and 
 and a cyclic  subgroup scheme 
$C$ of order 
$p^s$  on $A$ containing $P$.
\end{itemize}
 
The curves $X_r$ and  $X_{0,r}$  are smooth geometrically connected curves over $\Q$. The 
natural covering map $X_r\lra X_{0,r}$ 
  is   Galois with Galois group
 $(\Z/p^r\Z)^\times$   acting  on the left via the diamond operators defined by
\begin{equation}
\label{eqn:galois-auts-diamond}
  \langle a\rangle (A,P) = (A,aP).
 \end{equation}
 Let 
 \begin{equation}
 \label{eqn:def-varpi1}
 \varpi_{1}: X_{r+1}   \lra  X_r 
 \end{equation}
  denote the natural projection from level $r+1$ to level $r$ which corresponds to the map 
  $(A,P) \mapsto (A, pP)$, and to the map 
  $\tau\mapsto \tau$ on upper half planes. Let 
 $$\varpi_2: X_{r+1} \lra X_r$$ denote the other projection, corresponding to the map 
 $(A,P) \mapsto (A/ \langle p^r P\rangle, P + \langle p^r P\rangle)$, which 
 on the upper half plane sends $\tau$ to $p\tau$.
 These maps can be factored as 
\begin{equation}
\label{eqn:def-varpi-12}
 \xymatrix{  
X_{r+1} \ar[d]^-{\mu} \ar[dr]^-{\varpi_1}  &  \\
 X_{r,r+1}   \ar[r]_-{\pi_1}  & X_r,   } \qquad \qquad
 \xymatrix{  
X_{r+1} \ar[d]^-{\mu} \ar[dr]^-{\varpi_2}  &  \\
X_{r,r+1} \ar[r]_-{\pi_2}  & X_r.  }
\end{equation}
For all $r\ge 1$,  the vertical  map  $\mu$   is a   cyclic Galois covering
  of degree $p$, 
  while the  horizontal 
 maps $\pi_1$ and $\pi_2$ 
 are non-Galois coverings of degree $p$.
 When $r=0$,  the    map $\mu$ is a  cyclic Galois covering
  of degree $p-1$
 and
 $\pi_2$
are non-Galois coverings of degree $p+1$.

The $\Lambda$-adic representation $\mathbb{V}_{\hphi}$ shall be realised (up to twists)  in quotients of   the inverse 
limit of  \'etale cohomology groups arising from 
the tower
\begin{eqnarray*}
X_\infty^*  &:&    \cdots \stackrel{\varpi_1}{\lra} X_{r+1} \stackrel{\varpi_1} {\lra } X_r  
\stackrel{\varpi_1}{\lra} \cdots
  \stackrel{\varpi_1}{\lra}X_1  \stackrel{\varpi_1}{\lra}   X_0
  \end{eqnarray*}
\noindent
of modular curves.
Define the inverse limit 
\begin{equation}
\label{eqn:proj-lim-H1}
H^1_{\et}(\bar X_\infty^*, \Z_p) := \invlim_{\ \ \varpi_{1*}} H^1_{\et}(\bar X_r,\Z_p)
    \end{equation}
where the  transition maps  arise   from the  pushforward
induced by the morphism $\varpi_1$.
This  inverse limit is a module over the completed group rings $\Z_p[[\Z_p^\times]]$ 
arising from the action of the diamond operators, and is endowed with a plethora of extra structures that we now describe.

\medskip\noindent
{\em Hecke operators}.
The transition maps in \eqref{eqn:proj-lim-H1} are compatible with the action of the  Hecke operators $T_n$ 
for  all $n$ that are not divisible by $p$.
Of crucial importance for us in this article is
 Atkin's 
operator $U_p^*$, 
which operates  on $H^1_{\et}(\bar X_r,\Z_p)$    via the composition
$$   U_p^* := \pi_{1\ast} \pi_2^{\ast}$$
arising from the maps in \eqref{eqn:def-varpi-12}. 


The operator  $U_p^*$ 
is compatible  with the transition maps defining $H^1_{\et}(\bar X_\infty^*,\Z_p)$, 

\medskip\noindent
{\em Inverse systems of \'etale sheaves}.
The cohomology group $H^1_{\et}(\bar X_\infty^*,\Z_p)$ 
can be identified with the first cohomology group
of the base curve $X_1$ with values in a certain inverse systems of \'etale sheaves.

For each $r\ge 1$, let 
\begin{equation}
\label{def:Lr-sheaf}
 \cL_{r}^* :=  \varpi_{1 \ast}^{r-1} \Z_p
 \end{equation} 
be the pushforward  of the constant sheaf on $X_r$ 
via the map 
$$\varpi_1^{r-1}: X_r \lra X_1
$$

The stalk of $\cL_{r}^*$  at a geometric point $x=(A,P)$ on $X_1$ is
given by 
$$
\cL_{r,x}^* = \Z_p[A[p^r]\langle P \rangle],   
$$
where 
$$A[p^r]\langle P \rangle:= \{ Q \in A[p^r] \mbox{ such that } p^{r-1}Q = P \}. $$
The multiplication by $p$ map on the fibers gives rise to natural homomorphisms of sheaves
\begin{equation}
\label{eqn:transition-on-Lr}
  [p]: \cL_{r+1}^*\lra \cL_r^*, 
  \end{equation}
and 
Shapiro's lemma gives canonical identifications 
$$ H^1_{\et}(\bar X_r,\Z_p) = H^1_{\et}(\bar X_1,\cL_r^*),$$
for which the following diagram commutes:
$$
\xymatrix{
H^1_{\et}(\bar X_{r+1},\Z_p) \ar[r]^{\varpi_{1\ast}}  \ar@{=}[d] & H^1_{\et}(\bar X_r, \Z_p) \ar@{=}[d] \\
H^1_{\et}(\bar X_1, \cL_{r+1}^*) \ar[r]^{[p]} & H^1_{\et}(\bar X_1, \cL_r^*).}
$$

Let  $\cL_{\infty}^* := \invlim_r \cL_r^*$ denote
the inverse system of \'etale sheaves 
relative to the maps $[p]$ arising  in \eqref{eqn:transition-on-Lr}. By passing to the limit, we obtain an identification
\begin{equation}
\label{eqn:crucial-relation-for-induced}
 H^1_{\et}(\bar X_\infty^*,\Z_p) = 
\invlim_{r\ge 1} H^1_{\et}( \bar X_1, \cL_{r}^*) = H^1_{\et}(\bar X_1, \cL_\infty^*).
\end{equation}

\medskip\noindent
{\em Weight $k$ specialisation maps}.
Recall the $p$-adic \'etale sheaves  $\cH^\kk$  introduced in 
\eqref{eqn:sheaves-of-symmetric-tensors},
whose cohomology gave rise to the Deligne representations attached to modular forms of
weight $k = \kk+2$
  via  \eqref{Vphi-class}. The natural $\kk$-th power symmetrisation function
$$ A[p^r]  \lra \cH_r^\kk, \qquad Q \mapsto Q^\kk,$$  
restricted to $A[p^r]\langle P\rangle $
and extended 
 to $\cL_{r,x}^*$ 
by $\Z_p$-linearity, 
induces   morphisms 
\begin{equation}
\label{eqn:mom}
\spec_{k,r}^*:  {\cL}_r^*   \lra \cH^\kk_r
\end{equation}
of sheaves over $X_1$ (which are thus  compatible with the action of $G_\Q$ on the fibers).
These specialisation morphisms are compatible with the  transition maps $[p]$ in the sense that the diagram
$$
\xymatrix{ 
\cL_{r+1}^* \ar[r]^{[p]}  \ar[d]^-{\spec_{k,r+1}^*} & \cL_r^*  \ar[d]^-{\spec_{k,r}^*}\\
\cH_{r+1}^\kk  \ar[r] & \cH_r^\kk}
$$ 
commutes, where the bottom horizontal arrow denotes the natural reduction map.
The maps  $\spec_{k,r}^*$ 
can thus be pieced together into morphisms 
\begin{equation}
\label{eqn:mom-bis} 
\spec_k^*: \cL_\infty^* \lra \cH^\kk.
\end{equation}
The induced morphism 
\begin{equation}
\label{eqn:wtk-specialisation}
  \spec_k^*: H^1_{\et}(\bar X_\infty^*,\Z_p)   \lra H^1_{\et}(\bar X_1, \cH^\kk),
  \end{equation}
arising from those   on $H^1_{\et}(\bar X_1,\cL_\infty^*)$ 
via
 \eqref{eqn:crucial-relation-for-induced}
  will be denoted by the same symbol by abuse of notation, and is referred to as the {\em weight $k = \kk+2$ specialisation map}.
  The existence of such maps having 
 finite cokernel  reveals that  the $\Lambda$-adic Galois representation 
 $H^1_{\et}(\bar X_\infty^*,\Z_p)$ 
is rich enough to 
  capture
the Deligne representations attached to modular forms on $X_1$ of {\em arbitrary } weight $k\ge 2$.

For each $a\in 1+p\Z_p$, the  diamond operator $\langle a\rangle$ 
 acts trivially on $X_1$ and as multiplication by $a^\kk$ on
the stalks of the sheaves $\cH^\kk_r$. It follows that the weight $k$
specialisation 
 map
 $ \spec_k^*$ 
factors  through
 the quotient $H^1_{\et}(\bar X_\infty^*,\Z_p)\otimes_{\Lambda,\nu_\kk} \Z_p$,
i.e., one obtains 
a map
\begin{eqnarray*}
 \spec_k^*: H^1_{\et}(\bar X_\infty^*,\Z_p) \otimes_{\Lambda,\nu_\kk} \Z_p 
& \lra  & H^1_{\et}(\bar X_1, \cH^\kk).
\end{eqnarray*}

\begin{remark}
The inverse limit $\cL_\infty^*$  of the sheaves 
  $\underline{\cL}_r^*$ on $X_1$ has been systematically studied 
 by G. Kings in \cite[\S 2.3-2.4]{Ki}, and is referred to as a  
 {\em sheaf of Iwasawa modules}. 
 Jannsen introduced in \cite{Ja} the \'etale cohomology groups of  such 
inverse systems of sheaves, and proved the existence of a Hoschild-Serre spectral 
sequence, Gysin excision exact sequences and cycle map in this context. 
\end{remark}

\medskip\medskip\noindent
{\em Ordinary projections}.
Let 
\begin{equation}\label{ord-proj}
e^\ast := \lim_{n\rightarrow \infty} U_p^{*n!}
\end{equation}
 denote 
  Hida's 
(anti-)ordinary projector. 
 Since $U_p^\ast$ commutes with the push-forward maps $\varpi_{1 \ast}$, this idempotent 
operates on  $H^1_{\et}(\bar X_\infty^*,\Z_p)$.
While the structure   of the $\Lambda$-module
   $H^1_{\et}(\bar X_\infty^*,\Z_p)$ 
seems     rather  complicated,  a dramatic simplification 
occurs after passing to the quotient
$ e^\ast H^1_{\et}(\bar X_\infty^*,\Z_p)$, as the following classical theorem of Hida shows.

\begin{theorem} \cite[Corollaries 3.3 and 3.7]{Hida2}
\label{thm:hida-rep}
The Galois representation $ e^\ast H^1_{\et}(\bar X_\infty^*,\Z_p(1))$
is a free $\Lambda$-module.
For each $\nu_{\kk}  \in \cW$ with $\kk \ge 0$,  the  weight $k = \kk+2$ specialisation map
induces   
maps with bounded cokernel (independent of $k$) 
\begin{eqnarray*}
\label{eqn:weight-k-sheaf-spec} 
 \spec_k^* &: & e^\ast H^1_{\et}(\bar X_\infty^*,\Z_p(1))  \otimes_{\nu_{\kk}}  \Z_p \lra  e^\ast H^1_{\et}(\bar X_1, \cH^{\kk}(1)).
  \end{eqnarray*}

\end{theorem}

\medskip\noindent
{\em Galois representations attached to Hida families}. 
 The Galois representation $\mathbb{V}_{\hphi}$ of Theorem \ref{Vphi-thm} associated by Hida and Wiles to a Hida family $\hphi$ of  tame level $M$ and character $\chi$
 can be realised as  a quotient of 
the $\Lambda$-module $e^\ast H^1_{\et}(\bar X_\infty^*,\Z_p(1))$. 
   More precisely, let
    $$ \xi_{\hphi}: {\mathbb T}_\Lambda \lra \Lambda_{\hphi}$$
    be the $\Lambda$-algebra homomorphism from the $\Lambda$-adic 
    Hecke algebra ${\mathbb T}_\Lambda$ to the $\Lambda$-algebra $\Lambda_{\hphi}$ generated by the fourier coefficients of $\hphi$ sending $T_\ell$ to $a_\ell(\hphi)$.
   
Then we have,  much as in \eqref{vpi}, a quotient map of $\Lambda$-adic Galois representations 
\begin{equation}
\label{Vphi-hida}
\varpi^\ast_{\hphi}:  e^\ast H^1_{\et}(\bar X_\infty^*,\Z_p(1)) \lra  
e^\ast H^1_{\et}(\bar X_\infty^*,\Z_p(1))\otimes_{\mathbb T_\Lambda, \xi_{\hphi}} \Lambda_\phi =: \mathbb{V}_{\hphi}(M), \end{equation} 
for which 
the following 
     diagram of ${\mathbb T}_\Lambda[G_\Q]$-modules  is commutative:
\begin{equation}\label{fit}
\xymatrix{
e^\ast H^1_{\et}(\bar X_\infty^*,\Z_p(1))   \ar[rr]^{\qquad\varpi^\ast_{\hphi}} \ar[d]^{\spec_{k}^*} & &
\mathbb{V}_{\hphi}(M) \ar[d]^{x} \\
e^\ast H^1_{\et}(\bar X_1,\cH^{\kk}(1)) \ar[rr]^{\qquad  \varpi_{\hphi_x}} &  & V_{\hphi_x}(Mp), }
\end{equation}
for all  arithmetic points  $x$ of $\cW_{\hphi}$ of weight $k = \kk+2$ and trivial character.

As in  \eqref{vpi}, $\mathbb{V}_{\hphi}(M)$ is non-canonically isomorphic to a finite direct sum of 
 copies of a $\Lambda_{\hphi}[G_\Q]$-module $\mathbb{V}_{\hphi}$ of rank $2$ over $\Lambda_{\hphi}$, satisfying the properties stated in Theorem \ref{Vphi-thm}. 
 

 
One can of course work alternatively with the ordinary projection $e  := \lim_{n\rightarrow \infty} U_p^{n!}$ rather than the anti-ordinary one, in which case one similarly constructs a quotient map of $\Lambda$-adic Galois representations 
\begin{equation}
\label{Vphi-hida*}
\varpi_{\hphi}: e H^1_{\et}(\bar X_\infty,\Z_p(1)) := e  \invlim_{\ \ \varpi_{2*}}  H^1_{\et}(\bar X_r,\Z_p(1)) \lra  
\mathbb{V}_{\hphi}(M). \end{equation}

\subsection{Families of Dieudonn\'e modules}
\label{subsec:Dieu}

Let $\mathbf{B}_{\dR}$ denote Fontaine's field of de Rham periods, $\mathbf{B}_{\dR}^+$ be its ring of integers and $\log[\zeta_{p^{\infty}}]$ denote the uniformizer of $\mathbf{B}_{\dR}^+$ associated to a norm-compatible system $\zeta_{p^{\infty}} = \{ \zeta_{p^n}\}_{n\geq 0}$ of $p^n$-th roots of unity.
 (cf.\,e.g.\cite[\S 1]{BK}). 
For any finite-dimensional de Rham Galois representation $V$ of $G_{\Q_p}$ with coefficients in a finite extension $L_p/\Q_p$, define the de Rham Dieudonn\'e module
$$
D(V) = (V\otimes \mathbf{B}_{\dR})^{G_{\Q_p}}.
$$
It is an $L_p$-vector space of the same dimension as $V$,
 equipped  with a descending exhaustive filtration $\Fil^j D(V) = (V\otimes \log^j[\zeta_{p^{\infty}}]\mathbf{B}^+_{\dR})^{G_{\Q_p}}$ by $L_p$-vector subspaces.

Let $\mathbf{B}_{\cris} \subset \mathbf{B}_{\dR}$ denote Fontaine's ring of crystalline $p$-adic periods. 
If $V$ is crystalline (which is always the case if it arises as a subquotient of the \'etale cohomology of an algebraic variety with good reduction at $p$), then there is a canonical isomorphism
$$
D(V) \simeq  (V\otimes \mathbf{B}_{\cris})^{G_{\Q_p}},
$$
which furnishes $D(V)$ with a linear action of a Frobenius endomorphism $\Phi$. 

In \cite{BK} Bloch and Kato introduced a collection of subspaces of the local Galois cohomology group $H^1(\Q_p,V)$, denoted respectively 
$$
H^1_{\mathrm{e}}(\Q_p,V) \subseteq H^1_{\fin}(\Q_p,V) \subseteq H^1_{\mathrm{g}}(\Q_p,V) \subseteq H^1(\Q_p,V), 
$$ 
and constructed homomorphisms 
\begin{equation}\label{logBK}
 \log_{\mathrm{BK}}: H^1_{\mathrm{e}}(\Q_p,V) \stackrel{\sim}{\lra} D(V)/\big(\Fil^0D(V) + D(V)^{\Phi=1}\big)
 \end{equation}
and
\begin{equation}\label{expBK}
 \exp^*_{\mathrm{BK}}: H^1(\Q_p,V)/H^1_{\mathrm{g}}(\Q_p,V) \stackrel{\sim}{\lra} \Fil^0 D(V)
 \end{equation}
that are usually referred to as the Bloch-Kato logarithm and dual exponential map.

We illustrate the above Bloch-Kato homomorphisms with a few basic examples that shall be used several times in the remainder of this article.

\begin{example}\label{H1f}
As shown e.g.\,in \cite{BK}, \cite[\S 2.2]{Be}, for any unramified character $\psi$ of $G_{\Q_p}$ and all $n\in \Z$ we have:
\begin{enumerate}
\item[(a)] If $n \geq 2$, or $n=1$ and $\psi\ne 1$, then $H_{\mathrm{e}}^1(\Q_p,L_p(\psi \varepsilon_{\cyc}^n))  = H^1(\Q_p,L_p(\psi \varepsilon_{\cyc}^n))$ is one-dimensional over $L_p$ and the Bloch-Kato logarithm induces an isomorphism
$$
 \log_{\mathrm{BK}}: H^1(\Q_p,L_p(\psi \varepsilon_{\cyc}^n)) \stackrel{\sim}{\lra} D(L_p(\psi \varepsilon_{\cyc}^n)).
$$

\item[(b)] If $n < 0$, or $n=0$ and $\psi\ne 1$, then $H_{\mathrm{g}}^1(\Q_p,L_p(\psi \varepsilon_{\cyc}^n))=0$ and $H^1(\Q_p,L_p(\psi \varepsilon_{\cyc}^n))$ is one-dimensional. The dual exponential gives rise to an isomorphism
$$
\exp^*_{\mathrm{BK}}: H^1(\Q_p,L_p(\psi \varepsilon_{\cyc}^n))  \stackrel{\sim}{\lra} \Fil^0 D(L_p(\psi \varepsilon_{\cyc}^n)) = D(L_p(\psi \varepsilon_{\cyc}^n)).
$$

\item[(c)] Assume $\psi=1$. If $n=0$, then $H^1(\Q_p,L_p)$ has dimension $2$ over $L_p$, $H^1_{\fin}(\Q_p,L_p)=H^1_{\mathrm{g}}(\Q_p,L_p)$ has dimension $1$  and $H^1_{\mathrm{e}}(\Q_p,L_p)$ has dimension $0$ over $L_p$. The Bloch-Kato dual exponential map induces an isomorphism 
$$
 \exp^*_{\mathrm{BK}}: H^1(\Q_p,L_p)/H^1_{\mathrm{f}}(\Q_p,L_p) \stackrel{\sim}{\lra} \Fil^0 D(L_p) = D(L_p) = L_p.
$$
Class field theory identifies $H^1(\Q_p,L_p)$ with $\mathrm{Hom}_{\mathrm{cont}}(\Q_p^\times,\Q_p)\otimes L_p$, which is spanned by the homomorphisms $\ordi_p$ and $\log_p$. 


If $n=1$, then $H^1(\Q_p,L_p(1))=H^1_{\mathrm{g}}(\Q_p,L_p(1)) $ is $2$-dimensional and $H^1_{\fin}(\Q_p,L_p(1))=H^1_{\mathrm{e}}(\Q_p,L_p(1))$ has dimension $1$  over $L_p$. As proved e.g.\,in \cite[Prop. 2.9]{Be}, Kummer theory identifies the spaces $H_{\fin}^1(\Q_p,L_p(1)) \subset H^1(\Q_p,L_p(1))$ with $\Z_p^\times  \hat\otimes L_p$ sitting inside $\Q_p^\times \hat\otimes L_p$. Under this identification, the Bloch-Kato logarithm is the usual $p$-adic logarithm on $\Z_p^\times$. 

\end{enumerate}
\end{example}

Let $\hat\Z_p^{\nr}$ denote the ring of integers of the completion of the maximal unramified extension of $\Q_p$. If $V$ is unramified then there is a further canonical isomorphism
\begin{equation}\label{DV}
D(V) \simeq  (V\otimes \hat\Z_p^{\nr})^{G_{\Q_p}}.
\end{equation}

Let  $\phi$ be an eigenform (with respect to the good Hecke operators) of weight $k = \kk+2\geq 2$, level $M$ and character $\chi$, with fourier coefficients in a finite extension $L_p$ of $\Q_p$.
The comparison theorem \cite{Fal} of Faltings-Tsuji  combined with \eqref{Vphi-class} asserts that there is a natural isomorphism
$$
D(V_{\phi}(M)) \simeq H^1_{\dR}(X_1(M),\cH^{\kk}(1))[\phi]
$$
of Dieudonn\'e modules over $L_p$. Note that $D(V_{\phi}(M))$ is the direct sum of several copies of the two-dimensional Dieudonn\'e module $D(V_{\phi})$.

Assume that $p\nmid M$ and $\phi$ is ordinary at $p$. Then $V_\phi(M)$ is crystalline and $\Phi$ acts on $D(V_{\phi}(M))$ as 
\begin{equation}\label{Phi=Up}
\Phi 
 = \chi(p) p^{\kk+1} U_p^{-1}.
\end{equation}
In particular the eigenvalues of $\Phi$ on $D(V_{\phi}(M))$ are $\chi(p) p^{\kk+1} \alpha^{-1}_\phi = \beta_\phi$ and $\chi(p) p^{\kk+1} \beta^{-1}_\phi = \alpha_\phi$, the two roots of the Hecke polynomial of $\phi$ at $p$. For future reference, recall from \cite[Theorem 1.3]{DR1} the Euler factors
 \begin{equation}\label{E0E1}
\cE_0(\phi)  :=  1- \chi^{-1}(p)  \beta_{\phi}^2  p^{1-k} = 1-\frac{\beta_{\hphi_x^\circ}}{\alpha_{\hphi_x^\circ}}, \quad
\cE_1(\phi)  :=  1-  \chi(p)  \alpha_{\phi}^{-2} p^{k-2}.
\end{equation}

Let $\phi^*=\phi\otimes \bar\chi\in S_k(M,\bar\chi)$ denote the twist of $\phi$ by the inverse of its nebentype character.
Poincar\'e duality induces a perfect pairing
$$
\langle \,,\rangle: D(V_{\phi}(M)) \times D(V_{\phi^*}(M)) \lra D(L_p)=L_p.
$$ 

The exact sequence  \eqref{dec} induces in this setting an exact sequence of Dieudonn\'e modules
\begin{equation}\label{dec2}
0 \lra D(V_\phi^+(M)) \stackrel{i}{\lra} D(V_\phi(M)) \stackrel{\pi}{\lra} D(V_\phi^-(M)) \lra 0.
\end{equation}
Since $V_\phi^-(M$ is unramified, we have $D(V_\phi^-(M)) \simeq  (V_\phi^-(M)\otimes \hat\Z_p^{\nr})^{G_{\Q_p}}$. This submodule  may also be characterized as the eigenspace $D(V_\phi^-(M)) = D(V_\phi(M))^{\Phi=\alpha_\phi}$ of eigenvalue $\alpha_\phi$ for the action of frobenius.

The rule $\testphi \mapsto \omega_{\testphi}$ that attaches to a modular form its associated differential form gives rise to an isomorphism $S_k(M,\chi)_{L_p}[\phi] \stackrel{\sim}{\lra} \Fil^0(D(V_\phi(M))) \subset D(V_\phi(M))$. Moreover, the map $\pi$ of \eqref{dec2} induces an isomorphism 
\begin{equation}\label{isoFil0}
S_k(M,\chi)_{L_p}[\phi] \stackrel{\sim}{\lra} \Fil^0(D(V_\phi(M))) \stackrel{\pi}{\lra} D(V_\phi^-(M)).
\end{equation}
Any element $\omega\in D(V_{\phi^*}^-(M))$ gives rise to a linear map
$$
\omega:  D(V_{\phi}^+(M))  \, \lra \, L_p, \quad \eta \mapsto  \langle \eta,\pi^{-1}(\omega) \rangle.
$$
Similarly, any $\eta \in D(V_{\phi^*}^+(M))$ may be identified with a linear functional
$$
\eta:  D(V_{\phi}^-(M)) \, \lra \, L_p, \quad  \omega \mapsto \langle \pi^{-1}(\omega),\,\eta \rangle,
$$
and given $\testphi \in S_k(M,\chi)_{L_p}[\phi]$ we set $\eta_{\testphi}: D(V_{\phi^*}^-(M)) \ra L_p,$ $\varphi \mapsto \eta_{\testphi}(\varphi) = \frac{\langle \testphi,\varphi\rangle}{\langle \testphi,\testphi^*\rangle}$.

Let now $\tilde\Lambda$ be a finite flat extension of the Iwasawa algebra $\Lambda$ and let $\hU$ denote a free $\tilde\Lambda$-module of finite rank equipped with an {\em unramified} $\tilde\Lambda$-linear action of $G_{\Q_p}$. Define the $\Lambda$-adic Dieudonn\'e module
$$ \hD(\hU) :=  (\hU \hat\otimes \hat\Z_p^{\nr})^{G_{\Q_p}}.$$
As shown in e.g.\,\cite[Lemma 3.3]{Oc1}, $\hD(\hU)$ is a free module over $\tilde\Lambda$ of the same rank as $\hU$.

Examples of such $\Lambda$-adic Dieudonn\'e modules arise naturally in the context of families of modular forms thanks to Theorem \ref{Vphi-thm}. Indeed, let $\hphi$ be  a Hida family of tame level $M$ and character $\chi$, and let $\hphi^*$ denote the $\Lambda$-adic modular form obtained by twisting $\hphi$ by $\bar\chi$. 

Let $\hV_{\hphi}$ and $\hV_{\hphi}(M)$ denote the global $\Lambda$-adic Galois representations described in \eqref{Vphi-hida}.  It follows from \eqref{Wiles-ord} that to the restriction  of  $\hV_{\hphi}$ to $G_{\Q_p}$ one might associate two natural unramified $\Lambda[G_{\Q_p}]$-modules of rank one, namely
$$
\hV_{\hphi}^-\simeq \Lambda_{\hphi}(\psi_{\hphi}) \quad \mbox{and} \quad \mathbb{U}^+_{\hphi}=\mathbb{V}^+_{\hphi}(\chi^{-1} \varepsilon_{\cyc} \underline{\varepsilon}_{\cyc}^{-1}).
$$
Define similarly the unramified modules $\hV_{\hphi}^-(M)$ and $\mathbb{U}^+_{\hphi}(M)$.

Let
\begin{equation}
\label{eqn:hf-isotypic}
S^{\ordi}_{\Lambda}(M,\chi)[\hphi] := \left\{ \testhphi\in S^{\ordi}_{\Lambda}(M,\chi) \quad \mbox{s.t.} \quad 
\left|\begin{array}{ll}
T_\ell \testhphi = a_\ell(\hphi) \testhphi,  & \forall \ell\nmid Mp, \\
U_p \testhphi = a_p(\hphi) \testhphi & 
\end{array} \right. \right\}, 
\end{equation}

For any crystalline arithmetic point $x\in \cW_{\hphi}^\circ$ of weight $k$, the specialization of a $\Lambda$-adic test vector $\testphi\in S^{\ordi}_{\Lambda}(M,\chi)[\hphi]$ at $x$ is a classical eigenform $\testphi_x\in S_k(Mp,\chi)$ with coefficients in $L_p=x(\Lambda_{\hphi})\otimes \Q_p$ and the same eigenvalues as $\phi_x$ for the good Hecke operators. 

Likewise, define
\begin{equation*}
\label{eqn:hf-isotypic2}
S^{\ordi}_{\Lambda}(M,\bar\chi)^\vee[\hphi]=\left\{ \mathbf{\eta}:  S^{\ordi}_{\Lambda}(M,\bar\chi) \ra \Lambda_{\hphi}  \left|\begin{array}{ll}
 \mathbf{\eta} \circ T_\ell^*  = a_\ell(\hphi)  \mathbf{\eta},  & \forall \ell\nmid Mp, \\
 \mathbf{\eta} \circ U_p^*  = a_p(\hphi)   \mathbf{\eta} &   
\end{array} \right. \right\}
\end{equation*}

Let $\mathcal{Q}_{\hphi}$ denote the field of fractions of $\Lambda_{\hphi}$. Associated to any test vector $\testhphi\in S^{\ordi}_{\Lambda}(M,\chi)[\hphi]$, \cite[Lemma 2.19]{DR1} describes a $\mathcal{Q}_{\hphi}$-linear dual test vector
\begin{equation}\label{dualtest}
\testhphi^\vee \in S^{\ordi}_{\Lambda}(M,\bar{\chi})^\vee[\hphi]\,\hat\otimes \,\mathcal{Q}_{\hphi}
\end{equation}
such that for any $\boldsymbol{\varphi}\in  S^{\ordi}_{\Lambda}(M,\bar\chi)$ and any point $x\in \cW_{\hf}^\circ$, 
$$
x(\testhphi^\vee(\boldsymbol{\varphi})) =  \frac{\langle \testhphi_x, \boldsymbol{\varphi}_x\rangle} {\langle \testhphi_x,\testhphi^*_x \rangle}
$$ 
where $\langle \, ,\rangle$ denotes the pairing induced by Poincar\'e duality on the modular curve associated to the congruence subgroup $\Gamma_1(M)\cap \Gamma_0(p)$.  This way, the specialization of a $\Lambda$-adic dual test vector $\testhphi^\vee\in S^{\ordi}_{\Lambda}(M,\bar\chi)^\vee[\hphi]$ at $x$ gives rise to a linear functional $$\eta_{\testhphi_x}: S_k(Mp,\bar\chi)[\phi^*_x] \lra L_p.$$

A natural $\mathcal{Q}_{\hf}$-basis of $S^{\ordi}_{\Lambda}(M,\chi)[\hphi]\,\hat\otimes \, \mathcal{Q}_{\hphi}$ is given by the $\Lambda$-adic modular forms $\hphi(q^d)$ as $d$ ranges over the positive divisors of $M/M_{\hphi}$ and it is also obvious that $\{ \hphi(q^d)^\vee:  d\mid \frac{M}{M_{\hphi}} \}$ provides a $\mathcal{Q}_{\hphi}$-basis of $S^{\ordi}_{\Lambda}(M,\bar{\chi})^\vee[\hphi] \,\hat\otimes \, \mathcal{Q}_{\hphi}$.

The following statement shows that the linear maps described above can be made to vary in families. 

\begin{proposition}\label{klz-prop}  For any $\Lambda$-adic test vector $\testhphi\in S_{\Lambda}^{\ordi}(M,\chi)[\hphi]$ there exist homomorphisms of $\Lambda_{\hphi}$-modules
$$
\omega_{\testhphi}: D(\mathbb{U}^+_{\hphi^*}(M)) \, \lra \, \Lambda_{\hphi}, \quad \eta_{\testhphi}: D(\mathbb{V}_{\hphi^*}^-(M)) \lra \mathcal{Q}_{\hphi},
$$
whose specialization at a classical point $x\in \cW_{\hphi}^\circ$ such that $\phi_x$ is the ordinary stabilization of an eigenform $\hphi^\circ_x$ of level $M$ are, respectively

\begin{enumerate}

\item $x\circ \omega_{\testhphi} \, = \, \cE_0(\hphi^\circ_x)  \,e \, \varpi_1^*(\omega_{\testhphi^\circ_x})$ as functionals on $D(U^+_{\hphi^*_x}(Mp))$.

\item $
x\circ \eta_{\testhphi} 
=  \frac{1}{\cE_1(\hphi^\circ_x)} \cdot e \, \varpi_1^* (\eta_{\testhphi_x^{\circ}})$ 
as functionals on $D(V^-_{\hphi^*_x}(Mp))$.

\end{enumerate}
\end{proposition}

\begin{proof} This is essentially a reformulation of \cite[Propositions 10.1.1 and 10.1.2]{KLZ}, which in turn builds on \cite{Oh}. Namely, the first claim in Prop.\,10.1.2 of loc.\,cit.\,asserts that $\omega_{\testhphi}$ exists such that at any $x\in \cW_{\hphi}^\circ$ as above, $x\circ \omega_{\testhphi} \, = \omega_{\testhphi_x} = \mathrm{Pr}^{\alpha *}(\omega_{\testphi_x^\circ})$ where $\mathrm{Pr}^{\alpha *}$ is the map defined in \cite[10.1.3]{KLZ} sending $\testphi_x^\circ$ to its ordinary $p$-stablilization $\testphi_x$. Note that $\varpi_1^*(\phi_x^\circ) = \frac{\alpha_{\phi_x^\circ} \testphi_x}{ \alpha_{\phi_x^\circ} - \beta_{\phi_x^\circ}}  -  \frac{\beta_{\phi_x^\circ} \testphi'_x}{\alpha_{\phi_x^\circ}- \beta_{\phi_x^\circ}} $, where $\testphi'_x$ denotes the non-ordinary specialization of $\testphi_x^\circ$. Since $e  \omega_{\testphi'_x} = 0$ and $\cE_0(\hphi^\circ_x)  = \frac{ \alpha_{\phi_x^\circ} - \beta_{\phi_x^\circ}}{\alpha_{\phi_x^\circ}}$ the claim follows.

The second part of \cite[Proposition 10.1.2]{KLZ} asserts that there exists a $\Lambda$-adic functional $\tilde\eta_{\testhphi}$ such that for all $x$ as above:
$$
x\circ  \tilde\eta_{\testhphi} = \frac{\mathrm{Pr}^{\alpha *} \eta_{\testhphi^\circ_x}}{ \lambda(\hphi^\circ_x) \cE_0(\hphi^\circ_x) \cE_1(\hphi^\circ_x)}
$$
as $L_p$-linear functionals on $D(V^-_{\hphi^*_x}(Mp))$. Here  $\lambda(\hphi^\circ_x)\in \bar\Q^\times$ denotes the pseudo-eigenvalue of $\hphi^\circ_x$, which we recall is the scalar given by
\begin{equation}\label{pseudo}
W_M(\hphi^\circ_x) = \lambda(\hphi^\circ_x) \cdot \hphi^{\circ*}_x,
\end{equation}
where $W_M: S_{k}(M,\chi) \ra S_{k}(M,\chi^{-1})$ stands for the Atkin-Lehner operator. Since we are assuming that $\Lambda_{\hphi}$ contains the $M$-th roots of unity (cf.\,the remark right after Definition \ref{def-Hida-fam}), Prop. 10.1.1\,of loc.\,cit.\,shows that there exists an element $\lambda(\hphi)\in \Lambda_{\hphi}$ interpolating the pseudo-eigenvalues of the classical $p$-stabilized specializations of $\hphi$. The claim follows by taking $\eta_{\testhphi} = \lambda(\hphi) \tilde\eta_{\testhphi}$. The same argument as above yields that for all $x$ as above,  $x\circ \eta_{\testhphi} = \cE_0(\hphi^\circ_x) \frac{ e \varpi_1^*\eta_{\testhphi^\circ_x}}{\cE_0(\hphi^\circ_x) \cE_1(\hphi^\circ_x)}$, which amounts to the statement of the proposition.
 \end{proof}


\section{Generalised Kato classes}
\label{sec:GKC}

\subsection{A compatible collection of cycles}
This section  defines a collection of codimension two cycles in
 $X_1(Mp^r)^3$ indexed by elements of $(\Z/p^r\Z)^{\times 3}$ and records some of their 
 properties.
 
 We retain the notations that were in force in Section \ref{subsec:hida} regarding the meanings of the curves $X=X_1(M)$,
 $X_r= X_1(Mp^r)$ and  $X_{r,s}$.
 In addition,  let 
 $$\YY(p^r) := Y\times_{X(1)} Y(p^r), \quad \XX(p^r) := X\times_{X(1)} X(p^r)$$ 
  denote the (affine and projective, respectively) modular curve over $\Q(\zeta_r)$ with full level $p^r$ structure.
 The curve $\YY(p^r)$ classifies 
triples  $(A, P,Q)$  in which $A$ is an elliptic curve with $\Gamma_1(M)$ level structure and
  $(P,Q)$ is a basis 
for $A[p^r]$  satisfying
$\langle P,Q \rangle = \zeta_r$, 
where  $\langle \ , \rangle $ denotes the Weil pairing  and   $\zeta_r$ is a fixed primitive $p^r$-th root of unity.
 The curve  $\XX(p^r)$ is geometrically connected   but does not descend to a curve over $\Q$,
 as can be seen by noting that the description of its moduli problem depends on the choice of $\zeta_r$. The covering 
  $\XX(p^r)/X$  
is  Galois with Galois group
$\SL_2(\Z/p^r\Z)$, acting  on the left by the rule
\begin{equation}
\label{eqn:galois-auts}
  \left(\begin{array}{cc} a & b \\ c & d\end{array}\right)  (A, P, Q) = 
 (A, aP+b Q,  cP+d Q).
 \end{equation} 
Consider the natural projection map
\begin{equation}
\label{eqn:varpi1r3}
\varpi_1^r \times \varpi_1^r \times \varpi_1^r:   X_r^3 \lra X^3 
\end{equation}
  induced on triple products
  by the map $\varpi_1^r$ 
of \eqref{eqn:def-varpi1}.
Write $\Delta\subset X^3$ for the usual diagonal cycle, namely the image of $X$ under the diagonal embedding $x \mapsto (x,x,x)$.
Let 
$ \Delta_r $ be the fiber product $\Delta\times_{X^3} X_r^3$ via the natural inclusion and the map 
of \eqref{eqn:varpi1r3}, 
which fits into the cartesian diagram
$$  \xymatrix{
 \Delta_r \ar@{^{(}->}[r]  \ar@{>>}[d] & X_r^3 \ar@{>>}[d] \\
 \Delta \ar@{^{(}->}[r] & X^3. }
 $$ 
 An element of a  $\Z_p$-module  $\Omega$ is said to be
 {\em primitive} if it does not belong to 
 $p\Omega$, and the set of such primitive elements is denoted $\Omega'$.
 Let 
 $$\Sigma_r := ((\Z/p^r\Z \times \Z/p^r\Z)')^3  \subset  ((\Z/p^r\Z)^2)^3 $$ 
 be   the set of  triples of  primitive row 
  vectors of length $2$
 with entries in 
$\Z/p^r\Z$, equipped with the    action of $\GL_2(\Z/p^r\Z)$ acting diagonally
by right multiplication.

  \begin{lemma}
  \label{lemma:geometric-components}
 The geometrically irreducible components of $\Delta_r$ are defined over $\Q(\zeta_r)$ 
 and are in canonical   bijection with the  set of left orbits
 $$  \Sigma_r/\SL_2(\Z/p^r\Z).$$
 \end{lemma}
\begin{proof}
  Each triple 
$$(v_1,v_2,v_3) = \left( (x_1,y_1), (x_2,y_2), (x_3,y_3) \right)
\in \Sigma_r$$
 determines a   morphism 
$$ \varphi_{(v_1,v_2,v_3)}: \XX(p^r) \lra  \Delta_r\subset X_r^3 $$
of curves   over $\Q(\zeta_r)$,
defined in terms of the moduli descriptions on $\YY(p^r)$ by 
$$ 
(A,P,Q) \ \  \mapsto \ \  (\  (A,  x_1 P + y_1 Q), \  (A, x_2 P + y_2 Q), \ (A, x_3 P + y_3 Q) \ ).$$
It is easy to see that if  two elements $(v_1,v_2,v_3)$ and $(v_1',v_2',v_3') \in \Sigma_r$ 
satisfy 
$$ (v_1',v_2',v_3') =  (v_1,v_2,v_3) \gamma , \quad \mbox{ with } \gamma \in \SL_2(\Z/p^r\Z),$$
then 
$$ \varphi_{(v_1',v_2',v_3')} = \varphi_{(v_1,v_2,v_3)} \circ \gamma,$$
where $\ \gamma$ is being viewed as an   automorphism of $\XX(p^r)$  as in 
\eqref{eqn:galois-auts}.
It follows that the  geometrically irreducible cycle
$$ \Delta_r(v_1,v_2,v_3) := \varphi_{(v_1,v_2,v_3)*}(\XX(p^r)) $$
depends only on the $\SL_2(\Z/p^r\Z)$-orbit of $(v_1,v_2,v_3)$. 

Since $\SL_2(\Z/p^r\Z)$ acts transitively on $(\Z/p^r\Z \times \Z/p^r\Z)'$, one further checks that the collection of cycles $ \Delta_r(v_1,v_2,v_3)$ for $(v_1,v_2,v_3)\in \Sigma_r/\SL_2(\Z/p^r\Z)$ do not overlap on $Y_r^3$
and cover $ \Delta_r$. Hence the irreducible components of $\Delta_r$ are precisely  $\Delta_r(v_1,v_2,v_3)$ for $(v_1,v_2,v_3)\in \Sigma_r/\SL_2(\Z/p^r\Z)$.
\end{proof}

The quotient $\Sigma_r/\SL_2(\Z/p^r\Z)$
is equipped with a natural determinant map 
$$ D: \Sigma_r/\SL_2(\Z/p^r\Z)  \lra   (\Z/p^r\Z)^3$$
defined by
$$  D\left( (x_1 y_1), 
(x_2, y_2),
(x_3, y_3) \right) :=  
\left( \left| \begin{array}{cc} x_2 & y_2 \\ x_3 & y_3 \end{array}\right|,
 \left|\begin{array}{cc} x_3 & y_3 \\ x_1 & y_1 \end{array}\right|,
    \left|\begin{array}{cc} x_1 & y_1 \\ x_2 & y_2 \end{array}\right| \right).$$
For each   $[d_1,d_2,d_3]\in (\Z/p^r\Z)^3$, we can then write 
$$ \Sigma_r[d_1,d_2,d_3] := \left \{ (v_1,v_2,v_3)\in \Sigma_r \mbox{ with } D(v_1,v_2,v_3) = (d_1,d_2,d_3) \right\}.$$
The group  $\SL_2(\Z/p^r\Z)$ operates {\em simply transitively} on $\Sigma_r[d_1,d_2,d_3]$ if (and only if) 
\begin{equation}
\label{eqn:def-Gr-tilde}
 [d_1,d_2,d_3] \in  I_r := (\Z/p^r\Z)^{\times3}.
 \end{equation}
In particular, if $(v_1,v_2,v_3)$ belongs to $\Sigma_r[d_1,d_2,d_3]$, then the cycle 
$\Delta_r(v_1,v_2,v_3)$ depends only on $[d_1,d_2,d_3] \in I_r$ and will henceforth be denoted
\begin{equation}\label{Dabc} 
\Delta_r[d_1,d_2,d_3] \in \CH^2( X_r^3).
\end{equation}
A somewhat more intrinsic definition of $\Delta_r[d_1,d_2,d_3]$ as
a curve embedded in $X_r^3$ is that it corresponds to the schematic closure
of the locus of points $((A,P_1),(A,P_2),(A,P_3))$ satisfying
\begin{equation}
\label{eqn:intrinsic-cycle}
 \langle P_2,P_3\rangle = \zeta_r^{d_1}, \qquad
\langle P_3,P_1\rangle = \zeta_r^{d_2}, \qquad
\langle P_1, P_2\rangle = \zeta_r^{d_3}.
\end{equation}
This description also makes it apparent that the cycle $\Delta_r[d_1,d_2,d_3]$ is defined over 
$\Q(\zeta_r)$ but not over $\Q$. 
Let $\sigma_m\in \Gal(\Q(\zeta_r)/\Q) $ be the automorphism associated to $m\in(\Z/p^r\Z)^\times$,  
sending $\zeta_r$ to $\zeta_r^m$. 
The threefold $X_r^3$ is also equipped with an action of the group
\begin{equation}
\label{eqn:def-Gr}
 \tilde G_r := ((\Z/p^r\Z)^\times)^3 = \{ \langle a_1, a_2,a_3\rangle, \quad a_1,a_2,a_3\in (\Z/p^r\Z)^\times 
\}
\end{equation}
 of diamond operators, where the 
automorphism associated to a triple
$(\langle a_1\rangle ,\langle a_2\rangle , \langle a_3\rangle )$ has simply been
  denoted $\langle a_1, a_2, a_3\rangle$. 
 \begin{lemma}
 \label{lemma:galois-diamond-actions-on-cycles}
For all diamond operators 
$\langle a_1, a_2,a_3\rangle \in \tilde G_r$ and all $[d_1,d_2,d_3]\in I_r$, 
\begin{equation}
\label{eqn:diamond-cycles}
\langle a_1,a_2,a_3\rangle
 \Delta_r[d_1,d_2,d_3] = \Delta_r[ a_2 a_3 \cdot d_1, a_1 a_3 \cdot d_2, a_1 a_2 \cdot d_3].
 \end{equation}
  For all $\sigma_m \in \Gal(\Q(\zeta_r)/\Q)$, 
\begin{equation}
\label{eqn:sigma-cycles}
 \sigma_m\Delta_r[d_1,d_2,d_3] = \Delta_r[m \cdot d_1,m \cdot d_2,m \cdot  d_3].
\end{equation}
\end{lemma}
\begin{proof} Equation \eqref{eqn:diamond-cycles} follows directly from the identity
$$ D( a_1 v_1, a_2 v_2, a_3 v_3) = [a_2 a_3, a_1 a_3, a_1 a_2] D(v_1,v_2,v_3).$$ 
The first equality in \eqref{eqn:sigma-cycles}
 is most readily seen from the equation
 \eqref{eqn:intrinsic-cycle} defining the cycle $\Delta_r[d_1,d_2,d_3]$,
since applying the automorphism $\sigma_m\in \Gal(\Q(\zeta_r)/\Q)$ has the effect of replacing $\zeta_r$ by 
$\zeta_r^m$. 
\end{proof}

\begin{remark}
Assume $m$ is a quadratic residue
in $(\Z/p^r\Z)^\times$, which is the case, for instance, when $\sigma_m$ belongs to
$\Gal(\Q(\zeta_r)/\Q(\zeta_1))$. Then it follows from \eqref{eqn:diamond-cycles} and \eqref{eqn:sigma-cycles} that
\begin{equation}\label{rootm}
 \sigma_m\Delta_r[d_1,d_2,d_3] = \langle m,m,m\rangle^{1/2} \Delta_r[d_1,d_2,d_3].
\end{equation}
\end{remark}

 Let us now turn to the compatibility properties of the cycles $\Delta_r[d_1, d_2,d_3]$   
  as the level $r$ varies. Recall  the modular curve  
    $X_{r,r+1}$  classifying generalised 
  elliptic curves together with a distinguished cyclic subgroup of order $p^{r+1}$  and a point of order $p^r$ in it.
 The maps $\mu$,  $\varpi_1$, $\pi_1$, $\varpi_2$ and $\pi_2$ of
\eqref{eqn:def-varpi-12} induce similar
maps on the  triple products:
\begin{equation}
\label{eqn:triangle-of-morphisms}
 \xymatrix{  
 X_{r+1}^3 \ar[d]^{\mu^3} \ar[dr]^{\varpi_1^3}  &  \\
 X_{r,r+1}^3 \ar[r]_-{\pi_1^3}  & X_r^3,   } \qquad \qquad \qquad
 \xymatrix{  
 X_{r+1}^3 \ar[d]^{\mu^3} \ar[dr]^{\varpi_2^3}  &  \\
 X_{r,r+1}^3 \ar[r]_-{\pi_2^3}  & X_r^3.  }
\end{equation}

 A finite morphism $j:V_1\lra V_2$ of varieties induces maps
 $$j_\ast:\CH^j(V_1) \lra \CH^j(V_2), \qquad j^*:\CH^j(V_2) \lra \CH^j(V_1)$$ 
 between Chow groups, and  $j_\ast j^\ast$ agrees with the multiplication by
 deg$(j)$ on $\CH^j(V_2)$. If $j$ is a Galois cover with Galois group $G$,
\begin{equation}
\label{eqn:galois-pushforward}
 j^* j_\ast (\Delta) = \sum_{\sigma\in G} \sigma \Delta.
 \end{equation}
 By abuse of notation we will denote the associated maps on cycles (rather than just on cycle classes) by the same symbols.
 \begin{lemma}
 \label{lemma:compat-cycles-proj}
 For all $r \ge 1$ and all $[d_1',d_2',d_3']\in I_{r+1}$ whose image in 
 $I_r$ is $[d_1,d_2,d_3]$,
 \begin{eqnarray*}
  (\varpi_1^3)_\ast \Delta_{r+1}[d_1',d_2',d_3'] &=& p^3 \Delta_{r}[d_1,d_2,d_3],  \\
 \qquad (\varpi_2^3)_\ast \Delta_{r+1}[d_1',d_2',d_3'] &=& (U_p)^{\otimes 3} \Delta_{r}[d_1,d_2,d_3].
 \end{eqnarray*}
  The  cycles $\Delta_r[d_1,d_2,d_3]$ also  satisfy the distribution relations
 $$ \sum_{[d_1',d_2',d_3']} \Delta_{r+1}[d_1',d_2',d_3'] = (\varpi_1^3)^\ast \Delta_{r}[d_1,d_2,d_3],$$
 where the sum is taken over all triples $[d_1',d_2',d_3'] \in I_{r+1} $ which map  to 
 $[d_1,d_2,d_3]$ in $I_r$. 
 \end{lemma}
 \begin{proof}
 A direct verification based on the definitions shows that 
 the morphisms  $\mu^3$ and $\pi_1^3$ of
 \eqref{eqn:triangle-of-morphisms} induce morphisms
 $$
 \xymatrix{
  \Delta_{r+1}[d_1' d_2',d_3'] \ar[r]^{\mu^3} & \mu_{\ast}^3 \Delta_{r+1}[d_1',d_2',d_3'] \ar[r]^{\pi_1^3} & \Delta_r[d_1,d_2,d_3], }
  $$
  of degrees $1$ and  $p^3$ respectively.
 Hence 
 the restriction of
  $\varpi_1^3$ to $\Delta_{r+1}[d_1',d_2',d_3']$ induces a 
  map of degree $p^3$ from $\Delta_{r+1}[d_1',d_2',d_3']$ to $\Delta_r[d_1,d_2,d_3]$, 
  which implies the first assertion.
It also follows from this that
 \begin{equation}
 \label{eqn:compat-cycles-proj}
  \mu_{\ast}^3 \Delta_{r+1}[d_1',d_2',d_3'] = (\pi_1^3)^\ast \Delta_r[d_1,d_2,d_3].
  \end{equation}
  Applying $(\pi_2^3)_\ast$ to this identity implies that 
  $$ (\varpi_2^3)_\ast \Delta_{r+1}[d_1',d_2',d_3'] = (U_p)^{\otimes 3} \Delta_r[d_1,d_2,d_3].$$
  The second compatibility relation   follows.
  To prove the distribution relation, observe that the sum that occurs in it  is taken over the 
  $p^3$ translates of a fixed $\Delta_{r+1}[d_1',d_2',d_3']$ for the action of the Galois group of
   $X_{r+1}^3$ over $X_{r, r+1}^3$, 
  and hence, by \eqref{eqn:galois-pushforward}, that 
   $$  \sum_{[d_1',d_2',d_3']} \Delta_{r+1}[d_1',d_2',d_3']  = (\mu^*)^3 \mu_{*}^3 \Delta_{r+1}[d_1',d_2',d_3'].$$
   The result then follows from  \eqref{eqn:compat-cycles-proj}. 
  \end{proof}

  \subsection{Galois cohomology classes}
  The goal of this section is to parlay the cycles $\Delta_r[d_1,d_2,d_3]$   
   into   Galois cohomology classes with
  values in $H^1_{\et}(\bar X_r,\Z_p)^{\otimes 3}(2)$, essentially by considering their images  under the
  $p$-adic \'etale Abel-Jacobi map:
  \begin{equation}\label{AJ000}
   \AJ_{\et}: \CH^2(X_r^3)_0 \lra H^1(\Q,H^3_{\et}(\bar X_r^3, \Z_p(2))),
   \end{equation}
  where
  $$ \CH^2(X_r^3)_0 := \ker\left(\CH^2(X_r^3) \lra H^4_{\et}(\bar X_r^3,\Z_p(2)) \right)$$
  denotes the kernel of the \'etale  cycle class map, i.e., the group of null-homologous algebraic cycles
  defined over $\Q$. 
   There are two issues that need to be dealt with.
   Firstly, the cycles $\Delta_r[d_1,d_2,d_2]$  
   need not be null-homologous and have to be suitably modified so that they lie in the domain
   of the  Abel Jacobi map. 
   Secondly, these cycles are defined over $\Q(\zeta_r)$ and not over $\Q$, and it is desirable to  descend the  field of 
   definition of the associated extension classes.
   
   To deal with the first issue, let 
   $q$ be any prime not dividing $Mp$, and let $T_q$ denote the
    Hecke operator attached to this prime. It can be used to 
construct an
algebraic correspondence on $X_r^3$ 
by setting
  $$ \theta_{q} := (T_q - (q+1)) ^{\otimes 3}.$$
 \begin{lemma}
  \label{lemma:theta-annihilates}
  The element $\theta_q$ annihilates  the target  $H^4_{\et}(\bar X_r^3,\Z_p)$
  of the \'etale cycle class map on $\CH^2(X_r^3)$.
  \end{lemma}
  \begin{proof}
The correspondence 
  $T_q$ acts as multiplication by $(q+1)$ on $H^2_{\et}(\bar X_r, \Z_p)$ and  
  $\theta_q$ therefore annihilates all the terms in the K\"unneth decomposition of 
  $H^4_{\et}(\bar X_r,\Z_p)$. 
  \end{proof}

The {\em modified diagonal cycles} in $\CH^2(X_r^3)$  
are  defined by
 the rule 
\begin{equation}\label{mdc}
 \Delta_r^\circ[d_1,d_2,d_3]  :=  \theta_q\Delta_r[d_1,d_2,d_3].
 \end{equation}
 Lemma \ref{lemma:theta-annihilates} 
  shows that they 
  are null-homologous and defined over $\Q(\zeta_r)$. Define
\begin{eqnarray*}
 \kappa_r[d_1,d_2,d_3]  & :=  &
 \AJ_{\et}(\Delta_r^\circ[d_1,d_2,d_3]) \in H^1(\Q(\zeta_r), H^1_{\et}(\bar X_r, \Z_p)^{\otimes 3}(2)).
 \end{eqnarray*}
To deal with the circumstance that the cycles $\Delta_r^\circ[d_1,d_2,d_3]$ are
   only defined over $\Q(\zeta_r)$, and hence that
 the associated cohomology classes   $\kappa_r[d_1,d_2,d_3]$
need not (and in fact, do not)   extend to
 $G_\Q$, it is necessary
  to replace the $\Z_p[\tilde G_r][G_\Q]$-module
   $H^1_{\et}(\bar X_r, \Z_p)^{\otimes 3}(2)$
 by an appropriate twist over $\Q(\zeta_r)$.  Let 
$G_r$   denote the Sylow $p$-subgroup  of  
  the group  $\tilde G_r$   of 
  \eqref{eqn:def-Gr}, and let 
 $G_\infty  := \invlim G_r.$
 Let 
 $$ \Lambda(G_r)  := \Z_p[G_r], \qquad  \Lambda(G_\infty) = \Z_p[[G_\infty]] $$
 be the finite group ring  attached to  $G_r$  
  and the associated Iwasawa algebra, respectively.
  
Let $ \Lambda(G_r)(\pm \half)$ 
denote the Galois module which is isomorphic to
$ \Lambda(G_r)$ as a $\Lambda(G_r)$-module, and on which 
  the Galois group  $G_{\Q(\zeta_1)}$ is made to act   via its quotient 
 $\Gal(\Q(\zeta_r)/\Q(\zeta_1)) = 1 + p \Z/p^r\Z$, the element $\sigma_m$ acting as multiplication by the group-like
 element $\langle m,m,m\rangle^{\mp 1/2}$. 
   Let $ \Lambda(G_\infty)(\pm \half)$ 
 denote the projective limit of the $ \Lambda(G_r)(\pm \half)$. 
 It follows from the definitions that if
 $$ \nu_{\kk,\llll,\mm}: \Lambda(G_r) \lra \Z/p^r\Z, \qquad \mbox{ or }  \qquad
 \nu_{\kk,\llll,\mm}: \Lambda(G_\infty) \lra \Z_p$$
 is the  homomorphism sending $\langle a_1,a_2,a_3\rangle$  to $a_1^\kk a_2^\llll a_3^\mm$,
then
\begin{equation}
\label{eqn:spec-klm-Lambdar-tilde}
 \Lambda(G_r)(\half) \otimes_{ \nu_{\kk,\llll,\mm} } \Z/p^r\Z = (\Z/p^r\Z)( \varepsilon_{\cyc}^{-(\kk + \llll +\mm )/2}), 
 \end{equation}
 where the tensor product is taken over $\Lambda(G_r)$, and similarly for $G_\infty$. 
In particular if $\kk+\llll+\mm = 2t$ is an even integer,
 \begin{equation}
\label{eqn:spec-klm-Lambda-tilde}
 \Lambda(G_\infty)(\half) \otimes_{ \nu_{\kk,\llll,\mm} } \Z_p =  \Z_p(-t)(\omega^t)
  \end{equation}
 as $G_\Q$-modules.
 More generally, if $\Omega$ is any $\Lambda(G_\infty)$ module, 
 write
 $$ \Omega(\half) := \Omega \otimes_{\Lambda(G_\infty) } \Lambda(G_\infty)(\half),
 \qquad 
  \Omega(\mhalf) := \Omega \otimes_{\Lambda(G_\infty) } \Lambda(G_\infty)(\mhalf),
 $$
 for the  relevant twists of $\Omega$, which are isomorphic to
 $\Omega$ as a $\Lambda(G_\infty)[G_{\Q(\mu_{p^\infty})}]$-module but are endowed with   different 
   actions of $G_\Q$.

 There is   a canonical
 Galois-equivariant $\Lambda(G_r)$-hermitian 
 bilinear, $ \Lambda(G_r)$-valued  pairing 
 \begin{equation}
 \label{eqn:pairing-LambdaG}
 \langle\! \langle \ , \ \rangle\!\rangle_r: H^1_{\et}(\bar X_r, \Z_p)^{\otimes 3}(2)(\half) \    \times    \ H^1_{\et}(\bar X_r,\Z_p)^{\otimes 3}(1)(\half) \ \  \lra  \ \  \Lambda(G_r),
 \end{equation}
 given by the formula
 $$ \langle\!\langle a,  b \rangle\!\rangle_r := \sum_{\sigma=\langle d_1,d_2,d_3\rangle \in G_r } \langle a^\sigma, b\rangle_{X_r} \cdot \langle d_1, d_2, d_3 \rangle,$$
 where 
 $$\langle\ ,\ \rangle_{X_r}: H^1_{\et}(\bar X_r, \Z_p)^{\otimes 3}(2) \times H^1_{\et}(\bar X_r,\Z_p)^{\otimes 3}(1) \lra  H^2_{\et}(\bar X_r, \Z_p(1))^{\otimes 3} = \Z_p$$
 arises from the  Poincar\'e duality between  $H^3_{\et}(\bar X_r^3,\Z_p)(2)$ and  $H^3_{\et}(\bar X_r^3,\Z_p)(1)$. This pairing enjoys the following properties:
 \begin{itemize}
  \item For all $\lambda \in \Lambda(G_r)$, 
 $$ \langle\!\langle \lambda a,  b\rangle\!\rangle_r = \lambda^* \langle\!\langle   a,   b\rangle\!\rangle_r, \qquad 
  \langle\!\langle   a, \lambda b\rangle\!\rangle_r = \lambda \langle\!\langle   a,   b\rangle\!\rangle_r,$$
where $\lambda^*\in \Lambda(G_r)$ is obtained from $\lambda$ by applying the involution on the group ring which sends every group-like element to its inverse. In particular,  the pairing of
 \eqref{eqn:pairing-LambdaG}
 can and will also be viewed as a $\Lambda(G_r)$-valued
$*$-hermitian pairing
 $$ \langle\! \langle \ , \ \rangle\!\rangle_r: H^1_{\et}(\bar X_r, \Z_p)^{\otimes 3}(2)  \times H^1_{\et}(\bar X_r,\Z_p)^{\otimes 3}(1)  \lra    \Lambda(G_r).$$
 \item For all $\sigma\in G_{\Q(\zeta_1)}$, we have
 $\langle\!\langle \sigma a, \sigma b\rangle\!\rangle_r = \langle\!\langle   a,   b\rangle\!\rangle_r$.
  \item The $U_p$ and $U_p^*$ operators are adjoint to each other under this pairing,
  giving rise to a duality (denoted by the same symbol, by an abuse of notation)
\begin{equation}
\nonumber
\langle\! \langle \ , \ \rangle\!\rangle_r:    e^\ast H^1_{\et}(\bar X_r, \Z_p)^{\otimes 3}(2)(\half) \ \times  \   e H^1_{\et}(\bar X_r,\Z_p)^{\otimes 3}(1)(\half)  \ \ \lra  \ \  \Lambda(G_r).
\end{equation}
\end{itemize}
Define
\begin{eqnarray*}
\Hr  & :=  &   \mathrm{Hom}_{\Lambda(G_r)}(H^1_{\et}(\bar X_r, \Z_p)^{\otimes 3}(1)(\half), \Lambda(G_r)) \simeq H^1_{\et}(\bar X_r, \Z_p)^{\otimes 3}(2)(\half), \\
\Hrord  & :=  &   \mathrm{Hom}_{\Lambda(G_r)}(    e H^1_{\et}(\bar X_r, \Z_p)^{\otimes 3}(1)(\half),  \Lambda(G_r)) \simeq e^\ast H^1_{\et}(\bar X_r, \Z_p)^{\otimes 3}(2)(\half).
\end{eqnarray*}
The above identifications of $\Z_p[G_{\Q(\zeta_1)}]$-modules follow from the pairing
\eqref{eqn:pairing-LambdaG}.

To descend the field of definition of the classes $\kappa_r[d_1,d_2,d_3]$, 
     we package them together 
  into   elements
  \begin{eqnarray*}
 \hkappa_r[a,b,c] &\in& H^1(\Q(\zeta_r), \Hr)
 \end{eqnarray*}
 indexed by triples 
 \begin{equation}
  \label{eqn:indexed-by-the-triples}
[a,b,c]\in I_1 =  (\Z/p\Z)^{\times 3} = 
 \mu_{p-1}(\Z_p)^3 \subset (\Z_p^\times)^3.
 \end{equation}
   The class $\hkappa_r[a,b,c]$  is defined  by setting, for all $\sigma\in G_{\Q(\zeta_r)}$ and 
 all $\gamma_r \in H^1_{\et}(\bar X_r,\Z_p)^{\otimes 3}(1)$, 
\begin{equation}
\label{eqn:define-kappa-r-prime}
 \hkappa_r[a,b,c](\sigma)(\gamma_r)  = 
 \langle \!\langle \kappa_r[a,b,c](\sigma),  \gamma_r \rangle\!\rangle_r,
 \end{equation}
 where the elements  $a,b,c \in (\Z/p\Z)^\times$ are viewed as elements of
 $(\Z/p^r\Z)^\times$ via the Teichmuller lift alluded to in 
 \eqref{eqn:indexed-by-the-triples}. Note that there is a natural identification
$$
 H^1(\Q(\zeta_r), \Hr)  =  {\rm Ext}^1_{\Lambda(G_r)[G_{\Q(\zeta_r)}]}(H^1_{\et}(\bar X_r, \Z_p)^{\otimes 3}(1),  \Lambda(G_r),
$$
because $H^1_{\et}(\bar X_r, \Z_p)^{\otimes 3}(1) = H^1_{\et}(\bar X_r, \Z_p)^{\otimes 3}(1)(\half)$ {\em as $G_{\Q(\zeta_r)}$-modules} and the $\Lambda(G_r)$-dual of the latter
is $\Hr$. With these definitions we have
 \begin{lemma} 
 \label{lemma:extend}
 The class $\hkappa_r[a,b,c]$ is the restriction to $G_{\Q(\zeta_r)}$ of  a class  
 $$\hkappa_r[a,b,c] \in H^1(\Q(\zeta_1), \Hr)
 = {\rm Ext}^1_{\Lambda(G_r)[G_{\Q(\zeta_1)}]}(H^1_{\et}(\bar X_r, \Z_p)^{\otimes 3}(1)(\half),  \Lambda(G_r)). 
 $$
 Furthermore, for all $m\in \mu_{p-1}(\Z_p)$, 
 $$ \sigma_m \ \hkappa_r[a,b,c] = \hkappa_r[m a,m b,m c].$$
 \end{lemma}
\begin{proof}
We will prove this by giving a more conceptual description of the cohomology class $\hkappa_r[a,b,c]$. 
Let  $|\Delta|$ denote the support of an algebraic cycle $\Delta$, and let
\begin{equation}\label{DeltaDoubleBracket}
  \Delta_r^\circ[[a,b,c]] :=    |\Delta_1^\circ[a,b,c]| \times_{X_1^3} X_r^3
\end{equation}
 denote the  inverse image in $X_r^3$ of $|\Delta_1^\circ[a,b,c]|$, 
  which fits into the cartesian diagram
 $$ \xymatrix{ \Delta_r^\circ[[a,b,c]] \ar@{^{(}->}[r]   \ar@{>>}[d] & X_r^3  \ar@{>>}[d] ^{(\varpi_1^{r-1})^3} \\ |\Delta_1^\circ[a,b,c]| 
 \ar@{^{(}->}[r]   & X_1^3. }  $$
 As in the proof of Lemma 
  \ref{lemma:geometric-components},
observe that 
 $$ \Delta_r^\circ[[a,b,c]] =  \bigsqcup_{[d_1,d_2,d_3]\in   I_r^1}  |\Delta_r^\circ[ad_1, b d_2,cd_3]|$$
 where $I_r^1$ denotes the $p$-Sylow subgroup of $I_r$.
Consider  now the
commutative diagram of
 $\Lambda(G_r)[G_\Q(\zeta_1)]$-modules with exact rows: 
\medskip \medskip
\begin{equation}
\label{eqn:extension-construction}
 \xymatrix{  & & 
 \Lambda(G_r)(\mhalf)  \ar@{^{(}->}[d]^j \\
H^3_{\et}(\bar X_r^3,\Z_p)(2) \ar@{^{(}->}[r] \ar@{>>}[d]^p & 
H^3_{\et}(\bar X_r^3\!\!-\!\Delta_r^\circ[[a,b,c]], \Z_p)(2) \ar@{>>}[r]
& H^0_{\et}(\bar\Delta_r^\circ[[a,b,c]],\Z_p)_0  \   \\
  H^1_{\et}(\bar X_r,\Z_p)^{\otimes 3}(2),}
 \end{equation}
where
\begin{itemize}
\item  the map $j$ is the inclusion
 defined on group-like elements by
$$ j\left(\langle d_1,d_2,d_3\rangle \right) =  \cl(\Delta_r^\circ[a d_2 d_3 , b d_1 d_3 , c d_1 d_2 ]),$$
which is  $G_{\Q(\zeta_1)}$-equivariant by Lemma \ref{lemma:galois-diamond-actions-on-cycles};
\item the middle row arises from the excision exact sequence in \'etale cohomology (cf.\,\cite[(3.6)]{Ja} and \cite[p.\,108]{Milne});
\item   the subscript of $0$ appearing in the rightmost term
in the exact sequence
denotes
the kernel of the cycle class map, i.e.,
$$ H^0_{\et}(\bar\Delta_r^\circ[[a,b,c]],\Z_p)_0 := \ker\left(H^0_{\et}(\bar\Delta_r^\circ[[a,b,c]],\Z_p)_0 \lra 
H^4_{\et}(\bar X_r^3,\Z_p(2)) \right),$$
and the fact that the image of $j$ is contained in $H^0_{\et}(\bar\Delta_r^\circ[[a,b,c]],\Z_p)_0$
follows from Lemma  \ref{lemma:theta-annihilates};
\item  the projection $p$ is the one arising from the K\"unneth decomposition.
\end{itemize}
Taking the pushout and pullback of the extension in 
\eqref{eqn:extension-construction} via the maps $p$ and $j$
yields  an exact sequence of $\Lambda(G_r)[G_\Q(\zeta_1)]$-modules
\begin{equation}
\label{eqn:the-extension}
 \xymatrix{   
0 \ar[r] & H^1_{\et}(\bar X_r,\Z_p)^{\otimes 3}(2) \ar[r]   & 
E_r  \ar[r]
& \Lambda(G_r)(\mhalf)  \ar[r] & 0.}
 \end{equation}
 Taking the $\Lambda(G_r)$-dual of this exact sequence, we obtain
 \begin{equation*}
 \xymatrix{   
0 \ar[r] &   \Lambda(G_r)(\half) \ar[r]   & 
\check{E}_r  \ar[r]
&   H^1_{\et}(\bar X_r,\Z_p)^{\otimes 3}(1)^*    \ar[r] & 0.}
 \end{equation*}
where $M^*$ means the $\Lambda(G_r)$-module obtained from $M$ by letting act $\Lambda(G_r)$ on it by composing with the involution
$\lambda\mapsto \lambda^*$. Twisting this sequence by $(\mhalf)$ and noting that $M^*(\mhalf) \simeq M(\half)^*$ yields an extension
\begin{equation}
\label{eqn:the-extension-bis}
 \xymatrix{   
0 \ar[r] &   \Lambda(G_r) \ar[r]   & 
E'_r  \ar[r]
&   H^1_{\et}(\bar X_r,\Z_p)^{\otimes 3}(1)(\half)^*    \ar[r] & 0.}
 \end{equation}
Since
\begin{equation*} 
H^1_{\et}(\bar X_r,\Z_p)^{\otimes 3}(1)(\half)^*  = \mathrm{Hom}_{\Lambda(G_r)}(H^1_{\et}(\bar X_r,\Z_p)^{\otimes 3}(2)(\half),  \Lambda(G_r)),
\end{equation*}
it follows that
the  cohomology class realizing the extension  $E_r'$
is an element of 
$$H^1(\Q(\zeta_1), \mathrm{Hom}_{\Lambda(G_r)}(H^1_{\et}(\bar X_r,\Z_p)^{\otimes 3}(1)(\half),    \Lambda(G_r))) = H^1(\Q(\zeta_1), \Hr),$$
because the duality afforded by $\langle \!\langle  \ , \ \rangle\!\rangle_r$ is hermitian (and not $\Lambda$-linear). 
When restricted to $G_{\Q(\zeta_r)}$, this class coincides with
   $\hkappa_r[a,b,c]$, and the first assertion follows.

The second assertion is an immediate consequence of the definitions, using 
the Galois equivariance properties of the cycles $\Delta_r[d_1,d_2,d_3]$ given in 
   the first assertion of Lemma \ref{lemma:galois-diamond-actions-on-cycles}.
 \end{proof}
 
 \begin{remark}
 The extension $E_r'$ of \eqref{eqn:the-extension-bis} can also be realised as a subquotient of the  \'etale cohomology group
 $H^3_c(\bar X_r^3\!\!-\!\!\Delta_r^\circ[[a,b,c]],\Z_p)(1)$ with compact supports, 
 in light of the Poincar\'e duality 
 $$ H^3_{\et}(\bar X_r^3\!\!-\!\!\Delta_r^\circ[[a,b,c]],\Z_p)(2) \ \times \ 
 H^3_c(\bar X_r^3\!\!-\!\!\Delta_r^\circ[[a,b,c]],\Z_p)(1) \lra  \Z_p.$$
 \end{remark}

\subsection{$\Lambda$-adic cohomology classes}
Thanks to Lemma \ref{lemma:extend}, we now dispose, for each 
$[a,b,c]\in \mu_{p-1}(\Z_p)^3$, 
of a system
\begin{equation}
\hkappa_r[a,b,c] \in H^1(\Q(\zeta_1),\Hr)
\end{equation}
of cohomology classes indexed by the integers $r\ge 1$, so that $e^* \hkappa_r[a,b,c] \in H^1(\Q(\zeta_1), \Hrord)$.
Let 
$$p_{r+1,r}:   \Lambda(G_{r+1}) \lra  \Lambda(G_r)$$
be the  projection on finite group rings induced from the natural  homomorphism
$  G_{r+1}  \lra   G_r$.
  \begin{lemma}
 \label{lemma:compatibility-bis}
 Let $\gamma_{r+1} \in H^1_{\et}(\bar X_{r+1},\Z_p)^{\otimes 3}(1)$ and 
 $\gamma_r\in H^1_{\et}(\bar X_{r},\Z_p)^{\otimes 3}(1)$ be elements that are compatible under the pushforward
  by $\varpi_1^3$, i.e., that satisfy
 $(\varpi_1^3)_\ast (\gamma_{r+1}) = \gamma_r$. For all $\sigma\in G_{\Q(\zeta_1)}$,
  $$  p_{r+1,r}\left(\hkappa_{r+1}[a,b,c](\sigma)(\gamma_{r+1})\right)  = \hkappa_r[a,b,c](\sigma)(\gamma_r).$$
\end{lemma}
\begin{proof}
This  amounts to the statement that
$$ p_{r+1,r}(\langle\!\langle \kappa_{r+1}[a,b,c], \gamma_{r+1}\rangle\!\rangle_{r+1}) = 
\langle\!\langle \kappa_{r}[a,b,c], \gamma_{r}\rangle\!\rangle_{r}.$$
But the left-hand side of this equation is equal to
$$
  \sum_{ G_r } \langle  (\mu^3)^* (\mu^3)_* \kappa_{r+1}[a d_2' d_3', b d_1' d_3', c d_1' d_2'], \gamma_{r+1} \rangle_{X_{r+1}}  \cdot \langle d_1,d_2,d_3 \rangle, 
$$
where  the sum runs over $\langle d_1,d_2,d_3\rangle \in   G_r$ and  $\langle d_1',d_2', d_3'\rangle$ denotes an (arbitrary) lift of $\langle d_1,d_2,d_3\rangle$
 to $  G_{r+1}$. 
The third  assertion in Lemma \ref{lemma:compat-cycles-proj}  allows us to rewrite this as
\begin{eqnarray*}
 & & \!\!\!\!\!\!\!\!\!\!\!\!\!\!\!\!\!\!\!\!\!\!\!\!
 \sum_{   G_r} \langle  (\varpi_1^3)^\ast \kappa_{r}[a d_2 d_3, b d_1 d_3 , c d_1 d_2], \gamma_{r+1} \rangle_{X_{r+1}}  \cdot \langle d_1,d_2,d_3\rangle \\
 &=&  \sum_{   G_r} \langle  \kappa_{r}[a d_2 d_3, b d_1 d_3, c d_1 d_2],  (\varpi_1^3)_\ast\gamma_{r+1} \rangle_{X_{r}}  \cdot \langle d_1,d_2,d_3\rangle \\ 
 &=&  \sum_{   G_r} \langle  \kappa_{r}[a d_2 d_3, b d_1 d_3, c d_1 d_2],  \gamma_{r} \rangle_{X_{r}}  \cdot \langle d_1,d_2,d_3\rangle \\ 
 &=& \langle\!\langle \kappa_{r}[a,b,c], \gamma_{r}\rangle\!\rangle_{r},
 \end{eqnarray*}
 and the result follows.
\end{proof}

Define
\begin{eqnarray}\label{Hinf}
\Hinfty   :=    &      \mathrm{Hom}_{\Lambda(G_\infty)}(H^1_{\et}(\bar X_\infty^\ast, \Z_p)^{\otimes 3}(1)(\half), \Lambda(G_\infty)) \\
&  \nonumber   =  \mathrm{Hom}_{\Lambda(G_\infty)}(H^1_{\et}(\bar X_1,\cL_\infty^\ast)^{\otimes 3}(1)(\half), \Lambda(G_\infty)),
\end{eqnarray}
where the identification follows from \eqref{eqn:crucial-relation-for-induced}.

Thanks to Lemma \ref{lemma:compatibility-bis},  the classes $\hkappa_r[a,b,c]$
can be packaged into a  compatible  collection. Namely:

\begin{definition}\label{def-kappa-abc}
Set
\begin{equation}\label{kappa-abc}
\hkappa_\infty[a,b,c]  := \left(\hkappa_r[a,b,c]\right)_{r\ge 1} \in  H^1(\Q(\zeta_1),  \Hinfty).
\end{equation}
\end{definition}

It will also be useful to replace the classes
 $\hkappa_\infty[a,b,c]$
 by elements that are essentially  indexed by  triples
 $$ (\omega_1,\omega_2,\omega_3) : (\Z/p\Z^\times)^3 \lra \Z_p^\times$$
 of  tame characters of $\tilde G_r/G_r$. 
 Assume that  the product 
 $\omega_1\omega_2\omega_3$ is an {\em even }   character.
 This assumption is equivalent to requiring that 
 $$ \omega_1 \omega_2\omega_3 = \delta^2, \quad\mbox{ for some } \delta: (\Z/p\Z)^\times  \lra \Z_p^\times.$$
 Note that  for a given  $(\omega_1,\omega_2,\omega_3)$,   there are in fact two  characters $\delta$ as above,
  which differ by the unique quadratic
 character of conductor $p$.
 With the  choices  of $\omega_1,\omega_2,\omega_3$ and $\delta$ in hand, we set 
 \begin{equation}
 \label{eqn:cycles-tame}
  \hkappa_\infty(\omega_1,\omega_2,\omega_3;\delta)  :=   \frac{p^3}{(p-1)^3} \, \cdot \sum_{[a,b,c]}  
\delta^{-1}(abc)  \cdot \omega_{1}(a) \omega_2(b) \omega_3(c) 
 \cdot\hkappa_\infty[bc, ac,ab],
\end{equation}
where  the sum is taken over the triples $[a,b,c]$ of $(p-1)$st roots of unity in $\Z_p^\times$.
The classes $\hkappa_\infty(\omega_1,\omega_2,\omega_3;\delta)$   satisfy the following properties.
 \begin{lemma}
 \label{lemma:iwasawa-cycles}
 For all $\sigma_m \in \Gal(\Q(\zeta_\infty)/\Q)$, 
$$ \sigma_m \hkappa_\infty(\omega_1,\omega_2,\omega_3;\delta) = \delta(m) \hkappa_\infty(\omega_1,\omega_2,\omega_3;\delta).$$
For all diamond operators $\langle a_1, a_2,a_3\rangle \in \mu_{p-1}(\Z_p)^3$ 
$$ \langle a_1,a_2,a_3\rangle
 \hkappa_\infty(\omega_1,\omega_2,\omega_3;\delta) = \omega_{123}(a_1,a_2,a_3) \cdot
     \hkappa_\infty(\omega_1,\omega_2,\omega_3;\delta).$$
\end{lemma}
\begin{proof} This follows from a direct calculation based on the definitions, using the compatibilities of Lemma 
\ref{lemma:galois-diamond-actions-on-cycles}
satisfied by the cycles $\Delta_r[d_1,d_2,d_3]$.
\end{proof}

The classes $\hkappa_\infty[a,b,c]$  and 
$\hkappa_\infty(\omega_1,\omega_2,\omega_3;\delta)$ are
 called the  {\em $\Lambda$-adic cohomology classes} attached to the
triple 
$[a,b,c]\in \mu_{p-1}(\Z_p)^3$ or the quadruple $(\omega_1,\omega_2,\omega_3;\delta)$. 
As will be explained in the next section,  they are  three variable families of cohomology classes parametrised by
 points in  the triple product   $\cW\times \cW\times \cW$ of weight spaces, 
 and taking values in the three-parameter family of self-dual Tate twists of the Galois representations attached
 to the different specialisations of 
 a triple of Hida families.
 

\section{Higher weight balanced specialisations}\label{sec:higherweight}

For every integer $\kk \ge 0$ define
$$
W_1^\kk := H^1_{\et}(\bar X_1, \cH^\kk)
$$
and recall from the combination of  \eqref{eqn:crucial-relation-for-induced}, \eqref{eqn:mom-bis} and \eqref{eqn:wtk-specialisation} the  specialisation map
\begin{equation}
\label{eqn:specialisation-maps-bis}
\spec_\kk^\ast: H^1_{\et}(\bar X_\infty^\ast,\Z_p) =  H^1_{\et}(\bar X_1,  \cL_\infty^\ast) \lra  W_1^\kk. 
\end{equation}

Fix throughout this section a  triple
  $$k = \kk+2, \qquad \ell = \llll+2, \qquad m = \mm+2$$
of  integers $\ge 2$  for which $\kk+\llll+\mm= 2t$ is even. Let 
  $$ \cH^{\kk,\llll,\mm} := \cH^\kk \boxtimes \cH^\llll \boxtimes \cH^\mm$$
  viewed as a sheaf on $ X_1^3$, and
$$
W_1^{\kk,\llll,\mm} := W_1^\kk\otimes W_1^\llll \otimes W_1^\mm (2-t).  
$$

As one readily checks, the $p$-adic Galois representation  $W_1^{\kk,\llll,\mm}$  is 
 Kummer self-dual, i.e., there is an isomorphism of $G_{\Q}$-modules
$$
\mathrm{Hom}_{G_{\Q}}(W_1^{\kk,\llll,\mm},\Z_p(1)) \simeq W_1^{\kk,\llll,\mm}.
$$

 The specialisation maps 
give rise, in light of \eqref{eqn:spec-klm-Lambda-tilde}, to the triple 
product specialisation map
\begin{equation}\label{spklm}
\spec_{\kk,\llll,\mm}^\ast:= \spec_\kk^\ast \otimes  \spec_\llll^\ast \otimes \spec_\mm^\ast:   
\Hinfty \lra  W_1^{\kk,\llll,\mm}
\end{equation}
and to the associated collection of specialised classes
\begin{equation}
 \label{eqn:wt-klm-specialised-classes}
 \kappa_1(\kk,\llll,\mm)[a,b,c] := \spec_{\kk,\llll,\mm}(\hkappa_\infty[a, b,c]) \in  H^1(\Q(\zeta_1), W_1^{\kk,\llll,\mm}).
 \end{equation}
 
 Note that for $(\kk,\llll,\mm)=(0,0,0)$, it follows from the definitions (cf.\,e.g.\,the proof of Lemma \ref{lemma:extend}) that the class $\kappa_1(\kk,\llll,\mm)[a,b,c]$ is simply the image under the \'etale Abel-Jacobi map of the cycle $\Delta_1^\circ[a,b,c]$. 
 
 The main goal  of this section is to offer a similar geometric description for the above classes also when $(k,\ell,m)$ is {\em balanced} and $\kk,\llll,\mm>0$, which
 we assume henceforth for the remainder of this section.
 
 In order to do this, it shall be useful to dispose of an alternate description of 
the extension  \eqref{eqn:the-extension} in terms of the \'etale cohomology of the (open)
three-fold $X_1^3-|\Delta_1^\circ[a,b,c]|$  with values in appropriate sheaves.

   \begin{lemma}
 \label{lemma:disposing-of-kings-sheaf} Let $\cL_r^{\ast \boxtimes{3}}$ denote the exterior tensor product of $\cL_r^{\ast}$, over the triple product  $ X_1^3$.  There is a commutative diagram 
  $$ \xymatrix{     H^3_{\et}(X_r^3,\Z_p)(2)\ar[r] 
   \ar@{=}[d] 
&  H^3_{\et}(X_r^3-\Delta_r^\circ[[a,b,c]],\Z_p)(2)  
 \ar@{=}[d] 
  \ar[r] 
  &
  H^0_{\et}(\Delta_r^\circ[[a,b,c]],\Z_p)         
   \ar@{=}[d]      \\    
    H^3_{\et}(\bar X_1^3, \cL_r^{\ast \boxtimes 3})(2) \ar[r] &
    H^3_{\et}(\bar X_1^3- |\Delta_1^\circ[a,b,c]|,\cL_r^{\ast \boxtimes 3})(2) \ar[r] &
    H^0_{\et}(|\Delta_1^\circ[a,b,c]|,  \cL_r^{\ast \otimes 3})),   }   
   $$
  in which the leftmost maps are injective and the  horizontal sequences are exact.
  \end{lemma}
  
\begin{proof} 
Recall from \eqref{def:Lr-sheaf}
 that 
 $$ \cL_r^{\ast \boxtimes 3} = 
(\varpi_1^{r-1} \times \varpi_1^{r-1} \times \varpi_1^{r-1})_\ast \Z_p,$$
where $$\varpi_1^{r-1}  \times \varpi_1^{r-1} \times\varpi_1^{r-1}: X_r^3 \lra  X_1^3$$
is defined as  in 
\eqref{eqn:varpi1r3}.  The vertical isomorphisms then follow from Shapiro's lemma and the definition of $\Delta_r^\circ[[a,b,c]]$ in \eqref{DeltaDoubleBracket}. The horizontal sequence arises from the excision exact sequence in 
 \'etale cohomology of \cite[(3.6)]{Ja} and \cite[p.\,108]{Milne}.
 \end{proof}

  \begin{lemma}
  \label{lemma:clebsch-gordan}
   For all  $[a,b,c] \in I_1$,
     $$ H^0_{\et}(\bar\Delta_{1}[a,b,c], \cH^{\kk,\llll,\mm}) = \Z_p(t).$$
  \end{lemma}
  \begin{proof} The Clebsch-Gordan formula asserts 
   that  the space of tri-homogenous polynomials
  in $6= 2 + 2 + 2$ variables  of tridegree $(\kk,\llll,\mm)$ has a unique $\SL_2$-invariant element, namely, the polynomial
  $$  P_{\kk,\llll,\mm}(x_1,y_1,x_2,y_2,x_3,y_3) = \left|\begin{array}{cc} x_2 & y_2 \\ x_3 & y_3 \end{array} \right|^{\kk'} 
   \left|\begin{array}{cc} x_3 & y_3 \\ x_1 & y_1 \end{array} \right|^{\llll'}  \left|\begin{array}{cc} x_1 & y_1 \\ x_2 &  y_2 \end{array} \right|^{\mm'},$$
   where 
   $$ \kk'  =  \frac{-\kk+\llll +\mm}{2}, \qquad \llll'=\frac{\kk-\llll+\mm}{2}  \qquad \mm'  = \frac{\kk+\llll -\mm}{2}.$$
  
 Since the triplet of weights is balanced, it follows that $\kk', \llll', \mm' \geq 0$.  From the Clebsch-Gordan formula it follows that $H^0_{\et}(\bar\Delta_{1}[a,b,c], \cH^{\kk,\llll,\mm})$ is spanned by the global section whose stalk at a point $((A,P_1),(A,P_2),(A,P_3)) \in \Delta_{1}[a,b,c]$ as in \eqref{eqn:intrinsic-cycle} is given by   
  $$  (X_2 \otimes Y_3 - Y_2\otimes X_3)^{\otimes\kk'} \otimes
  (X_1 \otimes Y_3 - Y_1\otimes X_3)^{\otimes \llll'} \otimes (X_1 \otimes Y_2 - Y_1\otimes X_2)^{\otimes \mm'},$$
   where $(X_i,Y_i)$, $i=1,2,3$, is a basis of the stalk of $\cH$ at the point $(A,P_i)$ in $X_1$. The Galois action is given by the $t$-th power of the cyclotomic character because the Weil pairing takes values in $\Z_p(1)$ and $\kk' + \llll' + \mm' = t$.
\end{proof}

Write $\cl_{\kk,\llll,\mm}(\Delta_1[a,b,c]) \in H^0_{\et}(|\bar\Delta^\circ_1[a,b,c]|,\cH^{\kk,\llll,\mm})$ 
for the   standard generator  given by Lemma \ref{lemma:clebsch-gordan}.
Define
\begin{equation}\label{AJklm}
 \AJ_{\kk,\llll,\mm}(\Delta_1[a,b,c]) \in H^1(\Q(\zeta_1), W_1^{\kk,\llll,\mm})
\end{equation}
to be the extension class constructed by pulling back by $j$ and pushing forward by $p$ in
   the exact sequence of the middle row of the following diagram:
   \begin{equation}
\label{eqn:extension-construction-klm}
 \xymatrix{  & & \Z_p(t)\ar@{^{(}->}[d]^j \\
 H^3_{\et}(\bar X_1^3,\cH^{\kk,\llll,\mm})(2) \ar@{^{(}->}[r] \ar@{>>}[d] & 
H^3_{\et}(\bar X_1^3 \!\!-\!\!\bar\Delta, \cH^{\kk,\llll,\mm})(2) \ar@{->>}[r]
& H^0_{\et}(\bar\Delta,\cH^{\kk,\llll,\mm})    \\
   W_1^{\kk,\llll,\mm}(t),}
 \end{equation}
 where 
 \begin{itemize}
 \item $\Delta=\Delta_1[a,b,c]$;
 \item
   the map $j$ is the $G_{\Q(\zeta_1)}$-equivariant inclusion
 defined by $ j(1) =  \cl_{\kk,\llll,\mm}(\Delta);$
\item 
the surjectivity of the right-most horizontal row follows from the vanishing of the group
$H^4_{\et}(\bar X_1^3,\cH^{\kk,\llll,\mm})$, 
which in turn is a consequence of the K\"unneth formula and the vanishing of the terms
$H^2_{\et}(\bar X_1,\cH^{\kk})$ when $\kk>0$  (cf.\,\cite[Lemmas 2.1, 2.2]{BDP}). 
\end{itemize}
In particular  the image of $j$ lies in the image of  the right-most horizontal row and this
 holds regardless whether the cycle is null-homologous or not. 
The reader may compare this construction with \eqref{eqn:extension-construction},
 where the cycle $\Delta_r^\circ[[a,b,c]]$ is null-homologous and this property was crucially exploited.

 
  \begin{theorem}\label{AJklm-thm} Set $ \AJ_{\kk,\llll,\mm}(\Delta_1^\circ[a,b,c]) = \theta_q  \AJ_{\kk,\llll,\mm}(\Delta_1[a,b,c])$. Then
  the   identity 
   $$\kappa_1(\kk,\llll,\mm)[a,b,c] = \AJ_{\kk,\llll,\mm}(\Delta_1^\circ[a,b,c])$$
   holds in $H^1(\Q(\zeta_1), W_1^{\kk,\llll,\mm})$.
   \end{theorem}
   \begin{proof}
   Set $\Delta:= \Delta_1^\circ[a,b,c]$ in order to alleviate notations.
 Thanks to Lemma \ref{lemma:disposing-of-kings-sheaf}, 
the diagram in 
\eqref{eqn:extension-construction}
used to construct   the extension $E_r$ realising the class $\hkappa_r[a,b,c]$ 
  is {\em the same} as the diagram
\begin{equation}
\label{eqn:extension-construction-bis}
 \xymatrix{ & & &     \quad  \Lambda(G_r)(\mhalf) \ar@{^{(}->}[d] \\
0 \ar[r] & H^3_{\et}(\bar   X_1^3,  \cL_r^{\ast\boxtimes 3})(2) \ar[r] \ar@{>>}[d] & 
H^3_{\et}(\bar   X_1^3-|\bar\Delta|,  \cL_r^{\ast\boxtimes 3})(2) \ar[r]
& H^0_{\et}(|\bar\Delta|,  \cL_r^{\ast\otimes 3})  &  \\
 & H^1_{\et}(\bar X_1,  \cL_r^\ast)^{\otimes 3}(2).}
 \end{equation}
 Let 
 $$\nu_{\kk,\llll,\mm}: \Lambda(G_r) \lra \Z/p^r\Z$$
 be the algebra homomorphism sending the group like element $\langle d_1, d_2, d_3\rangle$ to $d_1^\kk d_2^\llll d_3 ^\mm$, and
observe that the moment maps of \eqref{eqn:mom} allow us to identify
$$
 \cL_r^{\ast\boxtimes 3} \otimes_{\nu_{\kk,\llll,\mm}}( \Z/p^r\Z) = \cH_r^{\kk,\llll,\mm}.$$

 Tensoring \eqref{eqn:extension-construction-bis} over $\Lambda(G_r)$ 
  with $\Z/p^r\Z$ via the map $ \nu_{\kk,\llll,\mm}: \Lambda(G_r)  \lra \Z/p^r\Z$, 
 yields the specialised
 diagram   which coincides exactly with  
the mod $p^r$ reduction of \eqref{eqn:extension-construction-klm}, with $\Delta= \Delta_1^\circ[a,b,c]$. The result follows 
by passing to the limit with $r$.  
\end{proof}

 \begin{cor} 
 \label{cor:special-klm}
 Let 
 \begin{equation}\label{def-cor315}
 \Delta_1^\circ(\omega_1,\omega_2,\omega_3;\delta) :=   \frac{p^3}{(p-1)^3} \, \cdot  \sum_{[a,b,c]\in   I_1}  \delta^{-1}(abc) \omega_1(a) \omega_2(b)\omega_3(c) \Delta_1^\circ[a,b,c].
 \end{equation}
 Then
   $$\spec^\ast_{\kk,\llll,\mm}(\hkappa_\infty(\omega_1,\omega_2,\omega_3;\delta) = \AJ_{\kk,\llll,\mm}(\Delta_1^\circ(\omega_1,\omega_2,\omega_3;\delta)).$$
   \end{cor}
\begin{proof}
This follows directly from the definitions.
\end{proof}

\section{Cristalline specialisations}
\label{sec:cris}

Let $\hf$, $\hg$, $\hh$ be three arbitrary primitive, residually irreducible $p$-adic Hida families of tame levels $M_{f}$, $M_{g}$, $M_{h}$ and tame characters $\chi_{f}$, $\chi_{g}$, $\chi_{h}$, respectively, with associated weight space
$\cW_{\hf} \times \cW_{\hg} \times \cW_{\hg}$. Assume $\chi_{f} \chi_{g} \chi_{h} = 1$ and set $M = \mathrm{lcm}(M_{f},M_{g},M_{h})$.  
Let $(x,y,z)\in \cW_{\hf}\times \cW_{\hg} \times \cW_{\hh}$ 
be a   point lying above a classical triple $(\nu_{\kk,\epsilon_1}, \nu_{\llll,\epsilon_2},
\nu_{\mm,\epsilon_3})\in \cW^3$ of weight space.
As in Definition \ref{tame-cris}, the point $(x,y,z)$ is said to be {\em  tamely ramified } if the three characters $\epsilon_1$, $\epsilon_2$ and 
$\epsilon_3$ are tamely ramified, i.e., factor through the quotient $(\Z/p\Z)^\times$ 
of $\Z_p^\times$, and is said to be {\em crystalline} if  $\epsilon_1\omega^{-\kk} = \epsilon_2 \omega^{-\llll} = \epsilon_3 \omega^{-\mm}=1$.

Fix such a crystalline point $(x,y,z)$ of balanced weight  $(k,\ell,m) = (\kk+2,\llll+2,\mm+2)$, and let 
  $(\hf_x,\hg_y,\hh_z)$  be the specialisations of $(\hf, \hg,\hh)$ at $(x,y,z)$. 
  The ordinariness hypothesis implies that, 
 for all but finitely many exceptions, 
these eigenforms   are the $p$-stabilisations of newforms of level
 dividing $M$, denoted   $f $, $g$ and $h$  respectively:
 $$
\hf_x (q) = f(q)-\beta_ff(q^p), \qquad \hg_y = g(q)-\beta_gg(q^p), \qquad \hh_z(q) = h(q) - \beta_h h(q^p).
$$ 
Since the point $(x,y,z)$ is fixed throughout this section,  the dependency of $(f,g,h)$ on $(x,y,z)$ has been suppressed from
the notations, and we also write
 $(f_\alpha, g_\alpha,h_\alpha) := (\hf_x,\hg_y,\hh_z)$ for the ordinary $p$-stabilisations of $f$, $g$ and $h$.

Recall the quotient  $X_{01}$  of $X_1$, having $\Gamma_0(p)$-level structure at $p$,
and  the projection map  $\mu: X_1 \lra X_{01}$ introduced in \eqref{eqn:def-varpi-12}. 
By an abuse of notation,  the  symbol $\cH^\kk$ is also used to denote
 the \'etale sheaves  appearing in \eqref{eqn:sheaves-of-symmetric-tensors} over any quotient of $X_1$, such as $X_{01}$.
Let
 \begin{eqnarray*}
 W_1   &:=& H^1_{\et}(\bar X_{1}, \cH^\kk) \otimes H^1_{\et}(\bar X_{1}, \cH^\llll) \otimes H^1_{\et}(\bar X_{1}, \cH^\mm)(2-t), \\
W_{01}  &:=& H^1_{\et}(\bar X_{01}, \cH^\kk) \otimes H^1_{\et}(\bar X_{01}, \cH^\llll) \otimes H^1_{\et}(\bar X_{01}, \cH^\mm)(2-t),
\end{eqnarray*}
be the Galois representations arising from the cohomology of  $X_1$ and $X_{01}$ 
 with values in these sheaves.  
  They are endowed
  with a natural action of the triple tensor product  of the Hecke algebras of weight $\kk$, $\llll$, $\mm$ and  level  $Mp$. 
  
  Let $W_1[f_\alpha,g_\alpha,h_\alpha]$  denote the  $(f_\alpha,g_\alpha,h_\alpha)$-isotypic component of $W_1$ on which the Hecke operators  act with the same eigenvalues as
on $f_\alpha\otimes g_\alpha\otimes h_\alpha$. Let $\pi_{f_\alpha,g_\alpha,h_\alpha}: W_1 \ra W_1[f_\alpha,g_\alpha,h_\alpha]$ denote the associated projection. Use similar notations for $W_{01}$.

Recall the family 
\begin{equation}\label{kappakappa}
\hkappa_\infty(\epsilon_1\omega^{-\kk}, \epsilon_2 \omega^{-\llll},\epsilon_3 \omega^{-\mm};1)
= \hkappa_\infty(1,1,1;1)
 \end{equation}
that was introduced in 
 \eqref{eqn:cycles-tame}. By Lemma \ref{lemma:iwasawa-cycles}, this class lies in $H^1(\Q, \Hinfty)$.

Recall the choice of auxiliary prime $q$ made in the definition of the modified diagonal cycle \eqref{mdc}. We assume now that $q$ is chosen so that $C_q := (a_q(f)-q-1) (a_q(g)-q-1) (a_q(h)-q-1)$ is a $p$-adic unit. Note that this is possible because the Galois representations $\varrho_{\hf}$, $\varrho_{\hg}$ and $\varrho_{\hh}$ were assumed to be residually irreducible and hence $f$, $g$ and $h$ are non-Eisenstein mod $p$.
 Let
\begin{equation}\label{kappa1}
\kappa_1(f_\alpha,g_\alpha,h_\alpha) := \frac{1}{C_q} \, \cdot \, \pi_{f_\alpha,g_\alpha,h_\alpha}  \spec^\ast_{x,y,z}\hkappa_\infty(1,1,1;1))  \in H^1(\Q,W_1[f_\alpha,g_\alpha,h_\alpha]) 
\end{equation}
be the specialisation at the
crystalline point $(x,y,z)$ of \eqref{kappakappa}, after projecting it to the $(f_\alpha,g_\alpha,h_\alpha)$-isotypic component of $W_1$ via $\pi_{f_\alpha,g_\alpha,h_\alpha}$. We normalize the class by multiplying it by the above constant in order to remove the dependency on the choice of $q$.

 The main goal of this section is to relate this class  to the  generalised Gross-Schoen diagonal cycles 
that were studied in \cite{DR1}, arising from cycles in Kuga-Sato varieties which are fibered over $X^3$ and have {\em good reduction} at $p$.

The fact that $(x,y,z)$ is a crystalline point implies that the diamond operators in $\Gal(X_1/X_{01})$ act trivially on the $(f_\alpha,g_\alpha,h_\alpha)$-eigencomponents, and hence the
 Hecke-equivariant projection $\mu^3_\ast: W_1 \lra W_{01}$ induces an isomorphism
$$ \mu_\ast^3: W_1[f_\alpha,g_\alpha,h_\alpha] \lra W_{01}[f_\alpha,g_\alpha,h_\alpha].$$

The first aim is to give a geometric description of the class 
$$\kappa_{01}(f_\alpha,g_\alpha,h_\alpha) :=
 \mu_\ast^3 \kappa_1(f_\alpha,g_\alpha,h_\alpha)$$
 in terms of appropriate algebraic cycles.
To this end, recall the cycles $\Delta_1[a,b,c]\in \CH^2(X_1^3)$ introduced in \eqref{Dabc},
and let  $p^\ast := \pm p$ be such that 
$\Q(\sqrt{p^*})$ is the quadratic subfield  of $\Q(\zeta_1)$. 
\begin{lemma}
\label{lemma-klm}
 The    cycle $\mu_\ast^3\Delta_1[a,b,c]$ depends only on the  quadratic residue symbol $(\frac{abc}{p})$ attached to 
   $abc\in(\Z/p\Z)^\times$. 
 The cycles 
  \begin{eqnarray*}
\Delta_{01}^{+} &:=&  \mu^3_\ast \Delta_1[a,b,c] \quad \mbox{ for any } a,b,c \mbox{ with } \left(\frac{abc}{p}\right)=1,    
  \\
\Delta_{01}^- &:=& \mu^3_\ast \Delta_1[a,b,c] \quad \mbox{ for any } a,b,c \mbox{ with } \left(\frac{abc}{p}\right) = -1,
 \end{eqnarray*}
  belong  to $\CH^2(X^3_{01}/\Q(\sqrt{p^*}))$
and   are interchanged by the non-trivial automorphism of $\Q(\sqrt{p^*})$. 
\end{lemma}
\begin{proof}  
 Arguing as in Lemma \ref{lemma:galois-diamond-actions-on-cycles} shows that for
  all $(d_1,d_2,d_3)\in I_1 = (\Z/p\Z)^{\times 3}$,
 \begin{equation*}
\langle d_1,d_2,d_3\rangle \Delta_1[a,b,c]  = \Delta_1[d_2 d_3 a,d_1 d_3 b,d_1 d_2 c].
 \end{equation*}
The orbit of the triple $[a,b,c]$ under the action of $I_1$  is precisely the set of triples $[a',b',c']$ for which 
 $(\frac{a'b'c'}{p}) = (\frac{abc}{p})$.
Since  $X_{01}$ is the quotient of $X_1$ by the group $I_1$,  it follows   that
$\mu^3_{\ast} \Delta_1[a,b,c]$ depends only on  this quadratic residue symbol, and hence
that the classes $\Delta_{01}^+$ and $\Delta_{01}^-$  in the statement of
Lemma \ref{lemma-klm}
are well-defined.
Furthermore,  Lemma \ref{lemma:extend} implies that, for all $m\in (\Z/p\Z)^\times$, 
the Galois automorphism $\sigma_m$   fixes $\Delta_{01}^+$ and $\Delta_{01}^-$  if $m$ is a square modulo $p$,
and interchanges these two cycle classes otherwise. 
It follows that they are  invariant under the Galois group $\Gal(\Q(\zeta_1)/\Q(\sqrt{p^*}))$ and hence descend to a pair of conjugate
cycles $\Delta_{01}^{\pm}$ defined over $\Q(\sqrt{p^*})$, as claimed. 
\end{proof}
 It follows from this lemma that the algebraic cycle 
 \begin{equation}
 \label{kappa-klm}
\Delta_{01} := \Delta_{01}^+   + \Delta_{01}^- \in  \CH^2(X_{01}^3/\Q).
\end{equation}
is defined over $\Q$.
To describe it concretely, note that a triple $(C_1,C_2,C_3)$ of distinct 
cyclic subgroups of
order $p$ in an elliptic curve $A$ admits a somewhat subtle discrete invariant in 
$(\mu_p^{\otimes 3} -\{1\})$ modulo the action of 
$(\Z/p\Z)^{\times 2}$, denoted $\co(C_1,C_2,C_3)$ and called the 
{\em orientation} of $(C_1,C_2,C_3)$. This orientation is defined
 by choosing generators $P_1,P_2,P_3$ of $C_1$, $C_2$ and $C_3$ 
respectively and setting
$$ \co(C_1,C_2,C_3) := \langle P_2,P_3\rangle \otimes \langle P_3,P_1\rangle \otimes \langle P_1,P_2\rangle \in \mu_p^{\otimes 3} - \{ 1\}.$$
It is easy to check that the value of  $\co(C_1,C_2,C_3)$ in $\mu_p^{\otimes 3}-\{1\}$
only depends on the   choices of generators  $P_1$, $P_2$ and $P_3$, up to multiplication by
a {\em non-zero square} in $(\Z/p\Z)^\times$.
In view of \eqref{eqn:intrinsic-cycle}, we then have
\begin{equation}
\label{Delta01}
 \Delta_{01} = \{ ((A,C_1), (A, C_2), (A, C_3) )  \quad \mbox{ with } \quad C_1\ne C_2\ne C_3 \},
 \end{equation}
 and
\begin{eqnarray*}
 \Delta_{01}^+ &=& \left\{ ( (A,C_1), (A,C_2), (A,C_3) ) \quad \mbox{ with } 
\co(C_1,C_2,C_3) =  a \zeta_1^{\otimes 3}, \quad  a\in (\Z/p\Z)^{\times 2} \right\},  \\
\Delta_{01}^- &=& \left\{ ( (A,C_1), (A,C_2), (A,C_3) ) \quad \mbox{ with } 
\co(C_1,C_2,C_3) =  a \zeta_1^{\otimes 3}, \quad  a\notin (\Z/p\Z)^{\times 2} \right\}.
\end{eqnarray*}

Recall the natural projections
$$ \pi_1,\pi_2: X_{01}\lra X, \qquad \varpi_1,\varpi_2: X_1 \lra X$$
to the curve $X=X_0(M)$ of prime to $p$ level, and set
\begin{eqnarray*}
W_0  &:= & H^1_{\et}(\bar X_0, \cH^\kk) \otimes H^1_{\et}(\bar X_0, \cH^\llll) \otimes H^1_{\et}(\bar X_0, \cH^\mm)(2-t),
\end{eqnarray*}
The Galois representation $W_0$ is endowed with a natural action of the triple tensor product  of the Hecke algebras of weight $\kk$, $\llll$, $\mm$ and  level  $M$. 
  Let $W_0[f,g,h]$   denote the $(f,g,h)$- isotypic component of $W_0$, on which the Hecke operators  act with the same eigenvalues as
on $f \otimes g \otimes h$.
Note that the $U_p^\ast$ operator does not act naturally on $W_0$ and hence one cannot speak of the
$(f_\alpha,g_\alpha,h_\alpha)$-eigenspace of this Hecke module. One can, however, 
 denote by $W_1[f,g,h]$ and  $W_{01}[f,g,h]$   the $(f,g,h)$-isotypic component of these Galois representations, 
 in which the action of the  $U_p^\ast$ operators on the three factors are
  not taken into account. Thus, $W_{01}[f_\alpha,g_\alpha,h_\alpha]$ is the image of $W_{01}[f,g,h]$ under the ordinary
 projection,  and likewise for $W_1$.  In other words, denoting by $\pi_{f,g,h}$ the projection to the $(f,g,h)$-isotypic component on any of these modules, one has
 $$ \pi_{f_\alpha,g_\alpha,h_\alpha} = e^\ast \pi_{f,g,h}$$
 whenever the left-hand projection is defined.

The projection maps  
$$(\pi_1,\pi_1,\pi_1): X_{01}^3 \lra X^3,  \quad (\varpi_1,\varpi_1,\varpi_1): X_1^3 \lra X^3$$ 
induces  push-forward maps 
\begin{eqnarray*}
 (\pi_1,\pi_1,\pi_1)_\ast &:& W_{01}[f_\alpha,g_\alpha,h_\alpha] \lra W_0[f,g,h],  \\
  (\varpi_1,\varpi_1,\varpi_1)_\ast &:& W_1[f_\alpha,g_\alpha,h_\alpha] \lra W_0[f,g,h]
  \end{eqnarray*}
on cohomology, 
as well as maps on the associated Galois cohomology groups. 

The goal is now to relate the class 
\begin{equation}\label{kappa101}
(\varpi_1,\varpi_1,\varpi_1)_\ast(\kappa_{1}(f_\alpha,g_\alpha,h_\alpha)) = (\pi_1,\pi_1,\pi_1)_\ast  (\kappa_{01}(f_\alpha,g_\alpha,h_\alpha))
\end{equation} 
to those arising from the  diagonal cycles on the curve $X_0=X$, whose level is prime to $p$.

To do this, it is key  to understand how the maps $\pi_{1\ast}$ and $(\pi_1,\pi_1,\pi_1)_\ast$ interact with the Hecke operators,  especially  with
the ordinary and anti-ordinary projectors $e$ and $e^\ast$, which  do not act naturally on the target of $\pi_{1\ast}$. 
Consider the map 
$$(\pi_1,\pi_2): W_{01}^{\kk} := H^1_{\et}(\bar X_{01}, \cH^{\kk}) \lra W_0^{\kk}:= H^1_{\et}(\bar X_0, \cH^{\kk}).$$
 It is compatible in the obvious
way with the good Hecke operators 
arising from primes $\ell\nmid Mp$, and  therefore induces a map
\begin{equation}
\label{eqn:set-context-f}
 (\pi_1,\pi_2): W_{01}^\kk[f] \lra W_0^\kk[f] \oplus W_0^\kk[f]
 \end{equation}
on the $f$-isotypic components for this Hecke action. As before, note
 that $W_{01}^\kk[f]$ is a priori
larger than $W_{01}^\kk[f_\alpha]$, which is its ordinary quotient.

Let $\xi_f := \chi_f(p) p^{k-1}$ be the determinant of the frobenius at $p$ 
acting on the two-dimensional Galois representation attached to $f$, and likewise for $g$ and $h$.
\begin{lemma}
\label{lemma:from-wiles}
For the map $(\pi_1,\pi_2)$ as in \eqref{eqn:set-context-f},
\begin{eqnarray*}
  \left( \begin{array}{c} \pi_1 \\ \pi_2 \end{array} \right) \circ U_p &=&
\left(\begin{array}{cc} a_p(f) & -1 \\ \xi_f  & 0 \end{array} \right) 
 \left( \begin{array}{c} \pi_1 \\ \pi_2 \end{array} \right),  
\\
 \left( \begin{array}{c} \pi_1 \\ \pi_2 \end{array} \right) \circ U_p^*  &=& 
\left(\begin{array}{cc} 0 & p  \\ -\xi_f p^{-1} & a_p(f) \end{array} \right)  \left( \begin{array}{c} \pi_1 \\ \pi_2 \end{array} \right). 
\end{eqnarray*}
\end{lemma}
\begin{proof} 
The definitions $\pi_1$ and $\pi_2$ imply that, viewing $U_p$ and $U_p^*$ 
(resp.~$T_p$) as correspondences on a Kuga-Sato variety fibered over
 $X_{01}$ (resp.~over  $X_0$), we have
 $$
\begin{aligned}
&\pi_1 U_p = T_p \pi_1 - \pi_2,  \qquad  && \pi_1 U_p^* = p \pi_2    \\
&\pi_2 U_p = p[p] \pi_1  \qquad && \pi_2 U_p^* = -[p] \pi_1 + T_p \pi_2,
\end{aligned} $$
where $[p]$ is the correspondence induced by  the multiplication by $p$ on the fibers and on the prime-to-$p$ part of the level structure.
The result follows by passing to the $f$-isotypic parts, using the fact that $[p]$ induces 
multiplication by $\xi_f p^{-1}$ on this isotypic part.
\end{proof}

For the next calculations, it shall be notationally convenient to introduce the 
notations
$$ \delta_f = \alpha_f-\beta_f, \qquad \delta_g = \alpha_g-\beta_g, \qquad \delta_h = \alpha_h -\beta_h, \qquad \delta_{fgh} = \delta_f \delta_g\delta_h.
$$
\begin{lemma}
\label{eqn:ordinary-proj}
For $(\pi_1,\pi_2)$ as in Lemma \ref{lemma:from-wiles}, 
$$
\begin{aligned}
\pi_1 \circ e &= \frac{\alpha_f \pi_1 - \pi_2}{\delta_f},  \qquad &&
\pi_2 \circ e   =  \frac{\xi_f  \pi_1 - \beta_f\pi_2}{\delta_f}   = \beta_f  \cdot (\pi_1\circ e), \\
\pi_1 \circ e^\ast & = \frac{-\beta_f\pi_1+p\pi_2}{\delta_f}, \qquad
&&
\pi_2 \circ e^\ast =\frac{-\xi_f p^{-1}\pi_1+\alpha_f \pi_2}{\delta_f}  = p\alpha_f^{-1} \cdot(\pi_1\circ e^\ast). 
\end{aligned}
$$
\end{lemma}
\begin{proof}
The matrix identities
\begin{eqnarray*}
    \left(\begin{array}{cc} a_p(f) & -1 \\ \xi_f  & 0 \end{array} \right)  &=& 
\left(\begin{array}{cc} 1 & 1 \\ \beta_f & \alpha_f \end{array} \right)   
\left(\begin{array}{cc} \alpha_f & 0 \\ 0 & \beta_f \end{array}\right) 
\left(\begin{array}{cc} 1 & 1 \\ \beta_f & \alpha_f \end{array} \right)^{-1},
  \\
  \left(\begin{array}{cc} 0 & p  \\ - \xi_f p ^{-1} & a_p(f)  \end{array} \right)  &=& 
\left(\begin{array}{cc} \beta_f & \alpha_f \\  \xi_f p^{-1} & \xi_f p^{-1} \end{array} \right)  
\left(\begin{array}{cc} \alpha_f & 0 \\ 0 & \beta_f \end{array}\right)
\left(\begin{array}{cc} \beta_f & \alpha_f \\ \xi_f p^{-1} & \xi_f(p) p^{-1} \end{array} \right)^{-1},  
\end{eqnarray*}
show that  
\begin{eqnarray*}
\lim \left(\begin{array}{cc} a_p(f) & -1 \\ \xi_f  & 0 \end{array} \right)^{n!} &=&
\left(\begin{array}{cc} 1 & 1 \\ \beta_f & \alpha_f  \end{array}\right)    
\left(\begin{array}{cc} 1 & 0 \\ 0 & 0  \end{array}\right) 
\left(\begin{array}{cc} 1 & 1 \\ \beta_f & \alpha_f  \end{array}\right)^{-1}  \\
&=& \delta_f^{-1} \left(\begin{array}{cc} \alpha_f & -1   \\ \xi_f  & - \beta_f  \end{array} \right),  \\
 \lim \left(\begin{array}{cc} 0 &p \\ -\xi_f  p^{-1}  & a_p(f) \end{array} \right)^{n!} &=&
  \delta_f^{-1}  \left(\begin{array}{cc} -\beta_f & p  \\ -\xi_f p^{-1} &\alpha_f  \end{array} \right),
 \end{eqnarray*}
and the result now follows from   Lemma  \ref{lemma:from-wiles}.
\end{proof}

\begin{lemma}
\label{lemma:any-cycle}
Let $\kappa\in H^1(\Q,W_{01}[f,g,h])$ be any cohomology class with values in the 
 $(f,g,h)$-isotypic subspace of $W_{01}$, and let $e,e^*: H^1(\Q,W_{01}[fgh]) \lra H^1(\Q,W_{01}[f_\alpha,g_\alpha,h_\alpha])$ denote the ordinary and anti-ordinary projections.
Then 
\begin{eqnarray*} 
    (\pi_1, \pi_1,\pi_1)_\ast (e \kappa) &=& \delta_{fgh}^{-1}\times\Big\{  \alpha_f \alpha_g \alpha_h  (\pi_1,\pi_1,\pi_1)_\ast     \\
 & &  \ \  - \alpha_g\alpha_h (\pi_2,\pi_1,\pi_1)_\ast - \alpha_f \alpha_h (\pi_1,\pi_2, \pi_1)_\ast -  \alpha_f\alpha_g (\pi_1,\pi_1,\pi_2)_\ast \\ 
 & & \ \  + \alpha_f (\pi_1, \pi_2,\pi_2)_\ast + \alpha_g(\pi_2,\pi_1,\pi_2)_\ast + \alpha_h  \cdot (\pi_2,\pi_2,\pi_1)_\ast \\
& & \ \    - (\pi_2,\pi_2,\pi_2) _\ast  \Big\}  (\kappa).  \end{eqnarray*}
\begin{eqnarray*}
(\pi_1, \pi_1,\pi_1)_\ast (e^\ast\kappa) &=& \delta_{fgh}^{-1}\times \Big\{  -\beta_f \beta_g \beta_h  (\pi_1,\pi_1,\pi_1)_\ast     \\
 & &  \ \  + p  \beta_g\beta_h (\pi_2,\pi_1,\pi_1)_\ast + p  \beta_f \beta_h (\pi_1,\pi_2, \pi_1)_\ast  +p  \beta_f\beta_g (\pi_1,\pi_1,\pi_2)_\ast \\ 
 & & \ \  - p^{2}  \beta_f (\pi_1, \pi_2,\pi_2)_\ast  - p^{2}  \beta_g(\pi_2,\pi_1,\pi_2)_\ast - p^{2}  \beta_h (\pi_2,\pi_2,\pi_1)_\ast \\
& & \ \     + p^{3} (\pi_2,\pi_2,\pi_2) _\ast  \Big\}  (\kappa),  
\end{eqnarray*}
where we recall that
 $\delta_{fgh}:= ((\alpha_f-\beta_f)(\alpha_g-\beta_g)(\alpha_h-\beta_h))$.
 \end{lemma}
 \begin{proof}
 This follows directly from Lemma \ref{eqn:ordinary-proj}.  
 \end{proof}

Recall the notations
$$ \kk := k-2, \qquad \llll  = \ell-2, \qquad \mm := m-2,  \qquad  r:= (\kk+\llll+\mm)/2.$$

Let $\cA$ denote the Kuga-Sato variety over $X$ as introduced in \ref{subsec:forms}.  In \cite[Definitions 3.1,3.2 and 3.3]{DR1},
   a generalized diagonal cycle $$\Delta^{\kk,\llll,\mm} = \Delta_0^{\kk,\llll,\mm} \, \in \, \CH^{r+2}(\cA^{\kk}\times \cA^{\llll}\times \cA^{\mm},\Q)$$
   is associated to the triple $(\kk,\llll,\mm)$.

When $\kk, \llll, \mm >0$,  $\Delta^{\kk,\llll,\mm} $ is 
 obtained by choosing subsets
 $A$, $B$ and $C$ of the set $\{1,\ldots, r\}$ which satisfy:
 $$ \# A = \kk, \qquad \#B =\llll,  \qquad \#C = \mm,  \quad \qquad A\cap B \cap C = \emptyset,
 $$
 $$  \#(B\cap C) = r-\kk, \qquad  \#(A\cap C) = 
r-\llll, \qquad   \#(A\cap B) = r-\mm.$$
 The cycle $\Delta^{\kk,\llll,\mm} $ is defined as the image of the
 embedding $\cA^{r}$ into $\cA^{\kk}\times \cA^{\llll} \times \cA^{\mm}$ given by sending 
 $(E,(P_1,\ldots,P_r))$ to $( (E,P_A), (E,P_B), (E,P_C))$, where for instance 
 $P_A$ is a shorthand for the $\kk$-tuple of points $P_j$ with $j\in A$.

Let also $\Delta_{01}^{\kk,\llll,\mm} \in \CH^{r+2}(\cA^{\kk}\times \cA^{\llll}\times \cA^{\mm}) $ 
denote the generalised diagonal cycle in the product of the three Kuga-Sato varieties 
of weights $(k,\ell,m)$ fibered over $X_{01}$, defined in a similar way as in \eqref{Delta01} and along the same lines as recalled above. 

More precisely, $\Delta_{01}^{\kk,\llll,\mm}$ is defined as the schematic closure in $\cA^{\kk}\times \cA^{\llll}\times \cA^{\mm}$ of the set of tuples $((E,C_1,P_A),(E,C_2,P_B),(E,C_3,P_C))$ where $P_A, P_B, P_C$ are as above, and $C_1, C_2, C_3$ is a triple of pairwise distinct subgroups of order $p$ in the elliptic curve $E$.

Since the triple $(\kk,\llll,\mm)$ is fixed throughout this section, in order to alleviate notations in the statements below we shall simply denote $\Delta^\sharp$ and $\Delta_{01}^\sharp$ for $\Delta^{\kk,\llll,\mm}$ and $\Delta_{01}^{\kk,\llll,\mm}$ respectively.
 
 \begin{lemma} 
\label{lemma:projecting-diagonal} 
The following identities hold in $\CH^{r+2}(\cA^{\kk}\times \cA^{\llll}\times \cA^{\mm})$:
\begin{eqnarray*}
 (\pi_1,\pi_1,\pi_1)_\ast (\Delta^\sharp_{01}) &=& {(p+1)p(p-1)}(\Delta^\sharp), \\
 (\pi_2,\pi_1,\pi_1)_\ast (\Delta^\sharp_{01}) &=& {p(p-1)} \times (T_p,1,1) (\Delta^\sharp), \\
  (\pi_1,\pi_2,\pi_1)_\ast (\Delta^\sharp_{01}) &=& {p(p-1)} \times (1,T_p,1) (\Delta^\sharp), \\
  (\pi_1,\pi_1,\pi_2)_\ast (\Delta^\sharp_{01}) &=& {p(p-1)} \times (1,1, T_p) (\Delta^\sharp), \\
   (\pi_1,\pi_2,\pi_2)_\ast (\Delta^\sharp_{01}) &=& {(p-1)} \times ( (1,T_p,T_p)(\Delta^\sharp) - p^{r-\kk} D_1)\\
  (\pi_2,\pi_1,\pi_2)_\ast (\Delta^\sharp_{01}) &=& {(p-1)} \times ( ( T_p,1, T_p)(\Delta^\sharp) -  p^{r-\llll} D_2)\\
  (\pi_2,\pi_2,\pi_1)_\ast (\Delta^\sharp_{01}) &=&  {(p-1)} \times ( ( T_p,  T_p, 1)(\Delta^\sharp) - p^{r-\mm}D_3)\\  
   (\pi_2,\pi_2,\pi_2)_\ast (\Delta^\sharp_{01}) &=&   (T_p,T_p,T_p)(\Delta^\sharp)  -  p^{r-\kk} E_1 -  p^{r-\llll} E_2 - p^{r-\mm} E_3   \\ 
   & & -  p^r (p+1) \Delta^\sharp, 
   \end{eqnarray*}
   where the cycles $D_1$, $D_2$ and $D_3$ satisfy
   \begin{eqnarray*}
 ([p],1,1)_\ast  (D_1) = p^{\kk} (T_p,1,1)_\ast (\Delta^\sharp),  &&  (1,[p],1)_\ast (D_2) = p^{\llll}(1,T_p,1)_\ast(\Delta^\sharp),  \\
  (1,1,[p])_\ast (D_3) = p^{\mm}(1,1,T_p) (\Delta^\sharp),  &&
  \end{eqnarray*}
   the cycles $E_1$, $E_2$ and $E_3$ satisfy
\begin{eqnarray*}
([p],1,1)_\ast (E_1)   = p^{\kk}(T_{p^2},1,1) (\Delta^\sharp), 
   & & 
    (1,[p],1)_\ast (E_2) = p^{\llll}(1,T_{p^2},1)(\Delta^\sharp), \\
 (1,1, [p])_\ast (E_3) = p^{\mm}(1,1,T_{p^2})(\Delta^\sharp), & &
  \end{eqnarray*}
   and 
    $T_{p^2} := T_p^2 -  (p+1)[p]$. 
   \end{lemma}
\begin{proof}
 The first four identities  are straightforward: the map
  $ \pi_1\times \pi_1\times\pi_1$ induces a finite map from
   $\Delta^\sharp_{01}$ to $\Delta^\sharp$ of degree $(p+1)p(p-1)$, which is 
   the number of possible choices of an ordered triple of distinct 
   subgroups of order $p$ on an elliptic curve, and likewise 
   $\pi_2\times\pi_1\times \pi_1$ induces a map of degree $p(p-1)$ from $\Delta^\sharp_{01}$ 
   to $(T_p,1,1) \Delta^\sharp$.
The remaining identities follow from combinatorial reasonings based on the explicit
 description of the  cycles $\Delta^\sharp_{01}$ and $\Delta^\sharp$. 
For the $5$th identity, observe that  $ (\pi_1,\pi_2,\pi_2)_\ast$
induces a degree $(p-1)$ map from $ \Delta^\sharp_{01}$ to the variety $X$ 
parametrising triples $((E,P_A), (E', P_B'), (E'',P_C''))$ for  which there are distinct 
cyclic $p$-isogenies
$\varphi':E \lra E'$ and $\varphi': E\lra E''$,  
the system of points   $P_B'\subset E'$ and $P_C''\subset E''$ 
indexed by the sets $B$ and $C$
satisfy
$$  \varphi'(P_{A\cap B}) = P_{A\cap B}', \quad
 \varphi''(P_{A\cap C}) =  P_{A\cap C}'',$$
 and for which there exists points $Q_{B\cap C}\subset E$ indexed by $B\cap C$ satisfying
 $$ \varphi'(Q_{B\cap C}) = P'_{B\cap C}, \qquad \varphi''(Q_{B\cap C}) = P''_{B\cap C}.$$
On the other hand, $(1,T_p,T_p)$ parametrises triples of the   same type, in
 which $E'$ and $E''$ need not be distinct. It follows that
\begin{equation}
\label{eqn:prr1}
 (1,T_p,T_p) (\Delta^\sharp)   = X+ p^{r-\kk} D_1,
 \end{equation}
where  the closed points of $D_1$ correspond to   triples  
 $((E,P_A), (E',P_B'), (E', P_C'))$ for  which there is a   cyclic $p$-isogeny $\varphi':E \lra E'$ satisfying 
 $$  \varphi'(P_{A\cap B}) = P_{A\cap B}', \qquad
 \varphi'(P_{A\cap C}) =  P_{A\cap C}'.$$
 The coefficient of $p^{r-\kk}$ in \eqref{eqn:prr1}
  arises because each closed point of $D_1$ comes from 
 $p^{\#(B\cap C)}$ distinct closed points on $(1,T_p,T_p)(\Delta^\sharp)$, obtained by translating the points $P_j \in P_{B\cap C}$  with 
 $j\in B\cap C$ by arbitrary elements of $\ker(\varphi)$.
 The fifth identity now follows after noting that the map $([p],1,1])$ induces a  map
 of degree $p^{\kk}$ from $D_1$ to $(T_p,1,1)_\ast\Delta^\sharp$. 
  The $6$th and $7$th identity can be treated with  an  identical reasoning by  interchanging  the three 
  factors in $W^{\kk} \times W^{\llll} \times W^{\mm}$. 
As for the last identity,
the map  $(\pi_2,\pi_2,\pi_2)$ induces a map of degree $1$ to the variety $Y $ consisting of triples 
$(E',E'',E''')$ of elliptic curves which are $p$-isogenous to a common elliptic curve $E$ and distinct.
But it is not hard to see that
$$ (T_p,T_p,T_p) (\Delta^\sharp) =  Y + p^{r-\kk} E_1 + p^{r-\llll} E_2 + p^{r-\mm} E_3 + p^r (p+1) \Delta^\sharp$$
where the additional terms on the right hand side account for triples $(E',E'',E''')$ where
$E'\ne E''=E'''$, where $ E'' \ne    E' = E'''$, where $E'''\ne E' = E''$, and where
$E'=E''=E'''$ respectively.
 \end{proof}
 
Assume for the remainder of the section that  $\kk,\llll,\mm>0$. Recall the projectors $\epsilon_{\kk}$ of \eqref{idempotent}. It was shown in \cite[\S 3.1]{DR1} that  $(\epsilon_{\kk},\epsilon_{\llll},\epsilon_{\mm}) \Delta^{\kk,\llll,\mm}$ is a null-homologous cycle and we may define
\begin{equation}
\label{cycleDR1}
\kappa(f,g,h) := \pi_{f,g,h}\, \AJ_{\et}((\epsilon_{\kk},\epsilon_{\llll},\epsilon_{\mm}) \Delta^{\kk,\llll,\mm})  \in H^1(\Q,W_0[f,g,h])
\end{equation}
as the image of this cycle under the $p$-adic \'etale Abel-Jacobi map, followed by the natural projection from $H^{2c-1}_{\et}(\bar\cA^{\kk}\times \bar\cA^{\llll}\times \bar\cA^{\mm},\Q_p(c))$ to $W_0^{\kk,\llll,\mm}$ induced by the K\"unneth decomposition and  the projection  $\pi_{f,g,h}$.

 It follows from \cite[(66)]{DR1}, \eqref{idempotent} and the vanishing of the terms
$H^i_{\et}(\bar X_1,\cH^{\kk})$ for $i\ne 1$ when $\kk>0$, that the class $\kappa(f,g,h)$ is realized by the $(f,g,h)$-isotypic component of the same extension class as in \eqref{eqn:extension-construction-klm}, after replacing  $X_1$  by the curve $X=X_0$ and $\Delta=\Delta^{0,0,0}$ is taken to be the usual diagonal cycle in $X^3$. In the notations of \eqref{AJklm}, this amounts to 
\begin{equation}\label{DR1bis}
\kappa(f,g,h) = \pi_{f,g,h}\AJ_{\kk,\llll,\mm}(\Delta).
\end{equation}

Similar statements holds over the curve $X_{01}$. Namely, we also have the following:

\begin{proposition}  The cycle $(\epsilon_{\kk},\epsilon_{\llll},\epsilon_{\mm}) \Delta_{01}^{\kk,\llll,\mm}$ is null-homologous and the following equality of classes holds in $H^1(\Q, W_{01}[f_\alpha,g_\alpha,h_\alpha])$: 
\begin{equation}\label{cycleDR01}
\kappa_{01}(f_\alpha,g_\alpha,h_\alpha) = p^3 \, \cdot \, \pi_{f_\alpha,g_\alpha,h_\alpha}\, \AJ_{\et}((\epsilon_{\kk},\epsilon_{\llll},\epsilon_{\mm}) \Delta_{01}^{\kk,\llll,\mm}).
\end{equation}
\end{proposition}

\begin{proof} Corollary \ref{cor:special-klm} together with \eqref{kappa1} imply that
  $$\kappa_1(f_\alpha,g_\alpha,h_\alpha) =  \frac{1}{C_q}\, \cdot \pi_{f_\alpha,g_\alpha,h_\alpha}\, \AJ_{\kk,\llll,\mm}(\Delta_1^\circ(1,1,1;\delta)),$$
  in which $\delta=1$ is the trivial character of $(\Z/p\Z)^\times$. Since
 $\mu^3$ induces a finite map of degree $(p-1)^3$ from the support of $\Delta_1(1,1,1;\delta)$ to $\Delta_{01}$, it follows from the convention adopted in \eqref{def-cor315} that
 $$
 \kappa_{01}(f_\alpha,g_\alpha,h_\alpha) := \mu^3_\ast \,\kappa_1(f_\alpha,g_\alpha,h_\alpha) =   \frac{p^3}{C_q}\, \cdot \pi_{f_\alpha,g_\alpha,h_\alpha}\, \AJ_{\kk,\llll,\mm}(\Delta_{01}^\circ),
 $$
where  $\AJ_{\kk,\llll,\mm}(\Delta_{01}^\circ)$ is defined to be the class realized by the same extension class as in \eqref{eqn:extension-construction-klm}, after replacing  $X_1$  by the curve $X_{01}$ and replacing $\Delta$ by the cycle $\Delta^\circ_{01}$ arising from \eqref{Delta01}. Since $\Delta_{01}^{\kk,\llll,\mm}$ is fibered over $\Delta_{01}$, the same argument as in \eqref{DR1bis} then shows that
$$
 \AJ_{\kk,\llll,\mm}(\Delta_{01}) =  \AJ_{\et}((\epsilon_{\kk},\epsilon_{\llll},\epsilon_{\mm}) \Delta_{01}^{\kk,\llll,\mm}).
$$
Since $\pi_{f_\alpha,g_\alpha,h_\alpha}(\Delta_{01}) = \frac{1}{C_q}\pi_{f_\alpha,g_\alpha,h_\alpha}(\Delta_{01}^\circ)$, the proposition follows.
\end{proof}

 \begin{theorem}
 \label{prop:main-crystalline}
With notations as before, letting $c= r+2$, we have
 \begin{equation*}
 (\varpi_1,\varpi_1,\varpi_1)_\ast \,\, \kappa_1(f_\alpha,g_\alpha,h_\alpha)  =  
 \frac{ \cE^{\mathrm{bal}}(f_\alpha,g_\alpha,h_\alpha) }{ \cE(f_\alpha) \cE(g_\alpha) \cE(h_\alpha)}
  \times \kappa(f,g,h),
  \end{equation*}
  where 
 $$
 \cE^{\mathrm{bal}}(f_\alpha,g_\alpha,h_\alpha) =   (1-\alpha_f\beta_g\beta_h p^{-c}) (1-\beta_f\alpha_g\beta_h p^{-c})
  (1-\beta_f\beta_g\alpha_h p^{-c})
 (1-\beta_f\beta_g\beta_h p^{ -c}), $$
 and
 $$
 \cE(f_\alpha) =  1-\chi_f^{-1}(p) \beta_f^2 p^{1-k}, \quad
 \cE(g_\alpha) = 1- \chi_g^{-1}(p) \beta_g^2 p^{1-\ell}, \quad
 \cE(h_\alpha) = 1- \chi_h^{-1}(p) \beta_h^2 p^{1-m}.
$$
 \end{theorem}

 \begin{proof} In view of \eqref{kappa101}, \eqref{cycleDR1} and \eqref{cycleDR01}, it suffices to prove the claim for the cycles $\Delta^{\kk,\llll,\mm}$ and $(\pi_1,\pi_1,\pi_1)_\ast \,\, e^\ast \Delta_{01}^{\kk,\llll,\mm}$. Since $\kk,\llll,\mm$ are fixed throughout the discussion, we again denote $\Delta^\sharp=\Delta^{\kk,\llll,\mm}$ and $\Delta_{01}^\sharp=\Delta_{01}^{\kk,\llll,\mm}$ to lighten notations.
 
 When combined with Lemma \ref{lemma:any-cycle}, 
 Lemma \ref{lemma:projecting-diagonal}
 equips us with a completely explicit formula for comparing $(\pi_1,\pi_1,\pi_1)_\ast e^\ast(\Delta_{01}^\sharp)$ 
 with the generalised diagonal cycle $\Delta^\sharp$. 
 Namely, since the correspondences $([p],1,1)$, $(1,[p],1)$ and $(1,1,[p])$ induce multiplication by 
 $p^{\kk}$, $p^{\llll}$ and $p^{\mm}$ respectively on the $(f,g,h)$-isotypic parts, while
 $(T_p,1,1)$, $(1,T_p,1)$, and $(1,1,T_p)$ induce multiplication by
 $a_p(f)$, $a_p(g)$, and $a_p(h)$ respectively, it follows that, with notations as in
 the proof of Lemma \ref{lemma:projecting-diagonal},
\begin{eqnarray*}
 \pi_{f,g,h} (D_1) &=& a_p(f) \pi_{f,g,h}(\Delta^\sharp), \\
 \pi_{f,g,h}(D_2) &=& a_p(g) \pi_{f,g,h}(\Delta^\sharp),  \\
  \pi_{f,g,h} (D_3) &=& a_p(h) \pi_{f,g,h}(\Delta^\sharp),
\end{eqnarray*}
    and  that
  \begin{eqnarray*}
\pi_{f,g,h} (E_1)   &=& (a_p^2(f) - (p+1)p^{\kk}) \pi_{f,g,h}(\Delta^\sharp), \\
 \pi_{f,g,h} (E_2) &=&  (a_p^2(g) - (p+1)p^{\llll}) \pi_{f,g,h}(\Delta^\sharp),  \\
 \pi_{f,g,h} (E_3) &=& (a_p^2(h) - (p+1)p^{\mm}) \pi_{f,g,h}(\Delta^\sharp).
  \end{eqnarray*}
  By projecting  the various formulae for
   $ (\pi_1,\pi_1,\pi_1)_\ast (\Delta^\sharp_{01}) $  that are given in Lemma 
   \ref{lemma:projecting-diagonal} to the $(f,g,h)$-isotypic component  and substituting them into Lemma
   \ref{lemma:any-cycle}, one obtains a expression for 
   $e_{f,g,h}(\pi_1,\pi_1,\pi_1)_\ast e^* (\Delta^\sharp_{01})$ as a multiple of 
   $\pi_{f,g,h}(\Delta^\sharp)$ by an explicit factor, which is a 
   rational function in   $\alpha_f$, 
   $\alpha_g$ and $\alpha_h$. This 
   explicit factor
   is somewhat tedious to calculate by hand, but  the identity asserted in 
   Theorem  \ref{prop:main-crystalline}
   is readily  checked with the help of a symbolic algebra package.
 \end{proof}

\section{Triple product $p$-adic $L$-functions and the reciprocity law}
\label{3.1}

Let $(\hf,\hg,\hh)$ be a triple of $p$-adic Hida families  of tame levels $M_{f}$, $M_{g}$, $M_{h}$ and tame characters $\chi_{f}$, $\chi_{g}$, $\chi_{h}$ as in the previous section. Let also $(\hf^*,\hg^*,\hh^*) = (\hf\otimes \bar\chi_f,\hg\otimes \bar\chi_g,\hh\otimes \bar\chi_h)$ denote the conjugate triple. As before, we assume $\chi_{f} \chi_{g} \chi_{h} = 1$ and set $M = \mathrm{lcm}(M_{f},M_{g},M_{h})$.

Let $\Lambda_{\hf}$, $\Lambda_{\hg}$ and $\Lambda_{\hh}$ be the finite flat extensions of $\Lambda$ generated by the coefficients of the Hida families $\hf$, $\hg$ and $\hh$, and set $\Lambda_{\hf \hg \hh} = \Lambda_{\hf} \hat\otimes_{\Z_p}  \Lambda_{\hg} \hat\otimes_{\Z_p}  \Lambda_{\hh}$. Let also $\mathcal{Q}_{\hf}$ denote the fraction field of $\Lambda_{\hf}$ and define
$$
\mathcal{Q}_{\hf,\hg \hh} := \mathcal{Q}_{\hf} \hat\otimes \Lambda_{\hg} \hat\otimes \Lambda_{\hh}.
$$

Let $\cW_{\hf \hg \hh}^{\circ} := \cW_{\hf}^{\circ} \times \cW_{\hg}^{\circ} \times \cW_{\hh}^{\circ} \subset \cW_{\hf \hg \hh} = \mathrm{Spf}(\Lambda_{\hf \hg \hh})$ denote the set of triples of {\em crystalline} classical  points, at which the three Hida families specialize to modular forms with trivial nebentype at $p$ (and may be either old or new at $p$). This set admits a natural partition, namely
$$
\cW_{\hf \hg \hh}^{\circ} =   \cW_{\hf \hg \hh}^{f} \, \sqcup \, \cW_{\hf \hg \hh}^{g} \,\sqcup \, \cW_{\hf \hg \hh}^{h} \,
 \sqcup \, \cW_{\hf \hg \hh}^{\mathrm{bal}}
$$
where 
\begin{itemize}
\item  $\cW_{\hf \hg \hh}^{f}$ denotes the set of points $(x,y,z)\in \cW_{\hf \hg \hh}^{\circ} $ of weights $(k,\ell,m)$ such that $k\geq \ell+m$. 

\item  $\cW_{\hf \hg \hh}^{g}$ and $\cW_{\hf \hg \hh}^{h}$ are defined similarly, replacing the role of $\hf$ with $\hg$ (resp.\,$\hh$).

\item $\cW_{\hf \hg \hh}^{\mathrm{bal}}$ is the set of {\em balanced} triples, consisting of points $(x,y,z)$ of weights $(k,\ell,m)$ 
such that each of the weights is strictly smaller than the sum of the other two.

\end{itemize}

Each of the four subsets appearing in the above partition is dense in $\cW_{\hf \hg \hh}$ for the rigid-analytic topology. 

Recall from \eqref{eqn:hf-isotypic} the spaces of $\Lambda$-adic test vectors $S^{\ordi}_{\Lambda}(M,\chi_f)[\hf]$. 
For any choice of a triple $$(\htf,\htg,\hth)\in S^{\ordi}_{\Lambda}(M,\chi_{f})[\hf] \times S^{\ordi}_{\Lambda}(M,\chi_{g})[\hg] \times S^{\ordi}_{\Lambda}(M,\chi_{h})[\hh]$$ of $\Lambda$-adic test vectors of tame level $M$, in \cite[Lemma 2.19 and Definition 4.4]{DR1} we constructed a $p$-adic $L$-function $\Lp^f(\htf,\htg,\hth)$ in $\mathcal{Q}_{\hf} \hat\otimes \Lambda_{\hg} \hat\otimes \Lambda_{\hh}$, giving rise to a meromorphic rigid-analytic function
\begin{equation}\label{padicL}
\Lp^f(\htf,\htg,\hth): \cW_{\hf \hg \hh} \, \lra \, \C_p.
\end{equation}



As shown in \cite[\S 4]{DR1}, this $p$-adic $L$-function is characterized by an interpolation property relating its values at classical points $(x,y,z)\in \cW_{\hf \hg \hh}^f$  to the central critical value of Garrett's triple-product complex $L$-function $L(\hf_x,\hg_y,\hh_z,s)$ associated to the triple of classical  eigenforms $(\hf_x,\hg_{y},\hh_z)$. The fudge factors appearing in the interpolation property depend heavily on the choice of test vectors: cf.\,\cite[\S 4]{DR1} and \cite[\S 2]{DLR} for more details. In a recent preprint, Hsieh \cite{Hs} has found an  explicit choice of test vectors, which yields a very optimal interpolation formula which shall be   useful for our purposes. We describe it below:

\begin{proposition}\label{padicL-Garrett} 
for every $(x,y,z)\in \cW_{\hf \hg \hh}^{f}$ of weights $(k,\ell,m)$ we have
\begin{equation}
\label{ggg}
\Lp^f(\htf,\htg,\hth)^2(x,y,z) =   \frac{ \mathfrak{a}(k,\ell,m)}{\langle \hf^\circ_x,\hf^\circ_x\rangle^2} \cdot \mathfrak{e}^2(x,y,z)\cdot \prod_{v\mid N\infty} C_v  \times L(\hf^\circ_x, \hg^\circ_y, \hh^\circ_z,c)
\end{equation}
where
\begin{enumerate}
\item[i)] $c=\frac{k+\ell+m-2}{2}$,

\vspace{0.1cm}
\item[ii)] $\mathfrak{a}(k,\ell,m)= (2\pi i)^{-2k}  \cdot (\frac{k+\ell+m-4}{2})! \cdot (\frac{k+\ell-m-2}{2})! \cdot (\frac{k-\ell+m-2}{2})!\cdot (\frac{k-\ell-m}{2})!$,


\vspace{0.1cm}
\item[iii)]  $\mathfrak{e}(x,y,z) =\cE(x,y,z)/ \cE_0(x)\cE_1(x)$ with
 \begin{eqnarray*}
\label{eqn:defE0}
\cE_0(x) & := & 1- \chi_f^{-1}(p)  \beta_{\hf_x}^2  p^{1-k},
\\
\label{eqn:defE1}
\cE_1(x) & := & 1- \chi_f(p) \alpha^{-2}_{\hf_x} p^{\kk},
\\
\label{eqn:defEE}
\cE(x,y,z) & := &
\left(1-  \chi_f(p) \alpha^{-1}_{\hf_x} \alpha_{\hg_y}\alpha_{\hh_z}p^{\frac{k-\ell-m}{2}} \right)\times
\left(1- \chi_f(p) \alpha^{-1}_{\hf_x}  \alpha_{\hg_y} \beta_{\hh_z}  p^{\frac{k-\ell-m}{2}} \right) \\
\nonumber
& & \times \left(1- \chi_f(p) \alpha^{-1}_{\hf_x} \beta_{\hg_y} \alpha_{\hh_z}  p^{\frac{k-\ell-m}{2}}
\right)\times
\left(1-  \chi_f(p)  \alpha^{-1}_{\hf_x} \beta_{\hg_y} \beta_{\hh_z} p^{\frac{k-\ell-m}{2}} \right).
\end{eqnarray*}

\item[iv)] The local constant $C_v\in \Q(\hf_x,\hg_y,\hh_z)$ depends only on the admissible representations of $\GL_2(\Q_v)$ associated to $(\hf_x,\hg_y,\hh_z)$ and on the local components
at $v$ of the test vectors.

\end{enumerate}

Moreover, there exists a distinguished choice of test vectors $(\htf,\htg,\hth)$ (as specified by Hsieh in \cite[\S 3]{Hs}) for which $\Lp^f(\htf,\htg,\hth)$ lies in $\Lambda_{\hf \hg \hh}$ and the local constants  may be taken to be $C_v=1$ at all $v\mid N \infty$.
\end{proposition}


\begin{proof} This follows from \cite[Theorem A]{Hs}, after spelling out explicitly the definitions involved in Hsieh's formulation. 

Let us remark that throughout the whole article \cite{DR1}, it was implicitly assumed that $\hf_x$, $\hg_\ell$ and $\hh_m$ are all old at $p$, and note  that the definition we have given here of the terms $\cE_0(x)$, $\cE_1(x)$ and $\cE(x,y,z)$  is exactly the same as in \cite{DR1} in such cases, because $\beta_{\hf_x} =  \chi_f(p)\alpha_{\hf_x}^{-1} p^{k-1}$ when $\hf_x$ is old at $p$. 

In contrast with loc.\,cit.,\,in the above proposition we also allow any of the eigenforms $\hf_x$, $\hg_\ell$ and $\hh_m$ to be new at $p$ (which can only occur when the weight is $2$); in such case, recall the usual convention adopted in \S \ref{subsec:forms} to set $\beta_{\phi}=0$ when $p$ divides the primitive level of an eigenform $\phi$. With these notations,  the current formulation of $\cE(x,y,z)$, $\cE_0(x)$ and $\cE_1(x)$ is the correct one, as one can readily verify by rewriting the proof of \cite[Lemma 4.10]{DR1}.
\end{proof}

\subsection{Perrin-Riou's regulator}
\label{4.2}

Recall the $\Lambda$-adic cyclotomic character $\underline{\varepsilon}_{\cyc}$ and the unramified characters $\Psi_{\hf}$, $\Psi_{\hg}$, $\Psi_\hh$ of $G_{\Q_p}$ introduced in Theorem \ref{Vphi-thm}. As a piece of notation, let $\underline{\varepsilon}_{\hf}: G_{\Q_p} \lra \Lambda_{\hf}^\times$ denote the composition of $\underline{\varepsilon}_{\cyc}$ and the natural inclusion $\Lambda^\times \subset \Lambda_{\hf}^\times$, and likewise for $\underline{\varepsilon}_{\hg}$ and $\underline{\varepsilon}_{\hh}$. Expressions like $\Psi_{\hf} \Psi_{\hg}\Psi_{\hh}$ or $\underline{\varepsilon}_{\hf} \underline{\varepsilon}_{\hg} \underline{\varepsilon}_{\hh}$ are a short-hand notation for the $\Lambda_{\hf \hg \hh}^\times$-valued character of $G_{\Q_p}$ given by the tensor product of the three characters. 

Let $ \mathbb{V}_{\hf}$, $\mathbb{V}_{\hg}$ and $\mathbb{V}_{\hh}$ be the Galois representations associated to $\hf$, $\hg$ and $\hh$ in Theorem \ref{Vphi-thm}.

The purpose of this section is describing in precise terms the close connection between the diagonal cycles constructed above and the three-variable triple-product $p$-adic $L$-function. In order to do that, let us introduce the $\Lambda_{\hf \hg \hh}$-modules
\begin{equation}
\mathbb{V}^\dag_{\hf \hg \hh} :=  \mathbb{V}_{\hf}\otimes \mathbb{V}_{\hg}\otimes \mathbb{V}_{\hh}(-1)(\half) =  \mathbb{V}_{\hf}\otimes \mathbb{V}_{\hg}\otimes \mathbb{V}_{\hh}(\varepsilon_{\cyc}^{-1} \underline{\varepsilon}_{\hf}^{-1/2}\underline{\varepsilon}_{\hg}^{-1/2} \underline{\varepsilon}_{\hh}^{-1/2}).
\end{equation}
and
\begin{equation}
\mathbb{V}^\dag_{\hf \hg \hh}(M) :=  \mathbb{V}_{\hf}(M)\otimes \mathbb{V}_{\hg}(M)\otimes \mathbb{V}_{\hh}(M)(-1)(\half).
\end{equation}

The pairing defined in \eqref{eqn:pairing-LambdaG} yields an identification $\Hinfty = H^1_{\et}(\bar X_\infty, \Z_p)^{\otimes 3}(2)(\half)$.
As explained in  \eqref{Vphi-hida*}, $\mathbb{V}^\dag_{\hf \hg \hh}(M)$ is isomorphic to the direct sum of several copies of $\mathbb{V}^\dag_{\hf \hg \hh}$ and there are canonical projections $\varpi_{\hf}$, $\varpi_\hg$, $\varpi_{\hh}$ which assemble into a $G_\Q$-equivariant map
$$
\varpi_{\hf, \hg, \hh}:  \Hinfty = H^1_{\et}(\bar X_\infty, \Z_p)^{\otimes 3}(2)(\half) \lra \mathbb{V}^\dag_{\hf \hg \hh}(M).
$$

Recall the three-variable $\Lambda$-adic global cohomology class
$$
\hkappa_\infty(\epsilon_1\omega^{-\kk}, \epsilon_2 \omega^{-\llll},\epsilon_3 \omega^{-\mm};1)
= \hkappa_\infty(1,1,1;1) \in       H^1(\Q,\Hinfty)
$$
introduced in \eqref{kappakappa}.

Set $C_q(\hf,\hg,\hh) := (a_q(\hf)-q-1)(a_q(\hg)-q-1)(a_q(\hh)-q-1)$. Note that $C_q(\hf,\hg,\hh)$ is a unit in $\Lambda_{\hf \hg \hh}$, because its classical specializations are $p$-adic units (cf.\,\eqref{kappa1}). 

\begin{definition}\label{def-family}
 Define
\begin{equation*}
\hkappa(\hf,\hg,\hh) :=  \frac{1}{C_q(\hf,\hg,\hh)}\, \cdot \, \varpi_{\hf, \hg, \hh *}(\hkappa_\infty(\epsilon_1\omega^{-\kk}, \epsilon_2 \omega^{-\llll},\epsilon_3 \omega^{-\mm};1)) \in  H^1(\Q,\mathbb{V}^\dag_{\hf \hg \hh}(M))
\end{equation*}
to be the projection of the above class to the $(\hf, \hg, \hh)$-isotypical component.  
\end{definition}

In the above definition, we normalize $\hkappa(\hf,\hg,\hh) $ by the constant $C_q(\hf,\hg,\hh)$ so that the classical specializations of $\hkappa(\hf,\hg,\hh)$ at classical points coincide with the classes $\kappa_1(f_\alpha,g_\alpha,h_\alpha)$ introduced in \eqref{kappa1}.

Let $$\mathrm{res}_p: H^1(\Q,\mathbb{V}^\dag_{\hf \hg \hh}(M)) \ra H^1(\Q_p,\mathbb{V}^\dag_{\hf \hg \hh}(M))$$ denote the restriction map 
to the local cohomology at $p$ and set
$$\hkappa_p(\hf,\hg,\hh):=  \mathrm{res}_p(\hkappa(\hf,\hg,\hh)) \in H^1(\Q_p,\mathbb{V}^\dag_{\hf \hg \hh}(M)).$$
 
The main result of this section asserts that the $p$-adic $L$-function $\Lp^f(\htf,\htg,\hth)$ introduced in \S \ref{3.1} can be recast as the  image of the $\Lambda$-adic class $\hkappa_p(\hf,\hg,\hh)$ under a suitable three-variable Perrin-Riou regulator map whose formulation relies on a choice of families of periods which depends on the  test vectors $\htf,\htg,\hth$.  

The recipe we are about to describe depends solely only on the projection of $\hkappa_p(\hf,\hg,\hh)$ to a suitable  
sub-quotient of $\hV^\dag_{\hf\hg\hh}$ which is free of rank one over $\Lambda_{\hf\hg\hh}$, and whose definition requires the following lemma.

\begin{lemma}
\label{lemma:4-step-filtration} 
The Galois representation 
$\mathbb{V}^\dag_{\hf \hg \hh}$ is endowed with a four-step filtration
$$ 0 \subset \mathbb{V}_{\hf \hg \hh}^{++} \subset \mathbb{V}_{\hf \hg \hh}^+ \subset \mathbb{V}_{\hf \hg \hh}^- 
\subset \mathbb{V}^\dag_{\hf \hg \hh}$$ by
$G_{\Q_p}$-stable $\Lambda_{\hf \hg \hh}$-submodules 
of ranks  $0$, $1$, $4$, $7$ and $8$ respectively. 

The group  
 $G_{\Q_p}$ acts 
 on the   successive quotients for this filtration (which are free over $\Lambda_{\hf\hg\hh}$ 
 of ranks $1$, $3$, $3$ and $1$  
 respectively) 
  as a direct sum of one dimensional characters,
 \begin{equation*}
\mathbb{V}_{\hf \hg \hh}^{++} = \heta^{\hf\hg\hh}, \qquad 
\frac{\mathbb{V}_{\hf \hg \hh}^+}{\mathbb{V}_{\hf \hg \hh}^{++}} =  \heta_{\hf}^{ \hg\hh} \oplus 
 \heta_{\hg}^{\hf\hh}  \oplus\heta_{\hh}^{\hf \hg}, \qquad
  \frac{ \mathbb{V}^-_{\hf \hg \hh}}{\mathbb{V}^+_{\hf \hg \hh}} =  \heta^{\hf}_{\hg \hh}  \oplus 
\heta^{\hg}_{\hf \hh}  \oplus \heta^{\hh}_{\hf \hg},   \qquad 
\frac{\mathbb{V}_{\hf \hg \hh}^\dag}{ \mathbb{V}_{\hf \hg \hh}^-} = \heta_{\hf\hg\hh},
\end{equation*}
where
$$
\begin{array}{ll}
 \heta^{\hf\hg\hh} = (\Psi_{\hf} \Psi_{\hg}\Psi_{\hh}  \times  \varepsilon_{\cyc}^{2}(\underline{\varepsilon}_{\hf} \underline{\varepsilon}_{\hg} \underline{\varepsilon}_{\hh})^{1/2},   &      
  \heta_{\hf\hg\hh} =       \Psi_{\hf}   \Psi_{\hg}    \Psi_{\hh} \times \varepsilon_{\cyc}^{-1} (\underline{\varepsilon}_{\hf} \underline{\varepsilon}_{\hg} \underline{\varepsilon}_{\hh})^{-1/2}, \\ \\ 
     \heta_{\hf}^{\hg \hh}  =   \chi_{f}^{-1} \Psi_{\hf}\Psi^{-1}_{\hg} \Psi_{\hh}^{-1} \times  \varepsilon_{\cyc}(\underline{\varepsilon}_{\hf}^{-1} \underline{\varepsilon}_{\hg} \underline{\varepsilon}_{\hh})^{1/2},    & 
            \heta^{\hf}_{\hg \hh}  =   \chi_{f}\Psi_{\hf}^{-1}\Psi_{\hg} \Psi_{\hh}  \times 
  (\underline{\varepsilon}_{\hf}  \underline{\varepsilon}_{\hg}^{-1} \underline{\varepsilon}_{\hh}^{-1})^{1/2}, \\ \\
      \heta_{\hg}^{\hf \hh}  =  \chi_{g}^{-1} \Psi_{\hg}\Psi^{-1}_{\hf}\Psi_{\hh}^{-1} \times \varepsilon_{\cyc}(\underline{\varepsilon}_{\hf} \underline{\varepsilon}_{\hg}^{-1} \underline{\varepsilon}_{\hh})^{1/2},    & 
     \heta^{\hg}_{\hf \hh}  =  \chi_{g}\Psi_{\hf}\Psi_{\hh}\Psi_{\hg}^{-1} \times  (\underline{\varepsilon}_{\hf}^{-1}  \underline{\varepsilon}_{\hg} \underline{\varepsilon}_{\hh}^{-1})^{1/2},  \\ \\
  \heta_{\hh}^{\hf \hg}  =   \chi_{h}^{-1}\Psi_{\hh}\Psi^{-1}_{\hf}\Psi^{-1}_{\hg} \times \varepsilon_{\cyc} (\underline{\varepsilon}_{\hf} \underline{\varepsilon}_{\hg} \underline{\varepsilon}_{\hh}^{-1})^{1/2},   &
               \heta^{\hh}_{\hf \hg}  =   \chi_{h} \Psi_{\hf}\Psi_{\hg}\Psi^{-1}_{\hh} \times (\underline{\varepsilon}_{\hf}^{-1}  \underline{\varepsilon}_{\hg}^{-1} \underline{\varepsilon}_{\hh})^{1/2}. 
         \end{array} $$
 \end{lemma}

 \begin{proof} 
 Let $\hphi$ be a Hida family of tame character $\chi$ as
  in \S \ref{subsec:hida}. 
Let $\psi_{\hphi}$ denote the unramified character of $G_{\Q_p}$ 
sending a Frobenius element $\Fr_p$  to $\mathbf{a}_p(\hphi)$ and recall from \eqref{Wiles-ord} that the restriction of $\mathbb{V}_{\hphi}$ to $G_{\Q_p}$ admits a filtration
\begin{equation*}
0 \, \ra \, \mathbb{V}_{\hphi}^+ \, \ra \, {\mathbb{V}_{\hphi}}  \, \ra \, \mathbb{V}_{\hphi}^- \, \ra \, 0
\end{equation*}
with 
$$
\mathbb{V}_{\hphi}^+ \simeq \Lambda_{\hphi}(\psi_{\hphi}^{-1} \chi \varepsilon_{\cyc}^{-1}\uepsilon ),  \qquad
\mathbb{V}_{\hphi}^- \simeq \Lambda_{\hphi}(\psi_{\hphi}).
$$
Set 
\begin{eqnarray*}
\mathbb{V}_{\hf \hg \hh}^{++} &=& \mathbb{V}_{\hf}^+\otimes \mathbb{V}_{\hg}^+\otimes \mathbb{V}_{\hh}^+(\varepsilon_{\cyc}^{-1} \underline{\varepsilon}_{\hf}^{-1/2}\underline{\varepsilon}_{\hg}^{-1/2} \underline{\varepsilon}_{\hh}^{-1/2}), \\
\mathbb{V}_{\hf \hg \hh}^{+} &= & \big(\mathbb{V}_{\hf}\otimes \mathbb{V}_{\hg}^+\otimes \mathbb{V}_{\hh}^+ \, + \, \mathbb{V}_{\hf}^+\otimes \mathbb{V}_{\hg}\otimes \mathbb{V}_{\hh}^+ \, + \, \mathbb{V}_{\hf}^+\otimes \mathbb{V}_{\hg}^+\otimes \mathbb{V}_{\hh}\big)(\varepsilon_{\cyc}^{-1} \underline{\varepsilon}_{\hf}^{-1/2}\underline{\varepsilon}_{\hg}^{-1/2} \underline{\varepsilon}_{\hh}^{-1/2})
 \\
\mathbb{V}_{\hf \hg \hh}^{-} &= & \big(\mathbb{V}_{\hf}\otimes \mathbb{V}_{\hg}\otimes \mathbb{V}_{\hh}^+ \, + \, \mathbb{V}_{\hf}\otimes \mathbb{V}_{\hg}^+\otimes \mathbb{V}_{\hh} \, + \, \mathbb{V}_{\hf}^+\otimes \mathbb{V}_{\hg}\otimes \mathbb{V}_{\hh}\big)(\varepsilon_{\cyc}^{-1} \underline{\varepsilon}_{\hf}^{-1/2}\underline{\varepsilon}_{\hg}^{-1/2} \underline{\varepsilon}_{\hh}^{-1/2}).
\end{eqnarray*}

It follows from the definitions that these three representations are $\Lambda_{\hf \hg \hh}[G_{\Q_p}]$-submodules of $\mathbb{V}^\dag_{\hf \hg \hh}$
of ranks  $1$, $4$, $7$ as claimed. Moreover, since $\chi_{f} \chi_{g} \chi_{h}=1$, the rest of the lemma follows from \eqref{Wiles-ord}.
\end{proof}

A one-dimensional character $\eta: G_{\Q_p}\lra \C_p^\times$ is said to be 
{\em of Hodge-Tate weight $-j$} if it is equal to a finite order character times
the $j$-th power of the cyclotomic character. The following is an immediate corollary of Lemma 
\ref{lemma:4-step-filtration}.


\begin{cor}
\label{cor:4-step-filtration} Let $(x,y,z)\in \cW_{\hf \hg \hh}^\circ$  be a triple of classical points of weights $(k,\ell,m)$.
The Galois representation $V^\dag_{\hf_x,\hg_y,\hh_z}$   is endowed with a four-step $G_{\Q_p}$-stable
 filtration
  $$ 0\subset V_{\hf_x,\hg_y,\hh_z}^{++} \subset V_{\hf_x,\hg_y,\hh_z}^+ \subset V_{\hf_x,\hg_y,\hh_z}^- 
\subset V^\dag_{\hf_x,\hg_y,\hh_z},$$
and the Hodge-Tate weights of its successive quotients are:

$$ \begin{array}{|c|c|}
 \hline
\mbox{\em Subquotient} &  \mbox{\em Hodge-Tate weights}    \\
\hline 
 V_{\hf_x,\hg_y,\hh_z}^{++} &  \frac{-k-\ell-m}{2}  + 1  \\ \hline
 V_{\hf_x,\hg_y,\hh_z}^+/ V_{\hf_x,\hg_y,\hh_z}^{++}   & \frac{k-\ell-m}{2}, \frac{-k+\ell-m}{2}, \frac{-k-\ell+m}{2}   \\ 
 \hline
 V_{\hf_x,\hg_y,\hh_z}^-/ V_{\hf_x,\hg_y,\hh_z}^{+}      
 & \frac{-k+\ell+m }{2} -1 , \frac{k-\ell+m}{2} -1 , \frac{k+\ell-m}{2} -1 \\ \hline
    V_{\hf_x,\hg_y,\hh_z}/ V_{\hf_x,\hg_y,\hh_z}^{-}  &\frac{k+\ell+m}{2} -2 \\
\hline
    \end{array}
$$

\end{cor}

\begin{cor}\label{corol}
The Hodge-Tate weights of $V_{\hf_x,\hg_y,\hh_z}^+$ are all strictly negative if and only if $(k,\ell,m)$ is balanced.
\end{cor}
 %
%

Let $\mathbb{V}_{\hf}^{\hg\hh}$ and $\mathbb{V}_{\hf}^{\hg\hh}(M)$ be the subquotient of $\hV_{\hf\hg\hh}^\dag$  (resp.\,of $\hV_{\hf\hg\hh}^\dag(M)$) on which
$G_{\Q_p}$ acts via (several copies of) the character 
\begin{equation}\label{eta}
\eta_{\hf}^{\hg\hh} := \Psi_{\hf}^{\hg\hh} \times \Theta_{\hf}^{\hg\hh}
\end{equation}
where
\begin{itemize}

\item $\Psi_{\hf}^{\hg\hh}$ is the unramified character of $G_{\Q_p}$ sending 
$\Fr_p$ to $\chi_f^{-1}(p){\bf a}_p(\hf) {\bf a}_p(\hg)^{-1} {\bf a}_p(\hh)^{-1}$,
and 

\item $\Theta_{\hf}^{\hg\hh}$ is the $\Lambda_{\hf\hg\hh}$-adic cyclotomic character whose
specialization at a point of weight $(k,\ell,m)$ is  $\varepsilon_{\cyc}^{t}$ with
 $t := (-k+\ell+m)/2$. 

\end{itemize}

The classical specializations of $\mathbb{V}_{\hf}^{\hg\hh}$ are
\begin{equation}\label{Vklm}
V_{\hf_x}^{\hg_y \hh_z}:= V_{\hf_x}^-\otimes V_{\hg_y}^+ \otimes V_{\hh_z}^+(\frac{-k-\ell-m+4}{2}) \simeq L_p\big(\chi_f^{-1}\psi_{\hf_x} \psi_{\hg_y}^{-1} \psi_{\hh_z}^{-1}\big)(t),
\end{equation}
where the coefficient field is $L_p=\Q_p(\hf_x,\hg_y,\hh_z)$. Note that $t>0$ when $(x,y,z) \in \cW_{\hf\hg\hh}^{\mathrm{bal}}$, while $t\leq 0$ when $(x,y,z) \in \cW_{\hf\hg\hh}^f$.

Recall now from \S \ref{subsec:Dieu} the Dieudonn\'e module $D(V_{\hf_x}^{\hg_y \hh_z}(Mp))$ associated to \eqref{Vklm}. As it follows from loc.\,cit.,\,every triple 
$$(\eta_1,\omega_2,\omega_3) \in D(V^+_{\hf^*_x}(Mp)) \times D(V^-_{\hg_y^*}(Mp)) \times  D(V^-_{\hh_z^*}(Mp))$$ gives rise to a linear functional $\eta_1\otimes \omega_2 \otimes \omega_3: D(V_{\hf_x}^{\hg_y \hh_z}(Mp)) \lra L_p$.

In order to deal with the $p$-adic variation of these Dieudonn\'e modules, write $\mathbb{V}_{\hf}^{\hg\hh}(M)$ as 
$$
\mathbb{V}_{\hf}^{\hg\hh}(M) = \hU(\Theta_{\hf}^{\hg\hh})
$$
where $ \hU$ is the unramified $\Lambda_{\hf\hg\hh}$-adic representation of $G_{\Q_p}$ given by (several copies of) the character $\Psi_{\hf}^{\hg\hh}$.

As in  \S \ref{subsec:Dieu}, define the $\Lambda$-adic Dieudonn\'e module
$$ \hD(\hU) :=  (\hU \hat\otimes \hat\Z_p^{\nr})^{G_{\Q_p}}.$$

In view of  \eqref{DV}, for every $(x,y,z)\in \cW_{\hf \hg \hh}^{\circ}$ there is a natural specialisation map 
$$
\nu_{x,y,z}: \hD(\hU) \lra D(U_{\hf_x}^{\hg_y \hh_z})
$$ 
where $U_{\hf_x}^{\hg_y \hh_z}:= \hU \otimes_{\Lambda_{\hf\hg\hh}} \Q_p(\hf_x,\hg_y,\hh_z)  \simeq V_{\hf_x}^{\hg_y \hh_z}(Mp)(-t).$

\begin{proposition}\label{var-per}
For any triple of test vectors $$(\htf,\htg,\hth)\in S^{\ordi}_{\Lambda}(M,\chi_{f})[\hf] \times S^{\ordi}_{\Lambda}(M,\chi_{g})[\hg] \times S^{\ordi}_{\Lambda}(M,\chi_{h})[\hh],$$
there exists a homomorphism of $\Lambda_{\hf \hg \hh}$-modules
$$
\langle \quad, \eta_{\htf^*} \otimes \omega_{\htg^*} \otimes \omega_{\hth^*}\rangle: \hD(\hU) \lra \mathcal{Q}_{\hf, \hg \hh}
$$
such that for all $\hlambda \in \hD(\hU)$ and all $(x,y,z) \in \cW_{\hf\hg\hh}^{\circ}$ such that $\hf_x$ is the ordinary stabilization of an eigenform $\hf^\circ_x$ of level $M$:
$$
\nu_{x,y,z}\big( \langle \hlambda, \eta_{\htf^*} \otimes \omega_{\htg^*} \otimes \omega_{\hth^*}\rangle\big) = \frac{1}{\cE_0(\hf^\circ_x) \cE_1(\hf^\circ_x)} \times \langle \nu_{x,y,z}(\hlambda), \eta_{\htf^*_x} \otimes \omega_{\htg^*_y} \otimes \omega_{\hth^*_z}\rangle.
$$
Recall from \eqref{E0E1} that
$$
\cE_0(\hf^\circ_x)  =  1- \chi^{-1}(p)  \beta_{\hf^\circ_x}^2  p^{1-k}, \quad
\cE_1(\hf^\circ_x)  =  1-  \chi(p)  \alpha_{\hf^\circ_x}^{-2} p^{k-2}.
$$
\end{proposition}

\begin{proof} 
Since $\hU$ is isomorphic to the unramified twist of $\hV_{\hf}^-\otimes \hV_{\hg}^+ \otimes \hV_{\hh}^+$, this follows from Proposition \ref{klz-prop} because $\cE_0(\hf^\circ_x)=\cE_0(\hf^{\circ*}_x)$ and $\cE_1(\hf^\circ_x)=\cE_1(\hf^{\circ*}_x)$.
 \end{proof}

It follows from Example \ref{H1f} (a) and (b) that the Bloch-Kato logarithm and dual exponential maps yield isomorphisms
\begin{eqnarray*} 
\log_{\mathrm{BK}}: H^1(\Q_p, V_{\hf_x}^{\hg_y \hh_z}) \stackrel{\sim}{\lra} D(V_{\hf_x}^{\hg_y \hh_z}) 
,  &  & \mbox{ if } t>0,  
\\ 
\exp^*_{\mathrm{BK}}: H^1(\Q_p, V_{\hf_x}^{\hg_y \hh_z}) \stackrel{\sim}{\lra} D(V_{\hf_x}^{\hg_y \hh_z}) 
, & & \mbox{ if } 
t\le 0.
\end{eqnarray*}

Define
\begin{equation}\label{E''}
\cE^{\mathrm{PR}}(x,y,z) = \frac{1-p^{\frac{k-\ell-m}{2}} \alpha^{-1}_{\hf_x} \alpha_{\hg_y} \alpha_{\hh_z}}{1-p^{\frac{\ell+m-k-2}{2}} \alpha_{\hf_x} \alpha^{-1}_{\hg_y} \alpha^{-1}_{\hh_z}} = \frac{1-p^{-c} \beta_{\hf_x} \alpha_{\hg_y} \alpha_{\hh_z}}{1-p^{-c} \alpha_{\hf_x} \beta_{\hg_y} \beta_{\hh_z}}.
\end{equation}

The following is a three-variable version of Perrin-Riou's regulator map constructed in \cite{PR} and \cite{LZ14}.

\begin{proposition}\label{pr}
There is a homomorphism
$$ \cL_{\hf, \hg \hh}: H^1(\Q_p, \mathbb{V}_{\hf}^{\hg\hh}(M)) \lra \hD(\hU)$$
such that for all $\hkappa_p \in H^1(\Q_p, \mathbb{V}_{\hf}^{\hg\hh}(M))$ the image $\cL_{\hf, \hg \hh}(\hkappa_p)$ satisfies the following interpolation properties:
\begin{enumerate}
\item[(i)]
For all balanced points $(x,y,z) \in \cW_{\hf\hg\hh}^{\mathrm{bal}}$,
$$\nu_{x,y,z}\big(\cL_{\hf, \hg \hh}(\hkappa_p)\big) = \frac{(-1)^t}{ t!} \cdot \cE^{\mathrm{PR}}(x,y,z) \cdot \log_{\mathrm{BK}}(\nu_{x,y,z}(\hkappa_p)),$$
\item[(ii)]
For all  points $(x,y,z) \in \cW_{\hf\hg\hh}^f$, 
$$\nu_{x,y,z}\big(\cL_{\hf, \hg \hh}(\hkappa_p)\big) =  (-1)^t \cdot (1-t)! \cdot \cE^{\mathrm{PR}}(x,y,z) \cdot \exp_{\mathrm{BK}}^*(\nu_{x,y,z}(\hkappa_p)). $$
\end{enumerate}
\end{proposition}

\begin{proof}
This follows by standard methods as in \cite[Theorem 8.2.8]{KLZ}, \cite[Appendix B]{LZ14}, \cite[\S 5.1]{DR2}.
\end{proof}

 \begin{proposition}
 \label{cor:crystalline-lambda}
 The class $\hkappa_p(\hf,\hg,\hh)$ belongs to the
 image of $H^1(\Q_p, \mathbb{V}_{\hf \hg \hh}^+(M))$ 
 in $\\ H^1(\Q_p,\mathbb{V}^\dag_{\hf \hg \hh}(M))$ under the map induced from the inclusion
 $\mathbb{V}_{\hf \hg \hh}^+(M) \hookrightarrow
\mathbb{V}^\dag_{\hf \hg \hh}(M)$.
  \end{proposition}

  \begin{proof} Let $(x,y,z)\in \cW_{\hf \hg \hh}^\circ$  be a triple of classical points of weights $(k,\ell,m)$.
By the results proved in \S \ref{sec:cris}, the cohomology class $\kappa_p(\hf_x, \hg_y,\hh_z)$ is proportional to the image under the 
 $p$-adic \'etale Abel-Jacobi map  of the cycles appearing in \eqref{cycleDR1}, that were introduced in \cite[\S 3]{DR1}.  The purity conjecture for the monodromy filtration is known to hold for the variety $\cA^{\kk}\times \cA^{\llll}\times \cA^{\mm}$ by the work of Saito (cf.\,\cite{Sa}, \cite[(3.2)]{Ne2}). By Theorem 3.1 of loc.cit., it follows that the extension $\kappa_p(\hf_x,\hg_y,\hh_z)$ is crystalline. Hence 
 $\kappa_p(\hf_x,\hg_y,\hh_z)$  belongs to $H^1_{\fin}(\Q_p, V^\dag_{\hf_x,\hg_y,\hh_z}(Mp)) \subset H^1(\Q_p, V^\dag_{\hf_x,\hg_y,\hh_z}(Mp))$.

Since $(k,\ell,m)$ is balanced, Corollary \ref{corol} implies that $V^+_{\hf_x,\hg_y,\hh_z}$ is the subrepresentation of $V^\dag_{\hf_x,\hg_y,\hh_z}$ on which the Hodge-Tate weights are all {\em strictly negative}.  As is well-known (cf.\,\cite[Lemma 2, p.\,125]{Fl}, \cite[\S 3.3]{LZsq} for similar results), the finite Bloch-Kato local Selmer group of an ordinary representation can be recast \`{a} la Greenberg \cite{Gr} as 
$$H^1_{\fin}(\Q_p, V^\dag_{\hf_x,\hg_y,\hh_z}) = \ker\left(H^1(\Q_p, V^\dag_{\hf_x,\hg_y,\hh_z}) \lra 
 H^1(I_p,V^\dag_{\hf_x,\hg_y,\hh_z}/ V_{\hf_x,\hg_y,\hh_z}^+)\right),$$ 
 where $I_p$ denotes the inertia group at $p$. 

Since the set of balanced classical points  is dense in $\cW_{\hf \hg\hh}$ for the rigid-analytic topology,
it follows that the $\Lambda$-adic class $\kappa_p(\hf, \hg,\hh)$ belongs to the kernel of the natural map
$$ H^1(\Q_p, \mathbb{V}^\dag_{\hf \hg \hh}(M)) \lra 
H^1(I_p,\mathbb{V}^\dag_{\hf \hg \hh}(M)/\mathbb{V}^+_{\hf \hg \hh}(M)).$$
Since the kernel of the restriction map
$$H^1(\Q_p,\mathbb{V}^\dag_{\hf \hg \hh}(M)/\mathbb{V}^+_{\hf \hg \hh}(M))\lra 
H^1(I_p,\mathbb{V}^\dag_{\hf \hg \hh}(M)/\mathbb{V}^+_{\hf \hg \hh}(M))$$
is trivial by Lemma \ref{lemma:4-step-filtration}, the result  follows.
\end{proof}

Thanks to Lemma \ref{lemma:4-step-filtration}  and Proposition \ref{cor:crystalline-lambda}, we are entitled to define
\begin{eqnarray}\label{kpf}
  \hkappa_p^f(\hf,\hg,\hh)^- &\in & H^1(\Q_p, \mathbb{V}_{\hf}^{\hg\hh}(M))
    \end{eqnarray}
as the  projection of the local class 
$\hkappa_p(\hf,\hg,\hh)$ to $\mathbb{V}_{\hf}^{\hg\hh}(M)$.

\begin{theorem}\label{L=L} For any triple of $\Lambda$-adic test vectors $(\htf,\htg,\hth)$, the following equality holds in the ring $\mathcal{Q}_{\hf, \hg \hh}$:
$$
\langle \cL_{\hf,\hg \hh}(\, \hkappa^f_p(\hf,\hg,\hh)^-\,), \,\eta_{\htf^*} \otimes \omega_{\htg^*} \otimes \omega_{\hth^*}\,\rangle \, = \, \Lp^{f}(\htf,\htg,\hth).
$$

\end{theorem}

\begin{proof} It is enough to prove this equality for a subset of classical points that is dense for the rigid-analytic topology, and we shall do so for all balanced triple of crystalline classical points $(x,y,z)\in \cW^{\mathrm{bal}}_{\hf \hg \hh}$ such that $\hf_x$, $\hg_\ell$ and $\hh_m$ are respectively the ordinary stabilization of an eigenform $f:= \hf^\circ_x$, $g:= \hg^\circ_y$ and $h:= \hh^\circ_z$ of level $M$.

Set $\hkappa_p^- := \hkappa_p^f(\hf,\hg,\hh)^-$ and $\cL= \langle \cL_{\hf,\hg \hh}( \hkappa_p^-), \eta_{\htf^*} \otimes \omega_{\htg^*} \otimes \omega_{\hth^*}\rangle $ for notational simplicity.
Proposition \ref{var-per} asserts that the following identity holds in $L_p$:
\begin{eqnarray*}
\nu_{x,y,z}(\cL) & = & \langle \nu_{x,y,z}(\cL_{\hf,\hg \hh}(\hkappa_p^-)), \eta_{\htf^*_x} \otimes \omega_{\htg^*_y} \otimes \omega_{\hth^*_z}\rangle.
\end{eqnarray*}

Recall also from Proposition \ref{klz-prop} that
$$
\eta_{\htf^*_x} =  \frac{1}{ \cE_1(f)} e \varpi_1^*(\eta_{\testf^*}), \quad  \omega_{\htg^*_y} = \cE_0(g) e \varpi_1^* (\omega_{\testg^*}), \quad \omega_{\hth^*_z} = \cE_0(h) e \varpi_1^*(\omega_{\testh^*})
$$
and
$$
\nu_{x,y,z}(\cL_{\hf,\hg \hh}(\hkappa_p^-)) = \frac{(-1)^t}{ t!} \cdot \cE^{\mathrm{PR}}(x,y,z)  \log_{\mathrm{BK}}(\nu_{x,y,z}(\hkappa_p^-))
$$
by Proposition \ref{pr}.

Recall the class $\kappa(f,g,h)=\kappa(\hf^\circ_x,\hg^\circ_y,\hh^\circ_z)$ introduced in \eqref{cycleDR1} arising from the generalized diagonal cycles of \cite{DR1}. As in \eqref{kpf}, we may define $\kappa_p^f(f,g,h)^- \in H^1(\Q_p, \mathbb{V}_{f}^{g h}(M))$ as the projection to $V_{f}^{g h}(M)$ of the restriction at $p$ of the global class $\kappa(f,g,h)$.

It follows from Theorem \ref{prop:main-crystalline} that
\begin{eqnarray*}
(\varpi_1,\varpi_1,\varpi_1)_*  \nu_{x,y,z}(\hkappa_p^-) &=& \frac{\cE^{\mathrm{bal}}(x,y,z)}{(1-\beta_{f}/\alpha_{f})(1-\beta_{g}/\alpha_{g})(1-\beta_{h}/\alpha_{h})} \times \kappa_p^f(f,g,h)^-
\end{eqnarray*}
where
$$
\cE^{\mathrm{bal}}(x,y,z) =   (1- \alpha_f \beta_g \beta_h p^{-c})(1-\beta_f \alpha_g \beta_h p^{-c})(1- \beta_f \beta_g \alpha_h 
p^{-c})(1- \beta_f \beta_g \beta_h p^{-c}).
$$

The combination of the above identities shows that the value of $\cL$ at the balanced triple $(x,y,z)$ is
$$
\nu_{x,y,z}(\cL) =  \frac{(-1)^t \cdot  \cE^{\mathrm{bal}}(x,y,z) \cE^{\mathrm{PR}}(x,y,z)}{t! \cdot \cE_0(f) \cE_1(f)}  \times \langle   \log_{\mathrm{BK}}( \kappa_p^f(f,g,h)), \eta_{\testf^{*}} \otimes \omega_{\testg^{*}} \otimes \omega_{\testh^{*}}\rangle
$$

Besides, since the syntomic Abel-Jacobi map appearing in \cite{DR1} is the composition of the \'etale Abel-Jacobi map and the Bloch-Kato logarithm, the main theorem of loc.\,cit.\,asserts in the present notations that 
\begin{eqnarray*}
\nu_{x,y,z}\big(\Lp^{f}(\htf,\htg,\hth)\big) &=&  \frac{(-1)^t}{t! }   \frac{\cE^{\mathrm{f}}(x,y,z)}{\cE_0(f) \cE_1(f)} \langle \log_{\mathrm{BK}}(\kappa^f_p(f,g,h)^-),
\eta_{\testf^{*}}  \otimes \omega_{\testg^{*}_y} \otimes \omega_{\testh^{*}} \rangle
\end{eqnarray*}
where 
 $$
\cE^{\mathrm{f}}(x,y,z)  = 
\left(1- \beta_{f}\alpha_{g}\alpha_{h} p^{-c}\right)\left(1-\beta_{f} \alpha_{g}\beta_{h} p^{-c}\right) \left(1- \beta_{f}\beta_{g} \alpha_{h} 
p^{-c}\right)
\left(1- \beta_{f} \beta_{g}\beta_{h} p^{-c}\right).$$

Since
$$
\cE^{\mathrm{f}}(x,y,z) =  \cE^{\mathrm{bal}}(x,y,z)  \times \cE^{\mathrm{PR}}(x,y,z)
$$
and the sign and factorial terms also cancel, we have
$$
\nu_{x,y,z}(\cL)  = \nu_{x,y,z}\big(\Lp^{f}(\htf,\htg,\hth)\big),
$$
as we wanted to show. The theorem follows. 
\end{proof}

\let\oldaddcontentsline\addcontentsline
\renewcommand{\addcontentsline}[3]{}

\let\addcontentsline\oldaddcontentsline

\end{document}